\documentclass[a4paper, 10pt]{amsart}
\usepackage{amsmath,amsthm,amssymb}
\usepackage[T1]{fontenc}
\usepackage[margin=3.1cm]{geometry}
\usepackage[numbered]{bookmark}
\usepackage{changepage}
\usepackage[utf8]{inputenc}
\usepackage{graphics, graphicx, caption}
\usepackage{amsmath, amsfonts, amsthm, amssymb}
\usepackage{mathrsfs, mathtools, mathabx}
\usepackage{verbatim}
\usepackage{enumerate}
\usepackage{xcolor}
\usepackage{cite}
\usepackage{hyperref}
\newcommand{\vep}{\varepsilon}
\numberwithin{equation}{section}

\newtheorem{theorem}{Theorem}[section]
\newtheorem{lemma}[theorem]{Lemma}
\newtheorem{proposition}[theorem]{Proposition}

\newtheorem{conjecture}[theorem]{Conjecture}
\theoremstyle{remark}
\newtheorem{remark}[theorem]{Remark}
\newtheorem{definition}[theorem]{Definition}
\newtheorem{example}{Example}

\newcommand{\Z}{\mathbb{Z}}
\newcommand{\N}{\mathbb{N}}

\newcommand{\R}{\mathbb{R}}
\newcommand{\C}{\mathbb{C}}
\newcommand{\T}{\mathbb{T}}

\newcommand{\A}{\mathcal{A}}
\newcommand{\I}{\mathcal{I}}

\date{}
\begin{document}
\sloppy
\title[Anti-symmetric skew products with singularities]
{On the ergodicity of anti-symmetric skew products with singularities and its applications}

\author[P.\ Berk]{Przemys\l aw Berk}
\address{Faculty of Mathematics and Computer Science, Nicolaus
Copernicus University, ul. Chopina 12/18, 87-100 Toru\'n, Poland}
\email{zimowy@mat.umk.pl}

\author[K.\ Fr\k{a}czek]{Krzysztof Fr\k{a}czek}
\address{Faculty of Mathematics and Computer Science, Nicolaus
Copernicus University, ul. Chopina 12/18, 87-100 Toru\'n, Poland}
\email{fraczek@mat.umk.pl}

\author[F.\ Trujillo]{Frank Trujillo}
\address{Centre de Recerca Matem\`atica, Campus de Bellaterra,
Edifici C, 08193 Bellaterra, Barcelona, Spain}
\email{ftrujillo@crm.cat}

\date{}

\makeatletter
\@namedef{subjclassname@2020}{\textup{2020} Mathematics Subject Classification}
\makeatother

\thanks{}
\maketitle

\begin{abstract}
We introduce a novel method for proving the ergodicity of skew products of interval exchange transformations (IETs) with piecewise smooth cocycles having singularities at the ends of exchanged intervals. This approach is inspired by Borel-Cantelli-type arguments in \cite{Fa-Le}. The key innovation of our method lies in its applicability to singularities beyond the logarithmic type, whereas previous techniques were restricted to logarithmic singularities. Our approach is particularly effective for proving the ergodicity of skew products for symmetric IETs and antisymmetric cocycles. Moreover, its most significant advantage is the ability to study the equidistribution of error terms in the spectral decomposition of Birkhoff integrals for locally Hamiltonian flows on compact surfaces, applicable not only when all saddles are perfect (harmonic) but also in the case of some non-perfect saddles.
\end{abstract}

\section{Introduction}
Let $T=T_{\pi,\lambda}:[0,1)\to[0,1)$ be an interval exchange transformation (IET) given by an irreducible permutation $\pi=(\pi_0,\pi_1)$ ($\pi_0,\pi_1:\mathcal{A}\to\{1,\ldots,d\}$ are bijections describing the position of intervals, labeled by elements of $\mathcal{A}$, before and after the translation) and the vector $\lambda=(\lambda_\alpha)_{\alpha\in\mathcal{A}}$ collecting the lengths of exchanged intervals. Denote by $I_\alpha=[l_\alpha,r_\alpha)$, $\alpha\in\mathcal{A}$ the intervals exchanged by $T$. We mainly deal with symmetric IETs, i.e.\ $\pi_0(\alpha)+\pi_1(\alpha)=d+1$ for any $\alpha\in\mathcal{A}$. If $\I:[0,1]\to[0,1]$ is the reflection across $1/2$, i.e.\ $\I x=1-x$, then $\I\circ T=T^{-1}\circ \I$ and this map acts on every interval $I_\alpha$ as the reflection across the center of $I_\alpha$ denoted by $m_\alpha$.

The main objective of the article is to develop novel methods to prove the ergodicity of skew products $T_f:[0,1)\times \R\to [0,1)\times \R$ of the form $T_f(x,r)=(Tx,r+f(x))$, where $f:[0,1)\to\R$ is a $C^1$-map on the interior of all exchanged intervals and having singularities at their ends. Suppose that
\begin{align}\label{prop:theta tau0}
\begin{aligned}
\text{$\theta:[x_0,+\infty)\to\R_{>0}$ is an increasing $C^1$-map such that $\int_{x_0}^{+\infty}\frac{dx}{x\theta(x)}=+\infty$,} \\
\text{ and the map $\tau:(0,x_0^{-1}]\to\R_{>0}$ given by $\tau(s)=\frac{s^2}{\theta'(1/s)}$ is increasing.}\quad
\end{aligned}
\end{align}
Notice that we can always extend the functions $\theta$ and $\tau$ so that $x_0 = 1$ while preserving the properties above.

In this work, we consider maps with singularities that behave like $s\mapsto \theta(1/s)$. More precisely, we deal with the space $\Upsilon_\theta(\bigsqcup_{\alpha\in \mathcal{A}}
I_{\alpha})$ of functions $f:[0,1)\to\R$ which are $C^1$ on the interior of all exchanged intervals $I_{\alpha}$, $\alpha\in \mathcal{A}$ and
\begin{equation}\label{def:Zf0}
Z_\theta(f):=\max_{\alpha\in\mathcal{A}}\Big\{\sup_{x\in(l_\alpha,m_\alpha]}|f'(x)\tau(x-l_\alpha)|,\sup_{x\in[m_\alpha,r_\alpha)}|f'(x)\tau(r_\alpha-x)|\Big\}<+\infty.
\end{equation}
For any $f\in\Upsilon_\theta(\bigsqcup_{\alpha\in \mathcal{A}}
I_{\alpha})$ let
\begin{equation}\label{def:zf0}
z_\theta(f):=\max_{\alpha\in\mathcal{A}}\Big\{\inf_{x\in(l_{\alpha},m_{\alpha}]}|f'(x)\tau(x-l_{\alpha})|,
\inf_{x\in[m_{\alpha},r_{\alpha})}|f'(x)\tau(r_{\alpha}-x)|\Big\}.
\end{equation}
Roughly speaking, the positivity of $z_\theta(f)$ expresses non-triviality of at least one singularity of the type $s\mapsto \theta(1/s)$.

In this article, we focus on the ergodic properties of skew products $T_f$ when the function $f$ is additionally \emph{anti-symmetric}, which means that $f\circ T^{-1}\circ \I=-f$, so is anti-symmetric with respect to the central reflection of any interval $I_\alpha$, $\alpha\in\mathcal{A}$. Moreover, we say that $f$ is \emph{piecewise monotonic} if $f$ it is increasing when restricted to the interior of each interval $I_\alpha$, $\alpha\in\mathcal{A}$ or is always decreasing.

The main basic result of the paper, based on ideas developed in \cite{Fa-Le} for rotations, is the following.
\begin{theorem}\label{thm:anti}
Let $\theta:[x_0,+\infty)\to\R_{>0}$, with $x_0>0$, be a slowly varying increasing $C^1$-map such that $\int_{x_0}^{+\infty}\frac{dx}{x\theta(x)}=+\infty$ and the map
$x\mapsto x^2\theta'(x)$ is increasing. Then for a.e.\ symmetric IET $T$ if
\begin{enumerate}
\item $f\in \Upsilon_\theta(\bigsqcup_{\alpha\in \mathcal{A}}I_{\alpha})$ with $z_\theta(f)>0$;
\item $f$ is anti-symmetric;
\item $f$ is piecewise monotonic,
\end{enumerate}
then the skew product $T_f$ is ergodic.
\end{theorem}

This theorem and its versions are directly applicable to studying of the behavior of error terms in the spectral decomposition of Birkhoff integrals for locally Hamiltonian flows on compact surfaces when all saddles are perfect. The aforementioned spectral decomposition has recently been studied in \cite{Fr-Ul2} (for non-degenerate saddles) and in \cite{Fr-Ki1} (more generally, when all saddles are perfect). We present all the details in Section~\ref{sec:introHam}. One of the main results of \cite{Fr-Ul2} is to prove the equidistribution of the error term when almost every flow has no saddle loops, which is a consequence of the ergodicity of a certain skew extension of the flow. In turn, the ergodicity of the extension follows from the ergodicity of a certain skew product $T_f$ for $f$ having logarithmic singularities of symmetric type.

One of the goals of this article is to show that the error term is equidistributed also when the locally Hamiltonian flow has saddle loops.
In this case, the methods developed so far fail because the symmetry condition is broken. To overcome this problem, we use a version of Theorem~\ref{thm:anti} for $\theta=\log$ to prove the equidistribution of the error term in the case of locally Hamiltonian flows (with perfect saddles) of hyper-elliptic type, see Theorem~\ref{mainthm:onesaddle}, which constitutes another main result of the paper.

However, Theorem~\ref{thm:anti} also has applications beyond the use of the logarithmic function. In Appendix~\ref{sec:imper}, we construct a family of locally Hamiltonian flows with a new type of degenerate and imperfect saddles whose antisymmetric skew product extensions are ergodic. This is an entirely new class of ergodic skew products unknown even in the context of extensions of rotations.

\section{Applications to locally Hamiltonian flows on compact surfaces}\label{sec:introHam}

Let $M$ be a smooth, compact, connected, orientable surface of genus $g\geq 1$.
We focus on smooth flows $\psi_\R = (\psi_t)_{t\in\R}$ on $M$ preserving a \emph{smooth} area form $\omega$, i.e., such that for any (orientable)
local coordinates $(x,y)$, we have $\omega=V(x,y)dx\wedge dy$ with $V$ positive and smooth.
Then for (orientable) local coordinates $(x,y)$ such that $\omega=V(x,y)dx\wedge dy$, the flow $\psi_\R$ is (locally) a solution to the Hamiltonian equation
\[
\frac{dx}{dt} = \frac{\frac{\partial H}{\partial y}(x,y)}{V(x,y)},\qquad
\frac{dy}{dt} = -\frac{\frac{\partial H}{\partial x}(x,y)}{V(x,y)},
\]
for a smooth real-valued locally defined function $H$, or equivalently $\frac{dz}{dt}=-2\iota\frac{\frac{\partial H}{\partial \overline{z}}(z,\overline{z})}{V(z,\overline{z})}$ in complex variables. The flows $\psi_\R$ are usually called \emph{locally Hamiltonian flows} or \emph{multivalued Hamiltonian flows}.

For any smooth observable $f: M\to\R$, we are interested in understanding the asymptotic behavior of the so-called \emph{Birkhoff integrals}
\[\int_0^Tf(\psi_tx)dt\quad\text{as}\quad T\to+\infty.\]
We always assume that all fixed points of the flow $\psi_\R$ are isolated, so the set $\mathrm{Fix}(\psi_\R)$ of fixed points of $\psi_\R$ is finite and, if $g \geq 2$, then it is non-empty. As $\psi_\R$ is area-preserving, fixed points are either centers, simple saddles, or multi-saddles (saddles with $2k$ prongs with $k \geq 2$). In this article, we will mainly consider \emph{perfect} (also known as \emph{harmonic}) saddles. A fixed point $\sigma\in \mathrm{Fix}(\psi_\R)$ is a (perfect) saddle of multiplicity $m=m_\sigma\geq 2$ if there exists a chart $(x,y)$ (called \emph{a singular chart}) in a neighborhood $U_\sigma$ of $\sigma$ such that $d\mu=V(x,y)dx\wedge dy$ and $H(x,y)=\Im (x+\iota y)^m$ ($(0,0)$ are coordinates of $\sigma$). Then the corresponding local Hamiltonian equation in $U_\sigma$ is of the form
$\frac{dz}{dt}=\frac{m\overline{z}^{m-1}}{V(z,\overline{z})}$. We denote the \emph{set of perfect saddles} of $\psi_\R$ by $\mathrm{PSd}(\psi_\R)$.

Note that perfect saddles are, in a sense, a generalization of non-degenerate saddles. Indeed, suppose that $(0,0)$ is a non-degenerate saddle for the local Hamiltonian equation $\frac{dz}{dt}=-2\iota\frac{\frac{\partial H}{\partial \overline{z}}(z,\overline{z})}{V(z,\overline{z})}$, i.e.\ the Hessian of $H$ at $(0,0)$ is non-zero. Then, by Morse Lemma, after a smooth change of coordinates, we have $H(z,\bar{z})=\Im z^2$, so the saddle is perfect of multiplicity $2$.

However, in the current article, we also go beyond the case of perfect saddles. The techniques we propose are also proving effective for flows with certain imperfect saddles, see Appendix~\ref{sec:imper}, which seems to be a major breakthrough in understanding the behavior of Birkhoff integrals for locally Hamiltonian flows.

A \emph{saddle connection} of $\psi_\R$ is an orbit of $\psi_\R$ running from a saddle to a saddle. A \emph{saddle loop} is a saddle connection
joining the same saddle.
We will deal only with flows for which all their saddle connections are loops.
We denote by $\mathrm{SL}$ the set of all saddle loops of the flow.
Recall that if every fixed point of $\psi_\R$ is isolated, then $M$ splits into a finite number of $\psi_\R$-invariant surfaces (with boundary) so that any such surface is either a \emph{minimal component} (i.e., every orbit, except fixed points and saddle loops, is dense in the component) or is a periodic component (filled by periodic orbits, fixed points and saddle loops). The boundary of each component consists of saddle loops and fixed points. Since the dynamics of flows on periodic components is not interesting, we only focus on the study of minimal components.

\subsection{Historical overview}
The first important step towards a full understanding of the asymptotic behavior of Birkhoff integrals was the phenomenon of deviation spectrum and its relation with Lyapunov exponents observed by Zorich \cite{Zo:how} while studying deviations of Birkhoff sums for piecewise constant observables for almost all interval exchange translations.
Inspired by this, Kontsevich \cite{Ko} and Zorich \cite{Ko-Zo} formulated their famous conjecture: for almost every $\psi_\mathbb{R}$ with non-degenerate fixed points
there exist values $1=\nu_1>\nu_2>\ldots>\nu_g>\nu_{g+1}=0$ so that
for every smooth map $f:M\to\mathbb{R}$ there exists $1\leq i\leq g+1$ such that
\[\limsup_{T\to+\infty}\frac{\log\left|\int_0^Tf(\psi_t(x))\,dt\right|}{\log T}=\nu_i, \quad \text{ for almost every }x\in M.\]
They related the values $\nu_i$, for $1\leq i\leq g$, with the positive Lyapunov exponents of the Kontsevich-Zorich cocycle.
Fundamental steps in the verification of the conjecture have been done in the seminal paper by Forni \cite{Fo2}, introducing Forni's distributions $D_i(f)$, for $1\leq i\leq g$, and later developed by Bufetov \cite{Bu}, constructing Bufetov's cocycles $u_i(t,x),$ for $1\leq i\leq g$, for a family of observables $f$. More precisely,
for almost every locally Hamiltonian flow $\psi_\mathbb{R}$ without saddle connections ($M$ is the only one minimal component) and for every observable $f:M\to\mathbb{R}$ from a weighted Sobolev space, we have
\begin{equation}\label{eq:sumbasic}
\int_0^T f(\psi_t(x))dt = \sum_{i=1}^g {D_i}(f)u_i(T,x) + err(f,T,x),
\end{equation}
where
\begin{gather}\label{eq:growthbasic}
\limsup_{T\to+\infty}\frac{\log\big|u_i(T,x)\big|}{\log T}=\nu_i,\ \quad
\lim_{T\to+\infty}\frac{\log|err(f,T,x)|}{\log T}=0, \quad \text{ for a.e. }x\in M.
\end{gather}
However, the final results on the deviation spectrum for Birkhoff integrals in full generality have only recently been obtained in \cite{Fr-Ul2} for non-degenerate saddles and in \cite{Fr-Ki1} for arbitrary perfect saddles. Both articles go beyond the setting considered so far, when the flow is minimal on $M$ and the observables $f$ vanishes at fixed points, and exploit and develop the techniques introduced by Marmi-Moussa-Yoccoz in \cite{Ma-Mo-Yo,Ma-Yo}.
In particular, in \cite{Fr-Ul2} the authors proved \eqref{eq:sumbasic} with \eqref{eq:growthbasic} for a.e.\ non-degenerate locally Hamiltonian flow $\psi_\mathbb{R}$ restricted to any of its minimal components and any smooth observable $f:M\to\mathbb{R}$, fully solving the Kontsevich-Zorich conjecture and giving a somewhat deeper analysis of the asymptotics of the error term is provided.

The transition to the setting in which perfect saddles of any multiplicity appear gives rise to new invariant distributions and new terms in the deviational spectrum, which we now intend to describe.

\subsection{Deviation spectrum in full generality}
Suppose $\psi_\R$ is a locally Hamiltonian flow on $M$ with isolated fixed points such that all saddles are perfect and all saddle connections are loops. Let $M'$ be its minimal component. For any saddle $\sigma\in \mathrm{PSd}(\psi_\R)$, any $0\leq l< 2m_\sigma$ and $0\leq k\leq m_\sigma-2$, let us consider after \cite{Fr-Ki2} the distribution $\mathfrak{C}^k_{\sigma,l}:C^k(M)\to\C$ given by
\begin{equation}\label{def:gothc}
\mathfrak{C}^k_{\sigma,l}(f):=\sum_{\substack{0\leq i\leq k}}
\theta_\sigma^{l(2i-k)}\binom{k}{i}\mathfrak{B}\big(\tfrac{(m_\sigma-1)-i}{m_\sigma},\tfrac{(m_\sigma-1)-k+i}{m_\sigma}\big)\frac{\partial^{k}(f\cdot V)}{\partial z^i\partial\overline{z}^{k-i}}(0,0),
\end{equation}
where $\theta_\sigma$ is the principal $2m_\sigma$-th root of unity and
$\mathfrak{B}(x,y)=\frac{\pi e^{\iota\frac{\pi}{2}(y-x)}}{2^{x+y-2}}\frac{\Gamma(x+y-1)}{\Gamma(x)\Gamma(y)}$, for any $x,y>0$, where we adopt the convention $\Gamma(0)=1$. By \cite{Fr-Ki2}, the distributions $\mathfrak{C}^k_{\sigma,l}$ are $\psi_\R$-invariant. Moreover, by definition,
\begin{gather*}
\mathfrak{C}^k_{\sigma,l+m_\sigma}=(-1)^k\mathfrak{C}^k_{\sigma,l}, \text{ for }0\leq l< m_\sigma, \quad \text{ and } \quad
\sum_{0\leq l<m_\sigma}\theta_\sigma^{(k-2j)l}\mathfrak{C}^k_{\sigma,l}=0, \text{ if }k< j<m_\sigma.
\end{gather*}
\begin{figure}[h!]
 \includegraphics[width=0.3\textwidth]{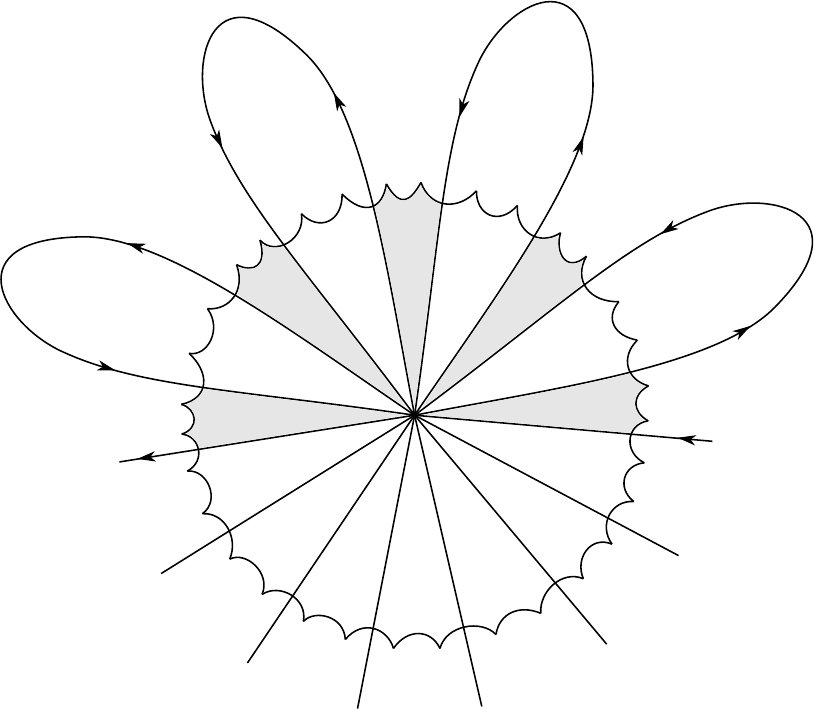}
 \caption{A chain of adjacent saddle loops.} \label{fig:class}
\end{figure}

Recall that the boundary of the minimal component $M'$ consists of saddle loops of the flow $\psi_\R$. Suppose that $\sigma\in M'\cap\mathrm{PSd}(\psi_\R)$ is a saddle on the boundary of $M'$, that is, $\sigma$ emanates at least one saddle loop. The neighborhood $U_\sigma$ (in local singular coordinates associated with $\sigma$) splits into $2m_\sigma$ invariant angular sectors
\[U_{\sigma,l}:=\{z\in U_\sigma:\operatorname{Arg} z\in[l\pi/m_\sigma,(l+1)\pi/m_\sigma]\}, \quad \text{ for }0\leq l<2m_\sigma.\]
Denote
\[\mathcal{L}^\sigma:=\{0\leq l<2m_\sigma: U_{\sigma,l}\subseteq M'\}\]
 and let us consider an equivalence relation $\sim$ on $\mathcal{L}^\sigma$ as follows:
 $l\sim l'$ if the angular sectors $U_{\sigma,l}$ and $U_{\sigma,l'}$ are connected through a chain of adjacent saddle loops emanating from the saddle $\sigma$. Two saddle loops are adjacent when they touch the same angular sector. For every equivalence class $\ell\in \mathcal{L}^\sigma_\sim:=\mathcal{L}^\sigma/\sim$ (an example of such a class is shown in Figure~\ref{fig:class}), let
\[\mathfrak{C}^k_{\sigma,\ell}(f):=\sum_{l\in\ell}\mathfrak{C}^k_{\sigma,l}(f).\]
The following result is a more subtle version of Theorem~1.1 from \cite{Fr-Ki1}, obtained by applying techniques from \cite{Fr-Ki1} together with Theorem~5.6 of \cite{Fr-Ki2}.

\begin{theorem}\label{thm:devspect}
For almost every locally Hamiltonian flow $\psi_\mathbb{R}$ on a compact surface $M$, the following holds. Let $M'$ be a minimal component of $\psi_\mathbb{R}$ of genus $g \geq 1$ and denote $m:=\max\{m_\sigma:\sigma\in \mathrm{Fix}(\psi_\mathbb{R})\cap M'\}$. There exist Lyapunov exponents
\[
1:= \nu_1 > \nu_2 > \cdots > \nu_g >0,
\]
invariant distributions ${D_i}:C^m(M)\to\mathbb{R}$, $1\leq i\leq g$, smooth cocycles $u_i(T,x) : \mathbb{R} \times M \rightarrow \mathbb{R}$, $1\leq i\leq g$, and smooth cocycles $c^k_{\sigma,\ell}(T,x): \mathbb{R} \times M \rightarrow \mathbb{R}$, for all $\sigma\in \mathrm{Fix}(\psi_\mathbb{R})\cap M'$, $0\leq k\leq m_\sigma-2$ and $\ell\in \mathcal{L}^\sigma_\sim$, such that for every $f \in C^{m}(M)$,
\begin{align}\label{eqn;dev-complete}
\begin{split}
\int_0^T f(\psi_t(x))dt &= \sum_{\sigma \in \mathrm{Fix}(\psi_\mathbb{R})\cap M'}\sum_{0\leq k < m_\sigma-2}\sum_{\ell\in \mathcal{L}^\sigma_\sim}
\mathfrak{C}^k_{\sigma,\ell}(f) c^k_{\sigma,\ell}(T,x) \\
&\quad + \sum_{i=1}^g {D_i}(f)u_i(T,x) + err(f,T,x),%
\end{split}
\end{align}
with the following conditions on the growth of main terms satisfied for a.e.\ $x \in M'$:
\begin{gather}
\limsup_{T \rightarrow \infty} \frac{\log{|c^k_{\sigma,\ell}(T,x)|}}{\log T} =\frac{(m_\sigma-2)-k}{m_\sigma}\quad \text{for } \sigma\in \mathrm{Fix}(\psi_\mathbb{R})\cap M',\ 0\leq k< m_\sigma-2,\ \ell\in \mathcal{L}^\sigma_\sim; \label{res;dev1}\\
\limsup_{T \rightarrow \infty} \frac{\log{| u_i(T,x)|}}{\log T} = \nu_i \quad \text{for }\ 1\leq i\leq g; \label{res;dev3}\\
\limsup_{T \rightarrow \infty} \frac{\log|err(f,T,x)|}{\log T} \leq 0 . \label{eqn;dev-remain}
\end{gather}
\end{theorem}

An outline of the proof for this result can be found in Section \ref{sc:PS}.

Suppose that $\psi_\R$ is minimal and has only non-degenerate saddles, i.e., $m=2$. Then the first part of the deviation spectrum in \eqref{eqn;dev-complete} disappears, $\mathcal{L}^\sigma_\sim=\{0,1,2,3\}$ and $\mathfrak{C}^0_{\sigma,l}(f)=2f(\sigma)$.
In this case, \cite{Fr-Ul2} provides a more in-depth analysis of the error term, which has subpolynomial growth. More precisely, if $f(\sigma)\neq 0$ for some $\sigma \in \mathrm{Fix}(\psi_\mathbb{R})$ then for a.e.\ $x\in M$ the error term $err(f,T,x)$ is equidistributed on $\R$ as $T\to+\infty$, that is,
\[\lim_{T\to+\infty}\frac{|\{t\in[0,T]:err(f,T,x)\in I\}|}{|\{t\in[0,T]:err(f,T,x)\in J\}|}=\frac{|I|}{|J|},\]
for any pair of bounded intervals $I, J\subseteq \R$. On the other hand, if $f(\sigma)= 0$ for all $\sigma \in \mathrm{Fix}(\psi_\mathbb{R})$ then the error term $err(f,T,x)$ is uniformly bounded.

In fact, in the first scenario, the equidistribution of the error term follows from the equidistribution of a.e. Birkhoff integral $\int_0^Tf(\psi_t(x))\,dt$ whenever $f:M\to\R$ is of class $C^2$, $D_i(f)=0$ for $1\leq i \leq g$, and $f(\sigma)\neq 0$, for some fixed point $\sigma$. In turn, for such an observable $f$, the equidistribution of the Birkhoff integrals follows directly, applying Hopf's ergodic theorem, from the ergodicity of the skew product flow $\psi^f_\R=(\psi^f_t)_{t\in\R}$ acting on $M\times\R$ by
\[\psi^f_t(x,r)=\left(\psi_t(x),r+\int_0^tf(\psi_s (x))\,ds\right).\]

Recall that the ergodicity of skew product flows as above can be deduced from the ergodicity of certain discrete skew products over IETs. Indeed, if $I\subseteq M$ is a transversal curve identified with an interval, then the first return map of the flow $\psi_\R$ to $I$ is an interval exchange transformation $T:I\to I$ for the so-called \emph{standard parameterizations}. Moreover, the first return map of the flow $\psi^f_\R$ to $I\times \R$ is the skew product map $T_{\varphi_f}:I\times\R\to I\times\R$ given by
\begin{equation}\label{def:phif}
T_{\varphi_f}(x,r)=(T(x),r+\varphi_f(x)), \quad \text{ with }\varphi_f(x)=\int_0^{\tau(x)}f(\psi_t (x))\,dt,
\end{equation}
where $\tau:I\to\R_{>0}\cup\{+\infty\}$ is the first return time map of the flow $\psi_\R$ to $I$. Thus, the ergodicity of the flow $\psi^f_\R$ is equivalent to the ergodicity of the skew product map $T_{\varphi_f}$.

\begin{remark}
Recall that a subset $A$ of locally Hamiltonian flows has full Lebesgue measure (in the sense of Katok's fundamental class, see \cite{Ka0}) if and only if a full measure set of IETs appears in the
base of special flows representations of flows in $A$. For a more detailed explanation of the space of locally Hamiltonian flows, the partition of the surface into invariant components, and the measure class, we refer the reader to \cite{Rav}.
\end{remark}

Let us mention that the proof of ergodicity in \cite{Fr-Ul2} relies on the fact that the map (cocycle) $\varphi_f:I\to\R$ has logarithmic singularities of \emph{symmetric type} at the ends of the intervals exchanged by $T$. A similar argument has been recently used in \cite{Be-Tr-Ul} to conclude the ergodicity of {a certain symmetric class of cocycles
with symmetric logarithmic singularities. This symmetry condition is a consequence of the absence of saddle loops. Notice that if there are saddle loops, then we are limited to studying the flow on the minimal components, and the symmetry condition is naturally broken. We, however, expect that the symmetry condition for logarithmic singularities is not relevant to the ergodicity of the skew product and, consequently, to the equidistribution of the error term in the deviation spectrum.
Our working conjecture is as follows:
\begin{conjecture}\label{conj}
Let $\psi_\mathbb{R}$ be a locally Hamiltonian flow on a compact surface $M$ for which Theorem~\ref{thm:devspect} holds. Let $M'$ be one of its minimal components of genus $g \geq 1$ and denote $m:=\max\{m_\sigma:\sigma\in \mathrm{Fix}(\psi_\mathbb{R})\cap M'\}$. For every $f\in C^m(M)$ such that
\begin{itemize}
\item $D_i(f)=0$ for all $1\leq i\leq g$,
\item $\mathfrak{C}^k_{\sigma,\ell}(f)=0$ for all $\sigma\in\mathrm{Fix}(\psi_\mathbb{R})\cap M'$, $0\leq k<m_\sigma-2$ and $\ell\in\mathcal{L}^\sigma_\sim$, and
\item $\mathfrak{C}^{m_\sigma-2}_{\sigma,\ell}(f)\neq 0$ for some $\sigma\in\mathrm{Fix}(\psi_\mathbb{R})\cap M'$ and $\ell\in\mathcal{L}^\sigma_\sim$,
\end{itemize}
the skew product flow $\psi^f_\R$ is ergodic.
\end{conjecture}
Note that this yields the equidistribution of the error term in \eqref{eqn;dev-complete} whenever $\mathfrak{C}^{m_\sigma-2}_{\sigma,\ell}(f)\neq 0$ for some $\sigma\in\mathrm{Fix}(\psi_\mathbb{R})\cap M'$ and $\ell\in\mathcal{L}^\sigma_\sim$.

\subsection{Main results and technical novelties}
In the current article, we positively verify Conjecture \ref{conj} in two important special cases. First, see Theorem~\ref{mainthm:mini}, we consider the case when the flow $\psi_\R$ has no saddle loops, that is, the flow is minimal over the entire surface $M$. Then, in view of \cite[Theorem~5.6]{Fr-Ki1} applied to $r=0$, the map $\varphi_f$ has logarithmic singularities. By the absence of saddle connections, $\varphi_f$ satisfies the symmetry condition. Then, we can use directly the ergodicity criterion developed in \cite[Theorem~8.1]{Fr-Ul2}. In conclusion, this case does not require the development of new tools.

Second, we consider the case when the flow $\psi_\R$ has exactly
one saddle point, and we restrict ourselves to a minimal component $M'$. We additionally assume that $\psi_\R$ on $M'$ is of hyper-elliptic type, that is, there is a transversal curve $I\subseteq M'$ such that the IET $T:I\to I$ given by the first return map of $\psi_\R$ to $I$ is symmetric. In this context, the main result is the following.

\begin{theorem}\label{mainthm:onesaddle}
For almost every locally Hamiltonian flow $\psi_\mathbb{R}$ on a compact surface $M$ having exactly one perfect saddle, the following holds. Let $M'$ be a minimal component of $\psi_\mathbb{R}$ of genus $g \geq 1$ and assume that $\psi_\mathbb{R}$ on $M'$ is of hyper-elliptic type. Then, for every $f\in C^{m_\sigma}(M)$ such that
\begin{itemize}
\item $D_i(f)=0$ for all $1\leq i\leq g$,
\item $\mathfrak{C}^k_{\sigma,\ell}(f)=0$ for all $0\leq k<m_\sigma-2$ and $\ell\in\mathcal{L}^\sigma_\sim$, and
\item $\mathfrak{C}^{m_\sigma-2}_{\sigma,\ell}(f)\neq 0$ for some $\ell\in\mathcal{L}^\sigma_\sim$,
\end{itemize}
the skew product flow $\psi^f_\R$ is ergodic.

In particular, for every $f\in C^{m_\sigma}(M)$ such that $\mathfrak{C}^{m_\sigma-2}_{\sigma,\ell}(f)\neq 0$ for some $\ell\in\mathcal{L}^\sigma_\sim$, we have
\[
\int_0^T f(\psi_t(x))dt = \sum_{0\leq k < m_\sigma-2}\sum_{\ell\in \mathcal{L}^\sigma_\sim}
\mathfrak{C}^k_{\sigma,\ell}(f) c^k_{\sigma,\ell}(T,x)
 + \sum_{i=1}^g {D_i}(f)u_i(T,x) + err(f,T,x)
\]
with the cocycles $u_i$ and $c^k_{\sigma,\ell}$ that grow polynomially according to rules given by \eqref{res;dev1} and \eqref{res;dev3}, and the error term $err(f,T,x)$ that oscillates sub-polynomially and which is equidistributed on $\R$ as $T\to+ \infty$.
\end{theorem}

Arguments coming from \cite{Fr-Ki2} show that $\varphi_f$ has logarithmic singularities, but due to the existence of saddle loops, the symmetry condition is broken, and the techniques developed in \cite{Fr-Ul2} do not work. In this case, we take advantage of the fact that the flow has only one saddle point and is hyperelliptic, so the corresponding interval transformation is symmetric. This, in turn, provides an opportunity to decompose $\varphi_f$ into a symmetric and anti-symmetric part. In fact, only its anti-symmetric part is responsible for the ergodicity of the skew product map $T_{\varphi_f}$, and finally also for the ergodicity of the skew product flow $\psi^f_\R$.

The main novelty of this article is the development of completely new techniques for proving the ergodicity of skew products based on the anti-symmetry of cocycles having singularities. The developed tools are mainly inspired by ideas from \cite{Fa-Le}, in which Fayad and Lema\'nczyk prove the ergodicity of skew products of a.a.\ rotations with cocycles having one logarithmic singularity using Borel-Cantelli-type results, and by ideas from \cite{Be-Tr}, in which anti-symmetry was used to prove the ergodicity of skew products of interval exchange transformations with piecewise constant cocycles. We should emphasize that the power of the developed techniques is also revealed by their applicability to anti-symmetric skew products with non-logarithmic singularities. In Appendix~\ref{sec:imper}, we present a family of locally Hamiltonian flows with a new type of degenerate and imperfect saddles whose antisymmetric skew product extensions are ergodic. This is an entirely new class of ergodic skew products unknown even in the context of extensions of rotations.

\section{Interval exchange transformations and translation surfaces}
 Let $\mathcal{A}$ be a $d$-element alphabet. An interval exchange transformation $T = (\pi, \lambda)$ is a piecewise isometry of the interval $I=[0,|I|)$ determined by a pair $\pi=(\pi_0,\pi_1)$ (to which we refer to as a \emph{permutation}) of bijections $\pi_\vep:\mathcal{A}\to\{1,\ldots,d\}$, for $\vep=0,1$, and a vector $\lambda=(\lambda_\alpha)_{\alpha\in\mathcal{A}}\in \R_{>0}^{\mathcal{A}}$. For any $\lambda=(\lambda_\alpha)_{\alpha\in\mathcal{A}}\in \R_{>0}^{\mathcal{A}}$, let
\[|\lambda|=\sum_{\alpha\in\mathcal{A}}\lambda_\alpha,\quad I=\left[0,|\lambda|\right),\]
and define
\[I_{\alpha}=[l_\alpha,r_\alpha),\qquad \text{ where } \qquad l_\alpha=\sum_{\pi_0(\beta)<\pi_0(\alpha)}\lambda_\beta,\;\;\;r_\alpha
=\sum_{\pi_0(\beta)\leq\pi_0(\alpha)}\lambda_\beta.\]
We denote the midpoint of the exchanged interval $I_\alpha$ by $m_\alpha=\frac{1}{2}(l_\alpha+r_\alpha)$. With these notations, $|I_\alpha|=\lambda_\alpha$.

Denote by $\Omega_\pi$ the matrix
$[\Omega_{\alpha\,\beta}]_{\alpha,\beta\in\mathcal{A}}$ given by
\begin{equation}\label{eq:Omega}
	\Omega_{\alpha\,\beta}=
	\left\{\begin{array}{cl} +1 & \text{ if
		}\pi_1(\alpha)>\pi_1(\beta)\text{ and
		}\pi_0(\alpha)<\pi_0(\beta),\\
		-1 & \text{ if }\pi_1(\alpha)<\pi_1(\beta)\text{ and
		}\pi_0(\alpha)>\pi_0(\beta),\\
		0& \text{ in all other cases.}
	\end{array}\right.\end{equation}
Then, the IET $T$ associated with $(\pi,\lambda)$ is given by $Tx=x+w_\alpha$, for any $x\in I_\alpha$, where $w=\Omega_\pi\lambda$. The vector $w$ is called the \emph{translation vector} of the IET.

In this work, we consider only \emph{irreducible} permutations, that is, pairs of bijections $(\pi_0, \pi_1)$ for which $\pi_1\circ\pi_0^{-1}\{1,\dots,k\}=\{1,\ldots,k\}$ implies $k=d$. We denote by $S^{\mathcal A}$ the set of all irreducible permutations.

As already mentioned, we will often deal in this article with \emph{symmetric} permutations. More precisely, a permutation $\pi$ on $\A$ is called \emph{symmetric} if $\pi_1(\alpha)=d+1-\pi_0(\alpha)$, for any $\alpha \in \A$, that is, any IET $T$ with combinatorial data given by $\pi$ exchanges intervals symmetrically.

Let $p:\{0,1,\ldots,d,d+1\}\to\{0,1,\ldots,d,d+1\}$ be the \emph{extended permutation} given by
\[p(j)=\left\{
\begin{matrix}
	\pi_1\circ\pi^{-1}_0(j)&\text{ if }&1\leq j\leq d\\
	j&\text{ if }&j=0,d+1.
\end{matrix}
\right.
\]
Following Veech (see \cite{Ve1}), denote by $\sigma=\sigma_\pi$ the
corresponding permutation on $\{0,1,\ldots,d\}$,
\[\sigma(j)=p^{-1}(p(j)+1)-1\text{ for }0\leq j\leq d.\]
Denote by $\Sigma(\pi)$ the set
of orbits for the permutation $\sigma$. Let $\Sigma_0(\pi)$ stand
for the subset of orbits that do not contain zero. The set $\Sigma(\pi)$ corresponds to singular points of translation flows, which in turn to the saddle points of the locally Hamiltonian flows. The cardinality of $\Sigma(\pi)$ relates to the matrix $\Omega_\pi$ by the formula
\[ \# \Sigma(\pi) = \dim(\textup{Ker}(\Omega_\pi)) + 1.\]
 For every $\mathcal{O}\in \Sigma(\pi)$, we denote by
\[
\mathcal{A}^-_{\mathcal{O}}=\{
\alpha\in\mathcal{A}, \ \pi_0(\alpha)\in \mathcal{O}\}, \qquad
\mathcal{A}^+_{\mathcal{O}}=\{ \alpha\in\mathcal{A}, \
\pi_0(\alpha)-1\in \mathcal{O}\} .
\]

On the space of all interval exchange transformations, there exists a classical notion of induction on the space of interval exchange transformations called \emph{Rauzy-Veech induction} (or \emph{Rauzy-Veech algorithm}), see \cite{Ra,Ve1}. Namely, we define
\[
\hat R:S^{\mathcal A}\times \R^{\A}_{>0}\to S^{\mathcal A}\times \R^{\A}_{>0},\ \hat R(\pi,\lambda):=(\pi^{(1)},\lambda^{(1)}),
\]
where $\tilde T:=(\pi^{(1)},\lambda^{(1)})$ is an IET obtained by inducing $T$ on $[0,|\lambda|-\min\{\lambda_{\pi_0^{-1}(d)},\lambda_{\pi_1^{-1}(d)}\})$. Keane in \cite{Keane} proved that for Lebesgue-a.e.\ parameter $\lambda$, the map $\hat R$ is defined indefinitely. We denote $\hat R^n(\pi,\lambda)=(\pi^{(n)},\lambda^{(n)})$ (also $(\pi,\lambda)=(\pi^{(0)},\lambda^{(0)})$). We also denote by $I^{(n)}_{\alpha}$, $\alpha\in\A$ the intervals exchanged by $\hat R^n(\pi,\lambda)$ and by $I^{(n)}$ its domain. For $j=0,1$, if $|I_{\pi_j^{-1}(d)}^{(n)}|>|I_{\pi_{1-j}^{-1}(d)}^{(n)}|$, then the symbol $\pi_j^{-1}(d)$ is called the $n$-th \emph{winner} of the Rauzy-Veech algorithm, while $\pi_{1-j}^{-1}(d)$ is the \emph{loser}.

Finally, we say that $(\pi,\lambda)$ is of \emph{top} (forward) type if $\lambda_{\pi_0^{-1}(d)}>\lambda_{\pi_1^{-1}(d)}$ and of \emph{bottom} (forward) type if $\lambda_{\pi_0^{-1}(d)}<\lambda_{\pi_1^{-1}(d)}$ and we describe a single action of $\hat R$ as of \emph{top} and \emph{bottom} type, in respective cases. If $\lambda_{\pi_0^{-1}(d)}=\lambda_{\pi_1^{-1}(d)}$, then $\hat R$ is not defined. It is easy to see that $\hat R$ is continuous on the set $\{(\pi,\lambda)\in S^{\mathcal A}\times \R^{\A}_{>0}\mid \lambda_{\pi_0^{-1}(d)}\neq \lambda_{\pi_1^{-1}(d)} \}$.

One of the crucial objects describing the action of $\hat R$ is the so-called \emph{Rauzy-Veech matrix}. More precisely, it is a $d\times d$ matrix $B(1):=B(1)(\pi,\lambda)$ such that
\[
\lambda^{(1)}B(1)=\lambda.
\]
Then for every $n\in\N$, we have
\[
\lambda^{(n)}B(n)=\lambda,
\]
where 
\begin{equation}\label{eq:defB}
B(n):=B(1)(\pi^{(n-1)},\lambda^{(n-1)})\cdot\ldots\cdot B(1)(\pi^{(0)},\lambda^{(0)}).
\end{equation} 
In other words, the matrices $B(\cdot)$ describe how the lengths of intervals change under the Rauzy-Veech induction. More precisely, for every $\alpha,\beta\in\A$, the coefficient $B_{\alpha\beta}(n)$ gives the number of visits of interval $I^{(n)}_{\alpha}$ in $I_{\beta}$ before its first return to $I^{(n)}$.

Another related object is the Rokhlin tower decomposition. Namely, for every $\alpha\in\A$ we denote by $q^{(n)}_{\alpha}$ the first return time of the interval $I^{(n)}_{\alpha}$ to $I^{(n)}$ via $T$. Then, the whole interval $I$ can be decomposed into the disjoint union of towers of intervals:
\[
I=\bigsqcup_{\alpha\in\A}\bigsqcup_{i=0}^{q^{(n)}_{\alpha}-1}T^i I^{(n)}_{\alpha}.
\]
In view of the interpretation of entries of $B(\cdot)$ as numbers of visits, we have that
\[
q^{(n)}_{\alpha}=\sum_{\beta\in\A}B_{\alpha\beta}(n).
\]
As we are going to deal with the cocycles, it is useful to introduce the natural action of $\hat R$ on the functions defined over the intervals. If $f:I\to \R$, then
\[
S(n)(f):I^{(n)}\to \R,\ \text{with}\ S(n)(f)(x):=S_{q^{(n)}_{\alpha}}(f)(x)\ \text{for every}\ x\in I^{(n)}_{\alpha},
\]
where $S_k(\cdot)(\cdot)$ denotes the $k$-th Birkhoff sum, see \eqref{def:BS}.

It is easy to see that $\hat R$ is not reversible. Indeed, it is typically a 2-to-1 map, with one pre-image coming from the top type induction and the other from the bottom type. However, there is the more natural setting of \emph{translation surfaces}, where the extended induction is actually reversible, which we now introduce.

Take $(\pi,\lambda)\in S^{\A}\times \R^{\A}_{>0}$ and consider a vector $\tau\in\R^{\mathcal A}$, such that for every $1\le j< d$ we have
\begin{equation}\label{eq: deftau}
\sum_{\alpha\in\A\mid \pi_0(\alpha)\le j}\tau_\alpha>0\quad\text{and}\quad
\sum_{\alpha\in\A\mid \pi_1(\alpha)\le j}\tau_\alpha<0.
\end{equation}
Then in $\C$ we consider a polygon made by connecting the vertices given by $\sum_{i=1}^j(\lambda_{\pi_0^{-1}(i)}+ i\tau_{\pi_0^{-1}(i)})$ and $\sum_{i=1}^j(\lambda_{\pi_1^{-1}(i)}+ i\tau_{\pi_1^{-1}(i)})$ for $1\le j\le d$ and the point 0. Due to \eqref{eq: deftau}, this polygon is made of two broken lines, one \emph{upper}, whose all non-extreme vertices are above the real axis, on the other, \emph{lower}, whose all non-extreme vertices are below the real axis. Note that both broken lines consist of $d$ intervals and that every segment in the upper broken line has its equivalent in the lower line obtained by translation. By gluing all of such pairs, using said translation, we obtain the \emph{translation surface} $M=(\pi,\lambda,\tau)$. We denote by $\Sigma=\Sigma(\pi,\lambda,\tau)$ the set of vertices of $(\pi,\lambda,\tau)$ which form the singularities of $M$. There exists a natural atlas of charts on $M\setminus\Sigma$ such that each transition map is a translation.

The translation structure on $(\pi,\lambda,\tau)$ allows to define the global notion of direction derived from $\C$ and, in particular, the directional flows on $(\pi,\lambda,\tau)$. We will consider two such flows: \emph{the vertical translation flow} $\phi=(\phi_t)_{t\in\R}$ which moves points upwards with unit speed and \emph{the horizontal translation flow} $\psi=(\psi_t)_{t\in\R}$ which moves the points rightwards with unit speed. Neither of these flows is defined in $\Sigma$. We often refer to the points in $\Sigma$ as \emph{singularities} of $\phi$ and $\psi$. It is easy to see that the directional flows preserve two-dimensional Lebesgue measure. Moreover, by considering a horizontal interval $I\subseteq M$ starting at $0$ of length $|\lambda|$ as a Poincar\'e section for $\phi$, we get that the first return map to $I$ is an IET $T=(\pi,\lambda)$. We can use this map to get a different representation of the surface $(\pi,\lambda,\tau)$. Namely, for every $\alpha\in\A$ we consider a \emph{zippered rectangle} $D_{\alpha}:=\bigcup_{t\in[0,h_{\alpha})}\phi_t(I_{\alpha})$, where $h_{\alpha}$ is the first (positive) return time of $I_{\alpha}$ to $I$ via $\phi$. We refer to $h_{\alpha}$ as the \emph{height of the zippered rectangle} $D_{\alpha}$. With these notations, it is easy to show that the \emph{heights vector} $h=(h_{\alpha})_{\alpha\in\A}$ satisfies $h=-\Omega_{\pi}\tau$. {We refer to the vertical segment $E_\alpha := \{\phi_t(l_\alpha) \mid t \in [0, h_\alpha)\}$ as the \emph{edge of the zippered rectangle} $D_\alpha$.}

The set $\Sigma(\pi)$ defined earlier corresponds to the set $\Sigma(\pi,\lambda,\tau)$ in the way that the permutation $\sigma$ describes the gluing of sides in the polygonal construction around a fixed singularity. In particular, if $\phi$ has only one singularity, then $\Sigma(\pi)$ has only one element. %
{Moreover, the genus of the corresponding surface is related to $\Sigma(\pi)$ by the formula
\begin{equation}
\label{eq:genus}
g = \frac{d + 1 - \#\Sigma(\pi)}{2}.
\end{equation}
}

As already mentioned, we can define an extension of $\hat R$ on the space of translation surfaces so that it is invertible. More precisely, we define
\[
R(\pi,\lambda,\tau):=(\pi^{(1)},\lambda^{(1)},\tau^{(1)}),\ \text{where}\ (\pi^{(1)},\lambda^{(1)})=\hat R(\pi,\lambda)\ \text{and}\ \tau^{(1)}:=\tau B(1)^{-1}.
	\]
	One can check that $\tau^{(1)}$ still satisfies \eqref{eq: deftau}, hence the surface $(\pi^{(1)},\lambda^{(1)},\tau^{(1)})$ is well defined. We refer to $R$ as the \emph{2-dimensional Rauzy-Veech induction}. However, we simply write "Rauzy-Veech induction" if it does not cause any confusion. As before, we define $(\pi^{(n)},\lambda^{(n)},\tau^{(n)}):=R^n(\pi,\lambda,\tau)$. Moreover, one can show that the vector of heights in the zippered rectangle representation $h^{(n)}=(h^{(n)})_{\alpha\in\A}$ is given by the formula $h^{(n)}=B(n)h$.
	
	For $(\pi,\lambda,\tau)$ there exists only one pre-image $R^{-1}(\pi,\lambda,\tau):=(\pi^{(-1)},\lambda^{(-1)},\tau^{(-1)})$ for which $\tau^{(-1)}$ satisfies \eqref{eq: deftau}. We say that $(\pi,\lambda,\tau)$ is of \emph{backward top} type if $\sum_{\alpha\in\A}\tau_{\alpha}<0$ and of \emph{backward bottom} type if $\sum_{\alpha\in\A}\tau_{\alpha}>0$. The backward induction $R^{-1}$ is not defined if $\sum_{\alpha\in\A}\tau_{\alpha}=0$. It is worth underlining that if $(\pi,\lambda,\tau)$ is of backward top (bottom) type, then $(\pi^{(-1)},\lambda^{(-1)},\tau^{(-1)})$ is of top (bottom) type.
	
The set $S^{\mathcal A}$ can be divided into a disjoint union of sets, minimal and invariant under projected action of $\hat R$ called \emph{Rauzy graphs}.

	One of the main reasons to consider Rauzy-Veech induction is that it defines a renormalization scheme on the space of translation surfaces as well as on the space of interval exchange transformations. To be more precise, we fix a Rauzy graph $\mathcal G\subseteq S^{\A}$ and consider a unit simplex $\Lambda^{\A}:=\{\lambda\in \R^{\A}_{>0}\mid |\lambda|=1\}$. Moreover we consider $\Theta_{\pi}:=\{\tau\in \R^{\A}\mid \tau\ \text{satisfies \eqref{eq: deftau}}\}$. We define a \emph{Rauzy-Veech renormalization} $\mathcal R$ on $\mathcal M:=\bigsqcup_{\pi\in\mathcal G}\{\pi\}\times \Lambda^{\A}\times \Theta_{\pi}$ given by
	\[
	\mathcal R(\pi,\lambda,\tau):=(\pi^{(1)},\frac{\lambda^{(1)}}{|\lambda^{(1)}|},|\lambda^{(1)}|\cdot\tau^{(1)}).
	\]
Let $\mathcal M^1:=\{(\pi,\lambda,\tau)\in\mathcal M\, |\, -\langle\lambda,\Omega_\pi\tau\rangle=1\}$ recalling that $-\langle\lambda,\Omega_\pi\tau\rangle=\langle\lambda,h\rangle$ is the area of the surface $(\pi,\lambda,\tau)$.	It was proven by Masur \cite{Ma:erg} and Veech \cite{Ve1}, that there exists an infinite measure $\mu$ on $\mathcal M^1$, which is equivalent to the product of counting and Lebesgue measures, with respect to which $\mathcal R$ is measure-preserving, recurrent and ergodic. This already allows us to deduce a large number of results that hold for typical translation surface or interval exchange transformation. However, to talk about the hyperbolic properties of the cocycles, we want to work with finite measures. For this reason, Zorich in \cite{Zor} introduced an acceleration of the Rauzy-Veech induction, obtained by inducing to an appropriately chosen set, such that the induced measure is finite. Namely, we consider a subset $\mathcal X$ of triples $(\pi,\lambda,\tau)\in \mathcal M^1$ such that $(\pi,\lambda,\tau)$ is of forward top type and backward bottom type or it is of forward bottom type and backward top type. Then $\mu(\mathcal X)<\infty$. The map obtained by inducing $\mathcal R$ on $\mathcal X$ is called \emph{Zorich induction}, and we denote it by $\mathcal Z$. The cocycle obtained by inducing $B$, { as defined in \eqref{eq:defB}}, on $\mathcal X$ is called \emph{Kontsevich-Zorich cocycle} (KZ-cocycle), and we denote it by $Z$. We also denote $m:=\mu|_{\mathcal X}$. For every $(\pi,\lambda,\tau)\in\mathcal X$, if $n_k:=n_k(\pi,\lambda,\tau)$ is the $k$-th return time of $(\pi,\lambda,\tau)$ to $\mathcal X$ via $\mathcal R$, where $k\ge 1$, then $\mathcal Z(\pi,\lambda,\tau)=\mathcal R^{n_k}(\pi,\lambda,\tau)$ and we denote
	\[
	Z(k)=Z(k)(\pi,\lambda,\tau):= B(n_{k}-n_{k-1})(\pi^{(n_{k-1})},\lambda^{(n_{k-1})},\tau^{(n_{k-1})}),
	\]
	as well as $Q(k):=Z(k)\cdot\ldots\cdot Z(1)$.
	Moreover, since $\mu$ is infinite and $m$ is finite, the function of the first return time $n_1(\cdot)$ is unbounded. Since $\mathcal R$ is ergodic, then so is $\mathcal Z$.
	
	One of the most important features of the Kontsevich-Zorich cocycle is its log-integrability, that is
	\begin{equation}\label{eq: KZintegrabilty}
	\int_{\mathcal X}\log\|Z(1)\|\,dm<\infty.
	\end{equation}
	Recall that the log-integrability is inherited by induced cocycles, i.e., accelerating the induction by considering returning to smaller subsets of $\mathcal X$ does not destroy this property.
	
	Later in the article, we will use many notions defined for the Rauzy-Veech induction and then used for the Kontsevich-Zorich induction. However, for simplicity of the notation, we will often write $(\cdot)^{(k)}$ instead of $(\cdot)^{(n_k)}$ whenever it does not cause confusion.
	
	For more information on interval exchange transformations and translation surfaces, we refer to \cite{Vi0,ViB} and \cite{Yoc}.

\section{Horizontal spacing}
In this section, we prove a crucial property of interval exchange transformations, namely that levels of each Rokhlin tower obtained via the Rauzy-Veech algorithm are roughly evenly distributed inside an interval $I=[0,|I|)$. {This is partially inspired by the proof} of the $d+2$-gaps theorem obtained by Taha in \cite{Ta}.

\begin{proposition}\label{prop: horizontal spacing}
	Let $T=(\pi,\lambda)$ and let $\tilde\tau\in\Theta_{\pi}$ be rationally independent. There exists $\epsilon:=\epsilon(\tilde\tau)>0$ and $C:=C(\tilde\tau,\epsilon)>0$ such that for every $\tau\in\Theta_{\pi}$, if $\mathcal{R}^n(\pi,\lambda,\tau)\in \{\pi\}\times\R^\A_{>0}\times B_{\Theta_{\pi}}(\tilde\tau,\epsilon)$\footnote{$B_{\Theta_{\pi}}(\tilde\tau,\epsilon)$ is the ball in $\Theta_\pi$ (on $\Theta_{\pi}$ we consider standard supremum metric) of radius $\epsilon$ centered at $\tilde\tau$.}, then for every $\alpha\in\A$ we have that {any interval ${\mathcal{J}}\subseteq I\setminus \bigsqcup_{j=0}^{q_\alpha^{(n)}-1}T^j I_{\alpha}^{(n)}$ satisfies
	\[
	|{\mathcal{J}}|<C\cdot \max_{\beta\in\A\setminus\{\alpha\}}|I_{\beta}^{(n)}|
	\]}
	and
	\[
	\max\big\{\min_{0\leq i<q^{(n)}_\alpha}\operatorname{dist}(0,T^i(I^{(n)}_{\alpha})),\min_{0\leq i<q^{(n)}_\alpha}\operatorname{dist}(|I|,T^i(I^{(n)}_{\alpha}))\big\}<C\cdot \max_{\beta\in\A\setminus\{\alpha\}}|I_{\beta}^{(n)}|.
	\]
\end{proposition}
A natural approach to studying the spacing between intervals is to consider them inside the corresponding surface and then study the orbit of the horizontal translation flow on the translation surface given by $(\pi,\lambda,\tau)$. Hence, we first prove the following reduction to the regime of translation surfaces.
\begin{lemma}\label{lem: horizontalspacing}
	Fix $\pi\in S^\A$ and let $\tilde\tau\in\Theta_{\pi}$ be a rationally independent vector. There exists $\epsilon:=\epsilon(\tilde\tau)>0$ and $C:=C(\tilde\tau,\epsilon)>0$ such that for any $\lambda\in\R^\A_{>0}$, any $\tau\in B_{\Theta_{\pi}}(\tilde\tau,\epsilon)$ %
and any $\alpha\in \A$ {the following holds.
Consider the translation surface $(\pi, \lambda, \tau)$ with associated zippered rectangles $\{D_\beta\}_{\beta \in \A}$ of heights $(h_\beta)_{\beta \in \A}$ and vertical (resp. horizontal) flow $\phi_t$ (resp. $\psi_t$). Then, for any vertical segment $J$ of length $h_\alpha$ starting at $I_\alpha$ and for any point $x \in J$, the orbit segment given by the first return of $x$ to $J$ by the horizontal flow crosses at most $C$ times the edges of $\{D_\beta\}_{\beta \neq \alpha}$.}
\end{lemma}
\begin{proof}
Take $\lambda\in\R^\A_{>0}$ and consider the surface $(\pi,\lambda,\tilde{\tau})$. Due to the rational independence of $\tilde{\tau}$, this surface has no horizontal saddle connections. Fix $\alpha\in\A$ and $x_0\in I_{\alpha}$. Let $J:=\bigcup_{t\in [0,h_{\alpha})} \phi_t(x_0)$ be a vertical segment inside the zippered rectangle $D_{\alpha}=\bigcup_{t\in [0,h_{\alpha})} \phi_t(I_{\alpha})$. Since $J$ is a transversal interval for the horizontal flow $\psi$, we may consider a first return map $S:J\to J$ via $\psi$. Note that the endpoints of $J$ lie on the outgoing separatrix of the singularity at $0$, hence, due to \cite[Corollary 5.5]{Yoc}
and the fact that $(\pi,\lambda,\tilde{\tau})$ has no horizontal connections, $S$ is an IET of $d = \#\A$ intervals.

{The general strategy of the proof is the following. The lengths of the intervals exchanged by $S$ vary continuously depending on $\tau$, without changing the combinatorial data of $S$, as long as the order of the first hitting points of the outcoming and incoming separatrices is kept. Moreover, as long as the perturbation of the parameter $\tau$ is small enough, the number and order of zippered rectangles visited by the points from $J$ via the horizontal flow do not change. Since there are only finitely many exchanged intervals, by taking $C$ as the maximal number of zippered rectangles visited this way, we obtain the desired result. Below, {for the sake of clarity and readability, we provide a detailed proof in a particular case, but let us mention that the general case follows along similar lines.}}

From now on, for simplicity we assume that $\pi_0(\alpha)=1$ and $\pi_1(\alpha)=d$. Moreover, we assume that
\begin{equation}\label{eq: conditionontau}
\tilde{\tau}_\alpha>0,\quad\text{and}\quad\tilde{\tau}_\beta<0, \quad \text{ for every }\beta\in \A\setminus\{\alpha\},
\end{equation}
see Figure~\ref{fig:spacing}.
Note that due to the rational independence of $\tilde{\tau}$, we also have $\sum_{\alpha\in\A} \tilde{\tau}_{\alpha}\neq 0$. Hence we can also assume that
\begin{equation}\label{eq: conditionontau2}
	\sum_{\beta\in\A}\tilde{\tau}_\beta<0.
\end{equation}

Let us describe the parameters of $S$. First, we have
\begin{equation}\label{eq: lengthofJ}
	|J|=h_{\alpha}=-(\Omega_{\pi}\tilde\tau)_{\alpha}=-\sum_{\beta\in\A\setminus\{\alpha\}}\tilde\tau_{\beta}.
	\end{equation}
Moreover, we know that the discontinuities of $S$ are given by the first backward intersections of $J$ with incoming separatrices for $\psi$. More precisely, if $\sigma\in\Sigma$ is a singularity
 and $(\psi_{-t}(\sigma))_{t\in(0,\infty)}$ is one of its incoming horizontal separatrices, then a point $y\in J$ is a discontinuity of $S$ obtained from this separatrix either if $y=\sigma$, or there exists $t(y)\in(0,\infty)$ such that
\[
y=\psi_{-t(y)}(\sigma)\quad\text{and}\quad(\psi_{-t}(\sigma))_{t\in(0,t(y))}\cap J=\emptyset.
\]
By choice of $J$, the only possible case for $y=\sigma$ to hold is if $y=x_0=\partial I_{\alpha}$. By shifting via $\psi$ slightly the interval $J$, we may assume that $y\notin\Sigma$. By shifting it even more, but still inside $D_{\alpha}$, since $\tau_{\pi_0^{-1}(2)}<0$, we may assume that if $z_0$ is the intersection of $J$ with the side connecting $0$ and $\lambda_{\alpha}+i\tilde\tau_{\alpha}$, then
\begin{equation}\label{eq: z_0definition}
	\operatorname{Im}(z_0)>\tilde\tau_{\pi_0^{-1}(1)}+\tilde\tau_{\pi_0^{-1}(2)}.
	\end{equation}
The initial segments of horizontal incoming separatrices correspond to the {horizontal} segments in the polygonal representation of $(\pi,\lambda,\tilde\tau)$ whose right-hand side endpoints are in the vertices of the polygon. Due to \eqref{eq: conditionontau}, the right-hand side endpoints of these segments are the vertices
\[ \Big\{ \sum_{\pi_0(\beta)\le k}(\lambda_{\beta}+i\tilde\tau_\beta) \,\Big| \, k=2,\ldots,d \Big\}.\]
By condition \eqref{eq: conditionontau2}, the vertex $\sum_{ \beta \in \A }(\lambda_{\beta}+i\tilde\tau_\beta)$, that is the one corresponding to $k=d$, is below the $x$-axis. Moreover, due to the fact that %
{$\tilde\tau_{\pi_1^{-1}(d)}=\tilde\tau_{\alpha}>0$,} the subinterval of $J$ given by $\bigcup_{t\in [h_{\alpha}+\sum_{\beta\in\A}\tilde{\tau}_\beta,h_{\alpha})} \phi_t(x_0)$ is fully below the $x$-axis. Hence, since $h_{\alpha}+\sum_{\beta\in\A}\tilde{\tau}_\beta=\tilde\tau_{\alpha}=\tilde\tau_{\pi_0^{-1}(1)}$, the discontinuity of $S$ corresponding to $k=d$ is the point $y_{d}$ given by
\begin{equation}\label{eq: lastendpoint}
y_{d}:=\phi_{\tilde\tau_{\pi_0^{-1}(1)}}(x_0).
\end{equation}
On the other hand, for every $k=2,\ldots,d-1$, the initial segment of the corresponding incoming separatrix lies above the $x$-axis and due to \eqref{eq: conditionontau} (and by \eqref{eq: z_0definition}), the first backward hit of the incoming separatrix via $\psi$ is the point
\begin{equation}\label{eq: otherendpoints}
y_k:=\phi_{t_k}(x_0), \qquad \text{ where } t_k:= \sum_{\pi_0(\beta)\le k}\tilde\tau_{\beta}.
\end{equation}
{Notice} that we can use the above notation also for $k=1$, i.e.\ $t_1=\tilde\tau_{\pi_0^{-1}(1)}$ and $y_1=\phi_{t_1}(x_0)=y_d$.
Note that none of the formulas \eqref{eq: lengthofJ}, \eqref{eq: lastendpoint}, or \eqref{eq: otherendpoints} depend on $\lambda$.

It remains to describe the zippered rectangles crossed by the horizontal orbit segment of any point $x\in J$. Note that due to \eqref{eq: conditionontau} and \eqref{eq: conditionontau2} the order of the points $y_k$ on $J$, counting upwards, is $y_{d-1},y_{d-2},\ldots,y_{3},y_{2},y_{d}$. Assume first that $x\in (y_{k+1},y_k)\subseteq J$ for some $k\in\{2,\ldots,d-2\}$. Then the orbit segment starting at $x$ given by the forward action of $\psi$ leaves the zippered rectangle $D_{\alpha}$ then goes through the cylinders $D_{\pi_0^{-1}(\ell)}$ for $2\leq \ell\leq k+1$ and inside the zippered rectangle $D_{\pi_0^{-1}(k+1)}$ it crosses the edge given by the vertices
$\sum_{ \pi_0(\beta)\le k}(\lambda_{\beta}+i\tilde\tau_\beta)$ and $\sum_{ \pi_0(\beta)\le k+1}(\lambda_{\beta}+i\tilde\tau_\beta)$. This describes the part of the segment in the part of the polygonal representation above the $x$-axis. After crossing the side of the polygon, it goes through zippered rectangles
$D_{\pi_1^{-1}(\ell)}$, with $\pi_1(\pi_0^{-1}(k+1))\leq \ell \leq d-1$ in the lower part of the polygon. After that, the orbit goes to $D_{\pi_1^{-1}(d)}=D_{\alpha}$ and returns directly to $J$. Hence in the case when $x\in (y_{k+1},y_k)$ for some $k\in\{2,\ldots,d-2\}$, the orbit crosses at most $2d-2$ edges of zippered rectangles different than $D_\alpha$.

To describe the horizontal orbit segment for $x\in (y_{2},y_{d})=(y_{2},y_{1})$ it %
{suffices} to take $k=1$ in the above description of the passage of the orbit through zippered rectangles. %

If $x\in(x_0,y_{d-1})$ then the orbit segment starting at $x$ obtained by acting forward via $\psi$ passes all zippered rectangles in the upper part of the polygonal representation of $(\pi,\lambda,\tilde{\tau})$, before crossing the edge given by $\sum_{ \pi_0(\beta) < d}(\lambda_{\beta}+i\tilde\tau_\beta)$ and $\sum_{\beta\in\A}(\lambda_{\beta}+i\tilde\tau_\beta)$. Then, before returning to the zippered rectangle $D_{\alpha}$ it passes through the zippered rectangles $D_{\pi_1^{-1}(\ell)}$, with $\pi_1(\pi_0^{-1}(d))\leq \ell\leq d-1$. After that, the orbit goes to $D_{\pi_1^{-1}(d)}=D_{\alpha}$ and returns directly to $J$. Hence, it also crosses no more than $2d-2$ edges of zippered rectangles different than $D_\alpha$.
\begin{figure}[h!]
	\includegraphics[width=0.5\textwidth]{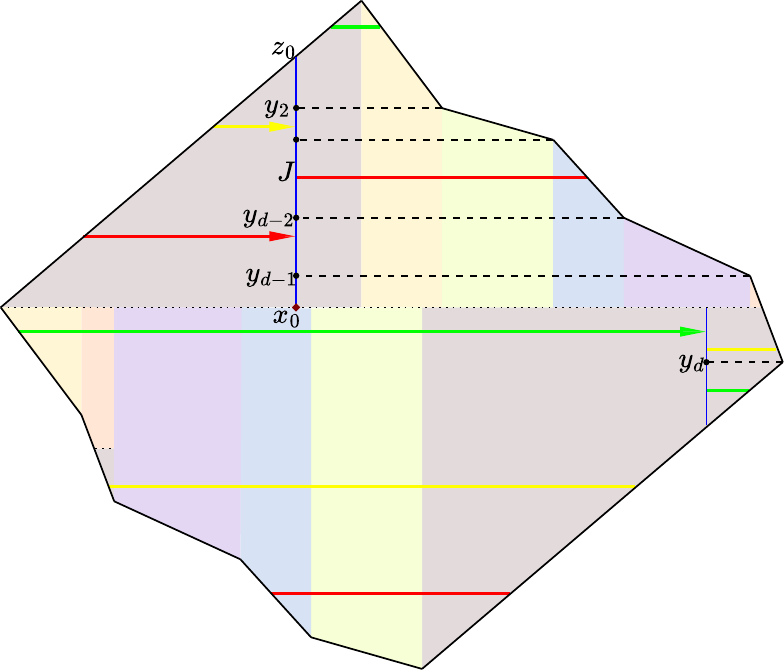}
	\caption{The polygonal representation of the surface $(\pi,\lambda,\tau)$ with indicated zippered rectangles. The picture shows three different orbits, corresponding to three of the cases described in the proof of Lemma \ref{lem: horizontalspacing}.} \label{fig:spacing}
\end{figure}

Finally, if $x\in({y_d},
\phi_{h_{\alpha}}(x_0))$, then the orbit of $x$ does not leave the lower part of the polygonal representation of $(\pi,\lambda,{\tilde\tau})$ before returning to $D_{\alpha}$. After leaving the zippered rectangle $D_{\alpha}$, it goes through zippered rectangles $D_{\pi_1^{-1}(\ell)}$, with $\pi_1(\pi_0^{-1}(d))+1\leq\ell\leq d-1$ (note that here, if $\pi_1(\pi_0^{-1}(d))=d-1$, then the set of the indices $\ell'$ is empty, that is the horizontal orbit of $x$ does not leave $D_{\alpha}$ before the first return to $J$). After that, the orbit goes to $D_{\pi_1^{-1}(d)}=D_{\alpha}$ and returns directly to $J$. Hence, in this case, the horizontal orbit of $x$ does not cross more than $d-1$ edges of zippered rectangles different than $D_\alpha$. Thus the thesis of the Lemma holds for $(\pi,\lambda,\tilde{\tau})$ with $C=2d-2$.

Since the values in \eqref{eq: lengthofJ}, \eqref{eq: lastendpoint} and \eqref{eq: otherendpoints} change continuously with respect to $\tilde\tau$, we can find an $\epsilon>0$ such that the inequalities \eqref{eq: conditionontau} and \eqref{eq: conditionontau2} hold {for any $\tau \in B_{\Theta_\pi}(\tilde\tau, \epsilon)$} and, in particular, the order in which the discontinuities coming from specific vertices appear on the vertical section remains unchanged. Hence, the order of visited zippered rectangles stays the same, and the thesis holds with $C=2d-2$.

\end{proof}

\begin{proof}[Proof of Proposition \ref{prop: horizontal spacing}]
Fix $\tilde\tau\in\Theta_\pi$ and let $\epsilon=\epsilon(\tilde\tau)$ and $C=C(\tilde\tau,\epsilon)$ be given by Lemma \ref{lem: horizontalspacing}. Fix $\alpha\in\A$ and $n\in\N$ such that $\mathcal R^{n}(\pi,\lambda,\tau)\in \{\pi\}\times\R_{>0}^\A\times B_{\Theta_{\pi}}(\tilde\tau,\epsilon)$. Let $(\pi,\lambda^{(n)},\tau^{(n)})=R^{n}(\pi,\lambda,\tau)$. Then $\mathcal R^{n}(\pi,\lambda,\tau)=(\pi,e^s\lambda^{(n)},e^{-s}\tau^{(n)})$ with $s=-\log |\lambda^{(n)}|$. Since ${R}^n(\pi, \lambda, \tau)$ and $(\pi, \lambda, \tau)$ describe the same surface, we can decompose the surface $(\pi, \lambda, \tau)$ into zippered rectangles $\{D^{(n)}_\beta\,|\ \beta\in\A\}$ according to numerical data $(\pi,\lambda^{(n)},\tau^{(n)})$, i.e.\ $D^{(n)}_\beta=\bigcup_{0\leq t<h^{(n)}_\beta}\phi_t I^{(n)}_\beta$, where $h^{(n)}=(h^{(n)}_\beta)_{\beta\in\A}=-\Omega_\pi\tau^{(n)}$. Then $T^jI^{(n)}_\beta=\phi_{t^\beta_j} I^{(n)}_\beta$ for $0\leq j\leq q^{(n)}_\beta$, where $0=t^\beta_0<t^\beta_1<\ldots<t^\beta_{q^{(n)}_\beta}=h^{(n)}_{\beta}$ are consecutive recurrence times of $I^{(n)}_\beta$ to the interval $I$.

Let $x\in T^m(I^{(n)}_{\alpha})$ for some $0\leq m< q^{(n)}_\alpha$. Consider $T^m(I^{(n)}_{\alpha})$ as a subset of the surface $(\pi,\lambda,\tau)$ and let $J$ be a {vertical segment of length $h_{\alpha}^{(n)}$} inside $D_{\alpha}^{(n)}$, which starts at $I_{\alpha}^{(n)}$ (at $x_0=T^{-m}x$),
which contains $x$. As $R^n(\pi, \lambda, \tau)$ is obtained from $\mathcal{R}^n(\pi, \lambda, \tau)$ by vertical and horizontal rescaling, they have the same (up to rescaling) decompositions into zippered rectangles. As $\mathcal R^{n}(\pi,\lambda,\tau)\in \{\pi\}\times\R_{>0}^\A\times B_{\Theta_{\pi}}(\tilde\tau,\epsilon)$, in view of Lemma \ref{lem: horizontalspacing} and denoting by $r$ the first return time of $x$ to $J$ via the horizontal flow $\psi$, the orbit $(\psi_{t}(x))_{t\in(0,r)}$ crosses at most $C$ edges of the zippered rectangles $D^{(n)}_\beta$, $\beta\in\A\setminus\{\alpha\}$ before returning to $J\subset D^{(n)}_\alpha$. In particular, $\psi_{r}(x)\in J\subset D^{(n)}_\alpha$.

There are two possibilities: either $(\psi_{t}(x))_{t\in(0,r)}$ is not included in $I$, which is equivalent to the fact that $T^m(I^{(n)}_{\alpha})$ is the rightmost level of the tower $\bigcup_{j=0}^{q^{(n)}_\alpha-1}T^j(I^{(n)}_{\alpha})$ {with respect to the linear order on $I$}, or $(\psi_{t}(x))_{t\in(0,r)}$ is a subinterval of $I$. Indeed, if $x$ does not belong to the rightmost level of the tower $\bigcup_{j=0}^{q^{(n)}_\alpha-1}T^j(I^{(n)}_{\alpha})$, then denote by $T^{m'}(I^{(n)}_{\alpha})$ the next level travelling horizontally. Then $\psi_r(x)\in T^{m'}(I^{(n)}_{\alpha})\subset I$. It follows that the horizontal interval $(\psi_{t}(x))_{t\in(0,r)}$ is a subinterval of $I$.

{Let us assume first that $(\psi_{t}(x))_{t\in(0,r)} \subseteq I$.} Recall that for any $\beta\in\A$, the tower $\bigcup_{j=0}^{q^{(n)}_\beta-1}{T^j(I^{(n)}_\beta)}$ is fully included inside the zippered rectangle {$D_{\beta}^{(n)}$.} %
 Hence, the fact that $(\psi_{t}(x))_{t\in(0,r)}$ crosses at most $C$ times the edges of zippered rectangles is equivalent to the fact that it crosses at most $C$ times the endpoints of levels of towers corresponding to other symbols than $\alpha$.
 Moreover, since $\psi_r(x)\in T^{m'}(I^{(n)}_\alpha)$ (this is the next level of the tower traveling horizontally), the interval between $T^{m}(I^{(n)}_\alpha)$ and $T^{m'}(I^{(n)}_\alpha)$ contains at most $C$ levels of towers other than $\bigcup_{j=0}^{q^{(n)}_\alpha-1}T^j(I^{(n)}_{\alpha})$.

 If on the other hand, we have $(\psi_{t}(x))_{t\in(0,r)} \not\subseteq I$, then $(\psi_{t}(x))_{t\in(0,r)}$ contains the right-hand-side endpoint of $I$ and hence contains the rightmost interval in $I\setminus \bigsqcup_{j=0}^{q_{\alpha}^{(n)}-1}T^jI^{(n)}_\alpha$. As $(\psi_{t}(x))_{t\in(0,r)}$ crosses at most $C$ edges of the zippered rectangles $D^{(n)}_\beta$, $\beta\in\A\setminus\{\alpha\}$, the argument used in the previous paragraph shows that the rightmost interval in $I\setminus \bigsqcup_{j=0}^{q_{\alpha}^{(n)}-1}T^jI^{(n)}_\alpha$ contains at most $C$ levels of towers other than $\bigcup_{j=0}^{q^{(n)}_\alpha-1}T^j(I^{(n)}_{\alpha})$.

 It remains to consider the leftmost component of $I\setminus \bigsqcup_{j=0}^{q_{\alpha}^{(n)}-1}T^jI^{(n)}_\alpha$, {with respect to the linear order on $I$}. This is done by choosing a point $x$ in the intersection of $J$ and the leftmost level of $\bigsqcup_{j=0}^{q_{\alpha}^{(n)}-1}T^jI^{(n)}_\alpha$ and taking $r_0\geq 0$ such that $\psi_{-r_0}(x)=0$ is the left endpoint of $I$. Then the orbit $(\psi_{-t}(x))_{t\in(0,r_0)}$ contains the leftmost component of $I\setminus \bigsqcup_{j=0}^{q_{\alpha}^{(n)}-1}T^jI^{(n)}_\alpha$. Next choose $x_\delta=\phi_\delta x\in J$ for small enough $\delta>0$. Let $y=\psi_{-r}x_\delta\in J$ be the first backward return point for the horizontal flow to $J$. For all $\delta>0$ small enough, we have $r>r_0$. In view of Lemma \ref{lem: horizontalspacing}, the orbit $(\psi_{t}(y))_{t\in(0,r)}=(\psi_{-t}(x_\delta))_{t\in(0,r)}$ crosses at most $C$ edges of the zippered rectangles $D^{(n)}_\beta$, $\beta\in\A\setminus\{\alpha\}$. As $r>r_0$, the orbit $(\psi_{-t}(x_\delta))_{t\in(0,r_0)}$, and hence the $\delta$-close orbit $(\psi_{-t}(x))_{t\in(0,r_0)}$, crosses at most $C$ edges of the zippered rectangles $D^{(n)}_\beta$, $\beta\in\A\setminus\{\alpha\}$. Finally, the argument used in the previous paragraphs shows that the leftmost interval in $I\setminus \bigsqcup_{j=0}^{q_{\alpha}^{(n)}-1}T^jI^{(n)}_\alpha$ contains at most $C$ levels of towers other than $\bigcup_{j=0}^{q^{(n)}_\alpha-1}T^j(I^{(n)}_{\alpha})$.
This completes the proof of the proposition.
\end{proof}

\section{Diophantine conditions}
\label{sc:Diophantine}
In this section, we show that a certain collection of properties holds for typical IETs. These properties concern the parametric description of properly chosen return sets for the Zorich induction $\mathcal Z$. For an overview of the history and recent development of techniques related to different types of Diophantine properties for IETs and translation surfaces, we refer to the survey article \cite{Ul:ICM}.

Fix a Rauzy graph $\mathcal G\subset S^\A$ and consider an IET $(\pi,\lambda)$ such that $\pi\in\mathcal G$ is an irreducible permutation over an alphabet $\A$ of $d\ge 2$ elements. Assume that there exists an increasing sequence $\{q_n\}_{n\in\N}$ of natural numbers, a sequence $\{\{T^i(\Delta^{n}) \mid 0\leq i<q_n\}\}_{n\in\N}$ of Rokhlin towers, and a constant $C>1$, such that
\begin{equation}
	\label{eq: qn1}
	\limsup_{n\to\infty}(q_n)^{\frac{1}{n}}<\infty;
\end{equation}
and, for any $n \in \N$,
	\begin{gather}
	\label{eq: qn2,5}
	q_{n+1}\geq 320C^2 q_n, \\%
	\label{eq: qn3}
	C\geq q_n|\Delta^n|\geq C^{-1}>0, \\%
		\label{eq: qn3,5}
	T^i(\Delta^n) \text{ is an interval disjoint from the discontinuities of $T$, for any $0 \leq i < q_n$,}\\
	\label{eq: qn4}
	\Big(\bigcup_{0\leq i<q_n}T^i(\Delta^n)\Big)^c\text{ consists of at most $q_{n}+1$ intervals (holes) of lengths less than $\tfrac{C}{q_n}$;}\\
	\label{eq: qn5}
	\text{$T^{q_n}\mid_{\Delta^n}$ is a translation by a number from $(\tfrac{1}{16}|\Delta^n|,\tfrac{1}{8}|\Delta^n|)$.}
\end{gather}

\begin{proposition}\label{prop:full}
For almost every $\lambda\in \Lambda^{\mathcal A}$ there exists an increasing sequence $\{q_n\}_{n\in\N}$ of natural numbers and a sequence $\{T^i(\Delta^n) \mid 0\leq i<q_n\}_{n\in\N}$ of Rokhlin towers satisfying \eqref{eq: qn2,5}--\eqref{eq: qn5}. Moreover, these conditions are obtained via returning to an appropriate subset $\mathfrak L\subseteq \mathcal M$ of positive measure.
\end{proposition}
\begin{proof}
Fix $\pi\in S^{\A}$ and let $\tau\in\Theta_{\pi}$ be rationally independent. We will choose an appropriate on which to induce the Zorich induction $\mathcal Z$. We will define $\Delta_n$ as a properly chosen subinterval of the base of the tower $\bigcup_{j=0}^{q^{(n)}_\alpha}T^j(I^{(n)}_{\alpha})$, for some $\alpha\in\A$.

Let $\mathcal D_1\subseteq\Lambda^\A$ be a set of length parameters such that
\begin{equation}\label{def:D1}
\mathcal D_1:=\{\lambda\in \Lambda^\A\mid \lambda_{\pi_0^{-1}(d)}\ge \tfrac{7}{8}\ \text{and}\ \tfrac{1}{16}<w_{\pi_0^{-1}(d)}< \tfrac{1}{8} \},
\end{equation}
where $w = \Omega_\pi \lambda $ is the translation vector associated with $(\pi, \lambda)$.

Note that since $\lambda_{\pi_0^{-1}(d)}\ge \tfrac{7}{8}$ for $\lambda\in\mathcal D_1$, the IET $(\pi,\lambda)$ has top type for the RV induction.
Moreover, whenever $\mathcal Z^n(\pi,\lambda, \tau)\in \{\pi\}\times\mathcal D_1\times \Theta_{\pi}$, we have that %
{$T^{q^{(n)}_{\pi_0^{-1}(d)}}$} acts on $I_{\pi_0^{-1}(d)}^{(n)}$ via translation by a number between $\tfrac{1}{16}|I_{\pi_0^{-1}(d)}^{(n)}|$ and $\tfrac{1}{8}|I_{\pi_0^{-1}(d)}^{(n)}|$. Thus, the interval $I_{\pi_0^{-1}(d)}^{(n)}$ satisfies \eqref{eq: qn5}.

Let now $\tilde\tau\in \Theta_{\pi}$ be such that $(\pi,\lambda,\tilde\tau)$ is of bottom backward type and $\tilde\tau$ is rationally independent. Let $\epsilon=\epsilon(\tilde\tau)>0$ and $C_1=C(\tilde\tau,\epsilon)>0$ be given by Proposition \ref{prop: horizontal spacing}. Denote
\[
\mathcal E:= \R_{>0}B_{\Theta_\pi}(\tilde\tau,\epsilon).
\]
In view of Proposition \ref{prop: horizontal spacing}, whenever $\mathcal Z^n(\pi,\lambda,\tau)\in \{\pi\}\times\mathcal D_1\times \mathcal E$, the tower $\bigcup_{j=0}^{q^{(n)}_{\pi_0^{-1}(d)}-1} {I_{\pi_0^{-1}(d)}^{(n)}}$ satisfies \eqref{eq: qn4} with $C=C_1$. Moreover, since the cone $\mathcal E$ is relatively compact (with respect to the Hilbert pseudometric) in $\Theta_{\pi}$ and, by considering smaller $\epsilon$ if necessary, we may assume that all $\tau\in\mathcal E$ are of bottom backward type and
\[
\sup_{\tau\in\mathcal E}\max_{\alpha,\beta\in \A}\frac{h_{\alpha}}{h_{\beta}}\ \leq\ C_2,
\]
for some $C_2>0$, where $h=h(\tau)=-\Omega_{\pi}\tau$ is the height vector. It follows that if $\mathcal Z^n(\pi,\lambda,\tau)\in \{\pi\}\times\mathcal D_1\times \mathcal E$, then
\[
\max_{\alpha,\beta\in \A}\frac{h^{(n)}_{\alpha}}{h^{(n)}_{\beta}}\ \leq\ C_2.
\]

For every $\tau\in\Theta_\pi$ let $M(\tau):=\max_{\alpha\in\A}h_\alpha(\tau)$, $m(\tau):=\min_{\alpha\in\A}h_\alpha(\tau)$ and $\delta(\tau):=\frac{M(\tau)}{m(\tau)}$.
As the zippered rectangle $D_\alpha^{(n)}$ consists of $q_\alpha^{(n)}$ rectangles of heights $h_\beta$ for $\beta\in\A$, we have
\[m(\tau)\leq \frac{h^{(n)}_{\alpha}}{q^{(n)}_{\alpha}}\leq M(\tau)\quad\text{for any }\alpha\in \A.\]
It follows that if $\mathcal Z^n(\pi,\lambda,\tau)\in \{\pi\}\times\mathcal D_1\times \mathcal E$, then $\frac{q^{(n)}_{\beta}}{q^{(n)}_{\alpha}}\ \leq\ C_2\delta(\tau)$ for all $\alpha,\beta\in \A$,
and letting $\alpha=\pi_0^{-1}(d)$, we have
\[1=|I|=\sum_{\beta\in\A}q^{(n)}_\beta|I^{(n)}_\beta|\leq C_2\delta(\tau)q^{(n)}_\alpha\sum_{\beta\in\A}|I^{(n)}_\beta|=C_2\delta(\tau)q^{(n)}_\alpha|I^{(n)}|\geq \frac{8}{7}C_2\delta(\tau)q^{(n)}_\alpha|I^{(n)}_\alpha|.\]
By taking $C:=\max\{C_1,\frac{8}{7}C_2\delta(\tau)\}$, both the conditions \eqref{eq: qn3} and \eqref{eq: qn4} are satisfied.

By \cite{Ma-Mo-Yo} (see Subsection 1.2.4), for almost every $\lambda\in \mathcal{D}_1$, there exists $k=k(\pi,\lambda)\in\N$ such that KZ-matrix $Q(k)(\pi,\lambda)$ is positive (note that this matrix does not depend on $\tau$). Using the same argument for successive iterations of Zorich induction and multiplying several matrices with positive (natural) coefficients, for almost every $\lambda\in \mathcal{D}_1$ we can find $N=N(\pi,\lambda)\in\N$ such that all coefficients of the KZ-matrix $Q(N)(\pi,\lambda)$ are greater than $320C^2$.
Since Zorich induction $\mathcal Z$ and KZ-cocycle are locally continuous, there exists a subset $\mathcal D_2\subseteq \mathcal D_1$ of positive measure such that the matrix $Q(N)$ is constant on $\{\pi\}\times\mathcal D_2$ and all its coefficients are greater than $320C^2$. %
Next, by the ergodicity of $\mathcal Z$ on $\mathcal M^1$, there exists a subset $\mathcal D_3\subseteq \mathcal D_2$ of positive measure such that $\mathcal Z^{j}(\{\pi\}\times\mathcal D_3\times \Theta_\pi)$ are pairwise disjoint for $0\leq j\leq N$. Thus, if $\mathcal Z^{n_k}(\pi,\lambda,\tau),\mathcal Z^{n_{k+1}}(\pi,\lambda,\tau)\in \{\pi\}\times \mathcal D_3\times\mathcal E$ are two consecutive visits $\{\pi\}\times \mathcal D_3\times\mathcal E$, then $q^{(k_{n+1})}_\alpha\geq 320C^2q^{(k_{n})}_\alpha$ for any $\alpha\in\A$ ensuring that condition \eqref{eq: qn2,5} is met.
Summarizing, by taking
\[
\Delta^n:=I_{\pi_0^{-1}(d)}^{(k_n)}\quad\text{and}\quad q_n:=q^{(k_n)}_{\pi_0^{-1}(d)},
\]
where $k_n:=k_n(\pi,\lambda,\tau)$ is such that $k_n$ is the $n$-th visit of $(\pi,\lambda,\tau)$ in $\mathfrak L:=\{\pi\}\times\mathcal D_3\times\mathcal E$, we obtain \eqref{eq: qn2,5}-\eqref{eq: qn5} are satisfied and the proof is complete.
\end{proof}
\begin{remark}\label{rem: sym_midpoints}If we assume additionally that the permutation $\pi$ is symmetric and $\tau$ in the beginning of the proof of Proposition \ref{prop:full} is chosen so that
	\begin{equation}\label{eq:sym1}
	\sum_{\alpha\in\A}\tau_{\alpha}>0>\sum_{\alpha\in \A}\tau_{\alpha}-\tfrac{1}{2}\tau_{\pi_1^{-1}(d)}
	\end{equation}
	then by choosing $\mathcal E$ so that all elements of $\mathcal E$ satisfy this condition and by Lemma~3.12 and Remark~4 in \cite{Be-Tr-Ul} we get that $\tfrac{1}{2}$ and the mid-points $m_{\alpha}$ of the intervals $I_{\alpha}$ are forward images of the mid-points of the interval $I^{(k_n)}$ or intervals $I_{\alpha}^{(k_n)}$. It follows also from the construction in \cite{Be-Tr-Ul} that if we also assume about $\tau$ that
	\begin{equation}\label{eq:sym2}
	\sum_{\alpha\in\A}\tau_{\alpha}< \sum_{\alpha\mid\pi_0(\alpha)\le k}\tau_{\alpha}\quad\text{for}\quad k=1,\ldots,d-1
	\end{equation}
	and $r\ge 0$ is such that the $r$-th iteration of the mid-point of $I^{(k_n)}$ is either $\tfrac{1}{2}$ or $m_{\alpha}$ then $T^j(I^{(k_n)})$ is a continuity interval for $T$ for $j=0,\ldots,r$.
	\end{remark}

The presented conditions on the returns of the IET via the Zorich induction $\mathcal Z$ yield the geometric picture used to prove Proposition~\ref{prob: BC} and Proposition~\ref{thm:ergcr}. However, they give no information on the behavior of the Birkhoff sum of the considered cocycles and their derivatives. A problem of this nature was already tackled by the second author and Corinna Ulcigrai in \cite{Fr-Ul2}. There, the authors introduced two Diophantine conditions called \emph{Uniform Diophantine Condition (UDC)} and \emph{Symmetric Uniform Diophantine Condition (SUDC)} (see Definition~3 and Definition~7 in \cite{Fr-Ul2}). For the sake of this article, it is not important to recall their rather technical formal definitions. However, their consequences will be of great use. One of the crucial implications is that the UDC condition gives a sufficient condition for \eqref{eq: qn1} to be satisfied for the heights of towers obtained by considering the renormalization times, which satisfy the UDC condition.

To obtain the conditions UDC and SUDC in \cite{Fr-Ul2}, the authors considered a proper inducing of the Kontsevich-Zorich cocycle to a subset $K$ of positive measure in $\mathcal M^1$, see Theorem 3.8 and Theorem 5.6 in \cite{Fr-Ul2}. To adjust our conditions to theirs, we first consider a set $\mathfrak L$ from Proposition~\ref{prop:full} with the set of recurrence $\mathfrak D$ chosen in \cite{Ul:abs} to guarantee the control over the Birkhoff sums of the derivatives of logarithmic cocycles, see Lemma~2.1 and Lemma~2.2 in \cite{Ul:abs}. In the proof of the first of the lemmas, the author uses Egorov's Theorem to obtain the desired set, while in the second, it is again a matter of choosing a sufficiently large defining constant. In particular, the set of recurrence from \cite{Ul:abs} can be made arbitrarily large. Thus, we may assume that the set $Y$ is constructed as an intersection of $\mathfrak L$, and this recurrence set $\mathfrak D$ has a positive Lebesgue measure. Now it remains to take this set $Y$ in Lemma 3.6 in \cite{Fr-Ul2} and choose the set $K$ in the proof of Theorem 3.8 as a subset of $\mathfrak D\cap \mathfrak L$. {By taking now $\mathfrak L$ as the set which satisfies Proposition~\ref{prop:full} (and Remark~\ref{rem: sym_midpoints}), we get the following Proposition}.
\begin{proposition}\label{prop:DC}
	Let $\pi\in S^{\mathcal A}$. For almost every $T=(\pi,\lambda)\in \{\pi\}\times\Lambda^{\mathcal A}$ there exists an increasing sequence $\{q_n\}_{n\in\N}$ of natural numbers and a sequence $\{\{T^i(\Delta^n) \mid 0\leq i<q_n\}\}_{n\in\N}$ of Rokhlin towers satisfying \eqref{eq: qn1}--\eqref{eq: qn5}. Moreover, the conditions $\mathrm{UDC}$ and $\mathrm{SUDC}$ are satisfied for the sequence of renormalization times that yields $\{q_n\}_{n\in\N}$. If additionally, $\pi$ is symmetric, then the conclusion of Remark~\ref{rem: sym_midpoints} is satisfied.
\end{proposition}

All the properties implied by the UDC and SUDC conditions that we use in this article come from \cite{Fr-Ul2}. Hence, whenever we use one of them, we give a precise reference to the other paper.

\section{Ergodicity Criterion and Borel-Cantelli Argument}

In this section, we show that appropriate control of the Birkhoff sums for cocycles with singularities over large enough sets, but whose measure still decreases to 0 in a controlled manner, is enough to deduce the ergodicity of the corresponding skew products over IETs for which there exist Rokhlin towers as described in Section \ref{sc:Diophantine}.

This ergodicity criterion (see Theorem \ref{thm:ergcr}) relies on the classical notion of \emph{essential values} (see Definition \ref{def:essential_values}). Before stating it, we give here a brief overview of the tools needed to prove ergodicity. For further background material concerning skew products and infinite measure-preserving dynamical systems, we
refer the reader to \cite{Aa} and \cite{Sch}.

Let $T$ be an ergodic measure-preserving automorphism on a probability measure space $(X,\mu)$.
For any measurable map (cocycle) $\varphi: X\to \R$ denote by $T_\varphi$ the skew product extension on $X\times \R$ given by $T_{\varphi}(x,r)=(Tx,r+\varphi(x))$. Then for any $n\in\Z$ we have $T^n_{\varphi}(x,r)=(T^nx,r+S_n\varphi(x))$, where
\begin{equation}\label{def:BS}
S_n\varphi(x)=\left\{
\begin{array}{cl}
\sum_{0\leq j<n}\varphi(T^jx)&\text{if }n\geq 0,\\
-\sum_{n\leq j<0}\varphi(T^jx)&\text{if }n< 0.
\end{array}
\right.
\end{equation}
The skew product $T_\varphi$ preserves the infinite product measure $\mu\times Leb$.

Two measurable cocycles $\varphi,\psi: X\to \R$ are
called \emph{cohomologous} if there exists a measurable function
$g:X\to \R$ (called a transfer function) such that
$\varphi=\psi+g-g\circ T$. Then the corresponding skew products $T_\varphi$ and
$T_{\psi}$ are measure-theoretically isomorphic via the maps
$(x,y)\mapsto(x,y+g(x))$. A
cocycle $\varphi:X\to \R$ is a \emph{coboundary} if it is
cohomologous to the zero cocycle.

\begin{definition}
\label{def:essential_values}
A real number $r\in \R$ is said to be an \emph{essential value} of $\varphi$, if for any $\epsilon>0$
 and any measurable set $\Omega \subseteq X$ with $\mu(\Omega)>0$, there exists $n\in\Z$ such that
\begin{eqnarray}
\mu(\Omega\cap T^{-n}\Omega\cap\{x\in X:S_n\varphi(x)\in (r-\epsilon,r+\epsilon)\})>0.
\label{val-ess}
\end{eqnarray}
The set of essential values of $\varphi$ is denoted by
${E}(\varphi)$.
\end{definition}
\begin{proposition}[See \cite{Sch}]\label{proposition:schmidt}
The set ${E}(\varphi)$ is a closed subgroup of $\R$.
The skew product $T_{\varphi}$ is ergodic if and only if
$E(\varphi)=\R$.
\end{proposition}

Recall that the singularities of the cocycles considered in this work can be described through an auxiliary increasing function $\theta:[1,+\infty)\to\R_{>0}$ (see \eqref{def:Zf0}) for which certain growth properties hold (see \eqref{prop:theta tau0}).

In the following, we use a similar auxiliary function to describe sets and properties of the (derivative of) Birkhoff sums of a piecewise smooth cocycle that guarantee the ergodicity of the associated skew product over IETs as in Section \ref{sc:Diophantine}.

From now on, for any interval $J=[a,b]$ and any $0\leq \delta<1/2$ we denote by $J(\delta)$ its \emph{$\delta$-slimming}, that is, $J(\delta):=[a+\delta|J|,b-\delta|J|]$. We define the $\delta$-slimming of open and semi-open intervals analogously.

\begin{theorem}[Ergodicity criterion]\label{thm:ergcr}
Let $T:[0,1)\to[0,1)$ be an IET and assume that there exist a sequence $\{q_n\}_{n \in \N} \subseteq \N$ and a sequence of Rokhlin towers $\{\{T^i(\Delta^n) \mid 0\leq i<q_n\}\}_{n\in\N}$ satisfying \eqref{eq: qn1}-\eqref{eq: qn5}. Let $\theta:[1,+\infty)\to\R_{>0}$ be an increasing function satisfying
\begin{equation}\label{eq: divsum_theta}
	\sum_{n=1}^{\infty}\frac{1}{\theta(q_n)}=\infty.
\end{equation}
Suppose that for any $n \in \N$ there exists a family of intervals $\{\Delta_i^n\mid0\leq i<q_n,n\in\N\}$ such that
\[\Delta_i^n \subseteq T^i\Delta^n(\tfrac{1}{4})\quad\text{and}\quad\frac{D_1}{q_n\theta( q_n)}\leq|\Delta_i^n|\leq\frac{D_2}{q_n\theta( q_n)}\quad\text{for all }0\leq i<q_n,\]
for some constants $0 < D_1 < D_2$.

Assume that $f:[0,1)\to \R$ is a piecewise $C^1$ map ($C^1$ on the interior of exchanged intervals) and there exist $E>1$ and $v\in\R$ such that, for every $n\in\N$ and $0\leq i<q_n$,
\begin{equation}\label{eq: 0 of f}
	S_{q_n}(f)(\xi_i^n)\in[v-D_1(4E)^{-1},v+D_1(4E)^{-1}],\ \text{where}\ \xi_i^n\text{ is the midpoint of }\Delta_i^n,
\end{equation}
$S_{q_n}(f)$ is strictly monotonic on $\Delta^n_i$ and
\begin{equation}\label{eq: derivative of f}
E^{-1}q_n\theta( q_n)	\le |S_{q_n}(f')(x)|\le Eq_n\theta( q_n)\text{ for }x\in \Delta_i^n.
\end{equation}
Then the skew product $T_f:[0,1)\times\R\to [0,1)\times\R$ is ergodic.
\end{theorem}

The proof of this criterion relies on the following fact, which is in the spirit of the classical Borel-Cantelli Lemma and whose proof is inspired by a similar result obtained in \cite{Fa-Le}.

\begin{proposition}\label{prob: BC}
	Let $T$, $\{q_n\}_{n \in \N}$, $\{ \Delta^n\}_{n \in \N}$ and $\theta$ as in Theorem \ref{thm:ergcr}. For every $n\in\N$ let $\{ J^n_i\}_{0\leq i<q_n}$ be a family of intervals and let $0<D_1<D_2$ be constants such that
	\begin{gather}\label{eq: midposition}
	J^n_i\subseteq T^i\big(\Delta^n(\tfrac{1}{4})\big),\\
	\label{eq: size_of_J}
	\frac{D_1}{q_n\theta(q_n)}<|J^n_i|<\frac{D_2}{q_n\theta(q_n)},
	\end{gather}
 for all $0\leq i<q_n$ and $n\in\N$. Let
	\[
	A_n:=\bigcup_{0\leq i<q_n} J^n_i.
	\]
Then for any interval $(r,s) \subseteq [0,1)$ and any $k\in\N$ with $\frac{10C}{q_{k}}\leq(s-r)$ and $\theta(q_k)\geq 32 CD_2$ the following holds. For any $n\ge k$ there exists a subset $\tilde A_n \subseteq A_n$, given by the union of a subcollection of intervals in $\{ J^n_i\}_{0\leq i<q_n}$, such that
	\begin{gather}\label{eq: pairdisj}
	 \big\{ \tilde A_n,\, T^{q_n} (\tilde A_n) \,\big|\, n \geq k\big\} \text{ is a collection of pairwise disjoint subsets of $(r,s)$;}\\
	\label{eq: pos_and_disj}
	Leb\,\Big(\bigcup_{n\geq k}\tilde A_n\Big)= \frac{1}{2}(s-r).
	\end{gather}
\end{proposition}
\begin{proof}
	The idea to prove this result is to show that for every $n> k$ a large proportion of the set $A_n$ ``fills'' the holes left by the sets $A_k,\ldots,A_{n-1}$ (more precisely, by sets $\tilde A_j$ and $T^{q_j} \tilde A_j$ for $k\leq j<n$, {which are properly chosen subsets of $A_j$. Later in the proof, we provide the construction of $\tilde A_n$.}%
	
	Note that, for every $n\ge k$, the set $A_n^c$ is a union of $q_n+1$ disjoint intervals. This follows from the fact that each interval $J^n_i$ belongs to a different level of the tower $\bigcup_{0\leq i<q_n} T^i(\Delta^n)$ and is separated from its left and right endpoints. In particular, the set $\bigcap_{j=k}^{n-1} A_j^c$ consists of at most $q_{n-1}+q_{n-2}+\ldots+q_k+1$ disjoint intervals. We refer to each of these intervals as a \emph{hole}. Passing to $\bigcap_{j=k}^{n-1} \hat{A}_j^c$, where
\[\hat{A}_j:=\tilde{A}_j\cup T^{q_j}\tilde{A}_j,%
\]
we extend the holes such that their lengths are at least $\frac{10C}{q_n}$. Roughly speaking, all intervals from $A_n$ that are contained in these holes are used in the construction of $\tilde{A}_n$.

\medskip

\textbf{Preliminary arguments.} Assume that $(a,b)\subseteq[0,1)$ is a hole (interval) such that $b-a\geq \frac{10C}{q_n}$.
By \eqref{eq: qn3} and \eqref{eq: qn4}, we have that the maximal distance between consecutive intervals forming $A_{n}$ is at most $\frac{3C}{2q_{n}}$. In view of \eqref{eq: size_of_J}, the hole $(a,b)$ intersects $l$ intervals of the form $J^{n}_i$ with $l \in \N$ satisfying
\begin{equation}\label{eq: el}
b-a\leq l\frac{D_2}{q_{n}\theta(q_{n})}+(l+1)\frac{3C}{2q_{n}}.
\end{equation}
Hence
\begin{equation*}
10\leq l\frac{D_2}{C\theta(q_{n})}+(l+1)\leq l\frac{1}{32}+\frac{3}{2}(l+1), %
\end{equation*}
so $l \geq 6$. In particular, we get that $2(l-2)\geq l+1$ and thus by \eqref{eq: el} we obtain
\begin{equation}\label{eq: el2}
	2(l-2)\ge \frac{b-a}{\frac{D_2}{q_{n}\theta(q_{n})}+\frac{3C}{2q_{n}}}.
\end{equation}

Denote by $\mathcal I^n_{(a,b)}$ the set of indices $0\leq i<q_n$ for which $J^{n}_i\cap (a,b)\neq \emptyset$ after throwing out the indices corresponding to the leftmost and rightmost intervals in $A_n$ intersecting $(a, b)$. Then $\# \mathcal I^n_{(a,b)} = l - 2\geq 4$ and
\[(a,b)\setminus \bigcup_{i\in \mathcal I^n_{(a,b)}}J^n_i\]
consists of $l-1$ holes. In view of \eqref{eq: midposition}%
, the length of every such hole is at least $\frac{1}{2}|\Delta^n|$ %
and hence, since by \eqref{eq: qn5} $T^{q_n}$ translates the interval $J^n_i$ by at most $\frac{1}{8}|\Delta^n|$, we obtain
\begin{equation}\label{eq: shiftstays}
	T^{q_n}J^n_i \subseteq (a,b), \qquad\text{for every }i\in\mathcal{I}^n_{(a,b)}.
\end{equation}
Note that, by \eqref{eq: qn5}, we have that
	\begin{equation}\label{eq: intervals_are_disjoint}
	\{ J_i^n, T^{q_n}J_i^n \mid 0\leq i<q_n\} \quad \text{is a collection of pairwise disjoint sets.}
	\end{equation}
	Indeed, as $T^{q_n}\Delta^n(\tfrac{1}{4}) \subseteq \Delta^n(\tfrac{1}{8}) \subseteq \Delta^n$, for every $0\leq i<q_n$, we have $J_i^n,\,T^{q_n}J_i^n \subseteq T^{i}\Delta^n$ and for $i_1\neq i_2$ the intervals $T^{i_1}\Delta^n$ and $T^{i_2}\Delta^n$ are two different levels of the Rokhlin tower.
It follows that $J^n_{i_1}$ and $T^{q_n}J^n_{i_1}$ are disjoint from $J^n_{i_2}$ and $T^{q_n}J^n_{i_2}$ and each pair of intervals is distant by at least $\frac{1}{4}|\Delta^n|$.
Moreover, by \eqref{eq: qn5}, $T^{q_n}J_i^n$ is a translation of $J_i^n$ by at least $\frac{1}{16}|\Delta^n|$. In view of \eqref{eq: size_of_J} and \eqref{eq: qn3},
$|J^n_i|\leq \frac{D_2}{q_n\theta( q_n)}\leq \frac{1}{32}|\Delta^n|$.
It follows that $J_i^n$ and $T^{q_n}J_i^n$ are also disjoint and they are distant by at least $\frac{1}{32}|\Delta^n|$.
In summary,
\begin{align}\label{eq: ab holes1}
(a,b)\setminus \bigcup_{i\in \mathcal I^n_{(a,b)}}(J^n_i\cup T^{q_n}J^n_i)\text{ consists of $2(l-2)+1$ holes of length at least }\frac{1}{32}|\Delta^n|.
\end{align}
Moreover, by \eqref{eq: qn2,5} and \eqref{eq: qn3},
\begin{equation}\label{eq: ab holes2}
\frac{1}{32}|\Delta^n|\geq \frac{1}{32Cq_n}\geq \frac{10C}{q_{n+1}}.
\end{equation}
On the other hand, by \eqref{eq: size_of_J} and \eqref{eq: el2},
	\begin{align}\label{eq: measure control ab}
\begin{split}
	Leb\Big(\bigcup_{i\in \mathcal I^n_{(a,b)}}J^n_i\Big)
=Leb\Big(\bigcup_{i\in \mathcal I^n_{(a,b)}}T^{q_n}J^n_i\Big)
&\ge (l-2)\cdot\frac{D_1}{q_{n}\theta(q_{n})}\geq \frac{b-a}{\frac{2D_2}{q_{n}\theta (q_{n})}+\frac{3C}{q_{n}}}\cdot\frac{D_1}{q_{n}\theta(q_{n})}\\
&\geq \frac{D_1(b-a)}{(3C+2D_2)\theta( q_{n})}=\frac{c}{\theta( q_n)}(b-a),
\end{split}
	\end{align}
with $c:=D_1/(3C+2D_2)$.

\medskip

\textbf{Inductive construction.}
Fix any interval $(r,s) \subseteq [0,1)$ and $k\in \N$ such that $\frac{10C}{q_{k}}\leq(s-r)$ and $\theta( q_k)\geq 32 CD_2$.
We will define inductively a sequence $\{\tilde{A}_n\}_{n\geq k}$ of subsets of $(r,s)$ such that, for every $n\ge k$, the set $\tilde A_n \subseteq A_n$ is the union of intervals $J_i^n$ with $0\leq i<q_n$ and
\begin{gather}
\label{eq: holes sepration}
\text{all holes of }(r,s)\,\setminus\bigcup_{k\leq j\leq n}(\tilde A_j\cup T^{q_j}\tilde A_j)\text{ have length at least }\frac{10C}{q_{n+1}},\\
\label{eq: measure control}
Leb(\hat A_n)\geq \frac{2c}{\theta( q_n)}Leb\Big((r,s)\setminus\bigcup_{k\leq j< n}\hat A_j\Big), \qquad
\text{ where }\hat A_j=\tilde A_j\cup T^{q_j}\tilde A_j.
\end{gather}
Moreover, we define this sequence so that the collection $\{ \tilde A_n$, $T^{q_n} \tilde A_n \mid n \geq k\}$ consists of pairwise disjoint subsets of $(r,s)$.

\medskip

\textbf{First step.} Let us define the set $\tilde A_k$. As $ s-r\geq \frac{10C}{q_k}$, using the preliminary arguments at the beginning of the proof, we can define $\tilde A_k$ as
\[\tilde A_k:=\bigcup_{i\in \mathcal I^k_{(r,s)}}J^k_i.\]
By \eqref{eq: intervals_are_disjoint}, the sets $\tilde A_k$ and $T^{q_k}\tilde A_k$ are disjoint and, by \eqref{eq: shiftstays}, $T^{q_k} \tilde A_k$ is a subset of $(r,s)$.
Moreover, \eqref{eq: holes sepration} follows directly from \eqref{eq: ab holes1} and \eqref{eq: ab holes2}, and \eqref{eq: measure control} from \eqref{eq: measure control ab}.

\medskip
\textbf{Inductive step.}
Suppose that the sets $\tilde A_k,\ldots, \tilde A_n$ are already defined, and let us define the set $\tilde A_{n + 1}$. Denote by $\mathcal H_n$ the set of all holes in $(r,s)\setminus\bigcup_{k\leq j\leq n}(\tilde A_j\cup T^{q_j}\tilde A_j)$. By assumption, for any $(a,b)\in \mathcal H_n$ we have $(b-a)\geq \frac{10C}{q_{n+1}}$. Let
\[\tilde A_{n+1}:=\bigcup_{(a,b)\in \mathcal H_n}\bigcup_{i\in \mathcal I^{n+1}_{(a,b)}}J^{n+1}_i.\]
By \eqref{eq: intervals_are_disjoint}, the sets $\tilde A_{n+1}$ and $T^{q_{n+1}}\tilde A_{n+1}$ are disjoint and by \eqref{eq: shiftstays}, we have that $T^{q_{n+1}}(\tilde A_{n+1}\cap (a,b)) $ is a subset of $(a,b)$ for any $(a,b)\in\mathcal{H}_n$.
That is, we have
\[\tilde A_{n+1},\ T^{q_{n+1}}\tilde A_{n+1} \subseteq \bigcup_{(a,b)\in \mathcal H_n}(a,b)=(r,s)\setminus\bigcup_{k\leq j\leq n}(\tilde A_j\cup T^{q_j}\tilde A_j),\]
and hence, we also have the disjointness of $\tilde A_{n+1}$ and $T^{q_{n+1}}\tilde A_{n+1}$ from $\tilde A_{j}$ and $T^{q_{j}}\tilde A_{j}$ for all $k\leq j\leq n$. Moreover, \eqref{eq: holes sepration} follows directly from \eqref{eq: ab holes1} and \eqref{eq: ab holes2}.
In view of \eqref{eq: measure control ab} for any $(a,b)\in \mathcal H_n$ we have
\[Leb\Big(\bigcup_{i\in \mathcal I^{n+1}_{(a,b)}}(J^{n+1}_i\cup T^{q_{n+1}}J^{n+1}_i)\Big)\geq\frac{2c}{\theta( q_{n+1})}(b-a).\]
Summing up along $(a,b)\in \mathcal H_n$, we obtain
\[Leb(\hat A_{n+1})\geq \frac{2c}{\theta( q_{n+1})}\bigcup_{(a,b)\in \mathcal H_n}(b-a)=\frac{2c}{\theta (q_{n+1})}Leb\Big((r,s)\setminus\bigcup_{k\leq j\leq n}\hat A_j\Big),\]
so \eqref{eq: measure control} holds.

\medskip
\textbf{Final arguments.} Using the construction above, we define a collection $\{ \tilde A_n$, $T^{q_n} \tilde A_n \mid n \geq k\}$ of pairwise disjoint subsets of $(r,s)$ satisfying \eqref{eq: holes sepration} and \eqref{eq: measure control}. Let $n \geq k$. Since $\hat A_{n} \subseteq (r,s)$ and is disjoint from $\hat A_{k},\ldots,\hat A_{n-1}$, by \eqref{eq: measure control}, we have
\[
Leb\Big((r,s)\setminus\bigcup_{k\leq j\leq n}\hat A_j\Big)= Leb\Big((r,s)\setminus\bigcup_{k\leq j< n}\hat A_j\Big)-Leb(\hat A_{n})
\leq \left(1-\frac{2c}{\theta( q_{n})}\right)Leb\Big((r,s)\setminus\bigcup_{k\leq j< n}\hat A_j\Big).
\]
It follows that
\[Leb\Big((r,s)\setminus\bigcup_{k\leq j\leq n}\hat A_j\Big)\leq (s-r)\prod_{k\leq j\leq n}\left(1-\frac{2c}{\theta( q_{j})}\right).\]
As $\sum_{n\geq k}\frac{1}{\theta( q_n)}=+\infty$, this gives
\[Leb\Big((r,s)\setminus\bigcup_{n\geq k}\hat A_n\Big)=\lim_{n\to\infty}Leb\Big((r,s)\setminus\bigcup_{k\leq j\leq n}\hat A_j\Big)=0.\]
Therefore,
\[2Leb\Big(\bigcup_{n\geq k}\tilde A_n\Big)=Leb\Big(\bigcup_{n\geq k}\tilde A_n\Big)+Leb\Big(\bigcup_{n\geq k}T^{q_n}\tilde A_n\Big)=Leb\Big(\bigcup_{n\geq k}\hat A_n\Big)=(s-r),\]
which completes the proof.
\end{proof}

\begin{proof}[Proof of Theorem \ref{thm:ergcr}]
In view of \eqref{eq: qn3}, the IET $T$ is ergodic.%

We will show that every number $a\in(v-D_1(4E)^{-1},v+D_1(4E)^{-1})$ is an essential value of $T_f$. Notice that, by Proposition \ref{proposition:schmidt}, this is enough to conclude the ergodicity of $T_f$.
	
	Fix any $a\in(v-D_1(4E)^{-1},v+D_1(4E)^{-1})$ and $0<\epsilon<D_1(4E)^{-1}-|v-a|$. Note that
\[(a-\epsilon,a+\epsilon)\subseteq[v-D_1(4E)^{-1},v+D_1(4E)^{-1}] \subseteq S_{q_n}(f)(\Delta_i^n),\]
for all $n\in\N$ and $0\leq i<q_n$. Let $J_i^n$ be a subinterval of $\Delta_i^n$ such that $S_{q_n}(f)(J_i^n)= [a-\epsilon,a+\epsilon]$, for any $n\in\N$ and $0\leq i<q_n$. By \eqref{eq: derivative of f}, each of these intervals satisfies $\frac{2\epsilon E^{-1}}{q_n\theta( q_n)}\le |J_i^n|\le \frac{2\epsilon E}{q_n\theta( q_n)}$. In particular, the family of intervals $J_i^n$ for $n\in\N$ and $0\leq i<q_n$ satisfies the assumptions of Proposition~\ref{prob: BC}. %
	
	Let $\Omega\subseteq[0,1)$ be any measurable set of positive Lebesgue measure. By the Lebesgue Density Theorem, there exists an interval $(r,s)\subseteq[0,1)$ such that
	\begin{equation}\label{eq: bigOmega}
	Leb(\Omega\cap (r,s))>\frac{9}{10}(s-r).
	\end{equation}
By Proposition \ref{prob: BC}, there exists $k\in\N$ such that for every $n\geq k$ there exists a set $\tilde A_n$ consisting of some intervals of the form $J_i^n$ such that $\big\{ \tilde A_n,\, T^{q_n} (\tilde A_n) \,\big|\, n \geq k\big\}$ is a collection of pairwise disjoint subsets of $(r,s)$ and
	\begin{equation}\label{eq: tildeAn_are_big}
	Leb\Big(\bigcup_{n\geq k}\tilde A_n\Big)=\frac{1}{2}(s-r).
	\end{equation}
	We claim that there exists $n\geq k$ and $0\leq i<q_n$ such that
	\begin{equation}\label{eq: verifyEV}
	Leb\left(\left\{x\in J_i^n \mid x, T^{q_n}x\in \Omega\cap(r,s)\right\}\right)>0.
	\end{equation}
	Before proving this, notice that this implies that $a$ is an essential value. Indeed, since we chose the interval $J_i^n$ so that $S_{q_n}f$ restricted to the interior of $J_i^n$ attains values in $(a-\epsilon,a+\epsilon)$, this would verify the condition required to check that $a$ is an essential value.
	
Therefore, it suffices to show that \eqref{eq: verifyEV} holds for some $n\geq k$ and $0\leq i<q_n$. For the sake of contradiction, let us assume that \eqref{eq: verifyEV} is never satisfied. Then for any $n\geq k$, any $0\leq i<q_n$, and for a.e.\ $x\in J_i^n$ we have $x\notin \Omega\cap(r,s)$ or $T^{q_n}x\notin \Omega\cap(r,s)$, so
\begin{align*}
Leb(J_i^n)&\leq Leb(J^n_i\setminus (\Omega\cap(r,s)))+Leb(J^n_i\setminus T^{-q_n}(\Omega\cap(r,s)))\\
&=Leb(J^n_i\setminus (\Omega\cap(r,s)))+Leb(T^{q_n}J^n_i\setminus (\Omega\cap(r,s))) \\
& =Leb((J^n_i\cup T^{q_n}J^n_i)\setminus (\Omega\cap(r,s))).
\end{align*}
For any $n\geq k$, by summing up over all $0\leq i<q_n$ with $J^n_i \subseteq \tilde{A}_n$, we get
	\[
	Leb((\tilde A_n\cup T^{q_n} \tilde A_n)\setminus (\Omega\cap(r,s)))\geq Leb(\tilde A_n).
	\]
	Then, by summing up over $n\ge k$ and using \eqref{eq: tildeAn_are_big}, we get
	\[
	Leb\left((r,s)\setminus (\Omega\cap (r,s))\right)\geq Leb\Big(\bigcup_{n\geq k}(\tilde A_n\cup T^{q_n} \tilde A_n )\setminus(\Omega\cap(r,s))\Big)
\geq Leb\Big(\bigcup_{n\geq k}\tilde A_n\Big)=\frac{1}{2}(s-r).
	\]
	Hence
	\[
	Leb(\Omega\cap(r,s))\le \frac{1}{2}(s-r),
	\]
	which is a contradiction with \eqref{eq: bigOmega}.
\end{proof}

\section{Antisymmetric skew products over symmetric IETs}

Let $T=(\pi, \lambda):[0, 1) \to [0, 1)$ be a symmetric IET, i.e.\ $\pi_0(\alpha)+\pi_1(\alpha)=d+1$ for any $\alpha\in\mathcal{A}$. Denote by $\I:[0,1]\to[0,1]$ the reflection across the midpoint $m_I:=\tfrac{1}{2}$, i.e., $\I(x)=1-x$. Then $\I \circ T=T^{-1}\circ \I$ and this map acts on every interval $I_\alpha$ as the reflection across the midpoint $m_\alpha$ of the exchanged interval $I_\alpha$, for any $\alpha \in \A$. Note that if the permutation $\pi$ is irreducible and \emph{non-degenerated}, that is, there are no two adjacent intervals translated in parallel by $T$, and the length of exchanged intervals are rationally independent (in fact, it is enough to know that they are different) then the condition $\I \circ T=T^{-1}\circ \I$ yields that $T$ is symmetric.

Let $f:[0,1)\to\R$ be a piecewise continuous function with a finite number of discontinuities. We say that $f$ is \emph{piecewise monotonic} if it is increasing (resp.\ always decreasing) when restricted to the interior of each interval of continuity. The map $f$ is said to be \emph{anti-symmetric with respect to $T$} if it is continuous on the interior of the intervals exchanged by $T$ and satisfies $f\circ T^{-1}\circ \I=-f$, which means that the restriction of $f$ to any of the intervals exchanged by $T$ is anti-symmetric with respect to the central reflection on the exchanged interval. If there is no risk of confusion, we will sometimes omit the dependence on $T$ in the previous definition and simply say that $f$ is anti-symmetric.

\begin{lemma}\label{lem:anti}
Let $T=(\pi, \lambda): I \to I$ be a symmetric IET and $f:I\to\R$ be piecewise monotonic and anti-symmetric. Then
{\begin{equation}\label{eq:antisymSqf}
S_nf(x)=-S_nf(\I\circ T^{n} x), \quad\text{for any }  x\in I\text{ and any }n\in\N.
\end{equation}}

Let $\alpha\in\mathcal{A}\cup\{I\}$ and $0 < \epsilon < \tfrac{|I|}{2}$. Suppose that $\Delta_\alpha=[m_\alpha-\vep,m_\alpha+\vep]$ and $q\in\N$ are such that $\{T^i(\Delta_\alpha) \mid 0\leq i<q\}$ is a Rokhlin tower of intervals included in the exchanged intervals of $T$ and
\[T^qm_\alpha=m_\alpha+\delta, \qquad \text{ with }0 < |\delta|<\vep/3.\]
Then, for every $-q<i<q$ the map $S_qf$ is continuous on $T^i[m_\alpha-2\vep/3,m_\alpha+2\vep/3]$ and there exists $\xi_i\in T^i[m_\alpha-\vep/2,m_\alpha+\vep/2]$ such that
\[ S_qf(\xi_i) = \left\{ \begin{array}{lcl} f(m_\alpha-\delta/2) & \text{ if } & \alpha\in\mathcal{A}, \\
0 & \text{ if } & \alpha = I. \end{array} \right.\]

\end{lemma}

\begin{proof}
{
 By the anti-symmetricity of $f$, for any $x\in[0,1)$ and any $n \in \N$, we have
\begin{align*}\label{eq:antif}
\begin{aligned}
S_nf(x)&=\sum_{0\leq j<n}f(T^jx)=-\sum_{0\leq j< n}f(T^{-1}\circ \I\circ T^jx)\\
&=-\sum_{0\leq j< n}f( T^{n-j-1}\circ T^{-n}\circ \I x)=-S_nf(T^{-n}\circ \I x)=-S_nf(\I\circ T^{n} x),
\end{aligned}
\end{align*}
which gives \eqref{eq:antisymSqf}.}

Let $-q<i<q$. Notice that $T^{\pm q}(m_\alpha) = m_\alpha \pm \delta$ and, more generally, $$T^{\pm q}(x) = x \pm \delta, \qquad \text{for any }x\in T^i[m_\alpha-2\vep/3,m_\alpha+2\vep/3].$$ Hence $T^{\pm q}[m_\alpha-2\vep/3,m_\alpha+2\vep/3] \subseteq \Delta_\alpha$, and the continuity of $S_qf$ on $T^i[m_\alpha-2\vep/3,m_\alpha+2\vep/3]$ follows directly from our assumptions.

Without loss of generality of reasoning, let us assume that $f$ is piecewise increasing and that $0< \delta<\vep/3$. In all other cases, the reasoning goes the same way. We now consider the cases $\alpha \in \A$ and $\alpha = I$ separately.

\textbf{Case 1: $\alpha \in \A$.} %
By the anti-symmetricity of $f$, for any $x\in[0,1)$, we have
\begin{align*}
S_qf(x)&=\sum_{0\leq j<q}f(T^jx)=f(x)-f(T^qx)+\sum_{0< j\leq q}f(T^jx)\\
&=f(x)-f(T^qx)-\sum_{0< j\leq q}f(T^{-1}\circ \I \circ T^jx)\\
&=f(x)-f(T^qx)-\sum_{0< j\leq q}f( T^{-j-1}\circ \I x)\\
&=f(x)-f(T^qx)-S_qf(T^{-q}\circ T^{-1}\circ \I x).
\end{align*}

	As $T^q(m_\alpha-\delta/2)=m_\alpha+\delta/2$ and $T^{-1}\circ \I(m_\alpha-\delta/{2})=m_\alpha+\delta/{2}$, it follows that
\begin{align*}
S_qf(m_\alpha-\tfrac{\delta}{2})&=f(m_\alpha-\tfrac{\delta}{2})-f(T^q(m_\alpha-\tfrac{\delta}{2}))-S_qf(T^{-q}\circ T^{-1}\circ \I(m_\alpha-\tfrac{\delta}{2}))\\
&=f(m_\alpha-\tfrac{\delta}{2})-f(m_\alpha+\tfrac{\delta}{2})-S_qf(T^{-q}(m_\alpha+\tfrac{\delta}{2}))\\
&=2f(m_\alpha-\tfrac{\delta}{2})-S_qf(m_\alpha-\tfrac{\delta}{2}).
\end{align*}
Therefore, $S_qf(m_\alpha-\tfrac{\delta}{2})=f(m_\alpha-\tfrac{\delta}{2})$.
Note that, if $0\leq i\leq q$, for any $x\in [m_\alpha-3\vep/4,m_\alpha+3\vep/4]$ we have
\begin{align*}
S_qf(T^ix)=\sum_{i\leq j<q}f(T^jx)+\sum_{0\leq j<i}f(T^q T^jx)=\sum_{i\leq j<q}f(T^jx)+\sum_{0\leq j<i}f(T^jx+\delta).
\end{align*}
As $f$ is piecewise increasing, this yields
\[S_qf(x)\leq S_qf(T^ix)\leq S_qf(x+\delta).\]
It follows that
\[S_qf(T^i(m_\alpha-\tfrac{\delta}{2}))\geq S_qf(m_\alpha-\tfrac{\delta}{2})=f(m_\alpha-\tfrac{\delta}{2})\]
and
\[S_qf(T^i(m_\alpha-\tfrac{3\delta}{2}))\leq S_qf(m_\alpha-\tfrac{\delta}{2})=f(m_\alpha-\tfrac{\delta}{2}).\]
By continuity, there exists
\[\xi_i\in T^i[m_\alpha-\tfrac{3\delta}{2},m_\alpha-\tfrac{\delta}{2}]\ \text{ such that }\ S_qf(\xi_i)=f(m_\alpha-\tfrac{\delta}{2}).\]
If $-q\leq i\leq 0$, then we take $\xi_i:=\xi_{q+i}$. Then
\[\xi_i=\xi_{q+i}\in T^{q+i}[m_\alpha-\tfrac{3\delta}{2},m_\alpha-\tfrac{\delta}{2}]=T^{i}[m_\alpha-\tfrac{\delta}{2},m_\alpha+\tfrac{\delta}{2}].\]
Finally, we have $\xi_i\in T^{i}[m_\alpha-\vep/2,m_\alpha+\vep/2]$ in both cases. \medskip

\noindent
\textbf{Case 2: $\alpha = I$.}  As $T^{-q}(m_I-\delta/2)=m_I+\delta/2$ and $\I(m_I+\delta/2)=m_I-\delta/2$, {by \eqref{eq:antisymSqf}} it follows that
\begin{align*}
S_qf(m_I-\tfrac{\delta}{2})=-S_qf(m_I-\tfrac{\delta}{2}),\quad \text{so}\quad S_qf(m_I-\tfrac{\delta}{2})=0.
\end{align*}
As $f$ is piecewise increasing, if $0 < i \leq q$, the arguments used in the first case show that
\[S_qf(x)\leq S_qf(T^ix)\leq S_qf(x+\delta), \qquad \text{ for any } x\in [m_I-2\vep/3,m_I+2\vep/3].\]
It follows that
\[S_qf(T^i(m_I-\tfrac{\delta}{2}))\geq S_qf(m_I-\tfrac{\delta}{2})=0\quad \text{ and }\quad S_qf(T^i(m_I-\tfrac{3\delta}{2}))\leq S_qf(m_I-\tfrac{\delta}{2})=0.\]
By continuity, there exists
\[\xi_i\in T^i[m_I-\tfrac{3\delta}{2},m_I-\tfrac{\delta}{2}]\ \text{ such that }\ S_qf(\xi_i)=0.\]
If $-q\leq i\leq 0$, then we take again $\xi_i:=\xi_{q+i}$.
\end{proof}

For any integrable map $f:I\to\R$ and any subinterval $J \subseteq I$ let $m(f,J)$ be the mean value of $f$ on $J$, that is, $m(f,J):=\frac{1}{|J|}\int_J f(x)\,dx$. For any IET $T:I\to I$ denote by $T_J:J\to J$ the induced map and by $f_J:J\to\R$ the induced cocycle, that is, $T_J(x)=T^{r_J(x)}(x)$ and $f_J(x)=S_{r_J(x)}f(x)$, for any $x \in J$, where $r_J:J\to\N$ denotes the first time return map.

Recall that for any interval $J=[a,b]$ and $0\leq \delta<1/2$ we let $J(\delta)=[a+\delta|J|,b-\delta|J|]$.

\begin{lemma}\label{lem:ind}
Let $T= (\pi,\lambda):I\to I$ be a symmetric IET whose interval lengths are rationally independent, and let $f:I\to\R$ be an anti-symmetric integrable map. Then for any $0\leq \delta<1/2$, we have $m(f,I_\alpha(\delta))=0$, for every $\alpha\in\mathcal{A}$.

Let $J \subseteq I$ be a symmetric interval, i.e., $\I(J)=J$, such that the induced IET $T_J$ exchanges the same number of intervals as $T$. Then $T_J$ is a symmetric IET, and $f_J:J\to\R$ is an antisymmetric cocycle with respect to $T_J$. In particular, $m(f_J,J_\alpha(\delta))=0$ for any $\alpha\in\mathcal{A}$ and any $0\leq \delta<1/2$.
\end{lemma}

\begin{proof}
Note that
\begin{align*}
\int_{I_\alpha(\delta)}f(x)\,dx&=\int_{l_\alpha+\delta|I_\alpha|}^{m_\alpha}f(x)\,dx+\int_{m_\alpha}^{r_\alpha-\delta|I_\alpha|}f(x)\,dx\\
&=\int_{l_\alpha+\delta|I_\alpha|}^{m_\alpha}f(x)\,dx-\int_{m_\alpha}^{r_\alpha-\delta|I_\alpha|}f(\I\circ T^{-1}x)\,dx\\
&=\int_{l_\alpha+\delta|I_\alpha|}^{m_\alpha}f(x)\,dx-\int_{l_\alpha+\delta|I_\alpha|}^{m_\alpha}f(x)\,dx=0.
\end{align*}

As $\I\circ T=T^{-1}\circ \I$, we have $\I\circ T^n=T^{-n}\circ \I$ for any $n\in\Z$. Hence, for any $x\in J$, we have
$T^{r_J(x)}\circ \I\circ T_Jx=\I x\in J$. It follows that $r_J(\I(T_Jx))=r_J(x)$ and $\I\circ T_J=T_J^{-1}\circ \I$. By assumption, the permutation of $T_J$ belongs to the Rauzy class of $\pi$, so it is irreducible and non-degenerated.
Moreover, since $\lambda$ is rationally independent, the lengths of the intervals exchanged by $T_J$ are different. Therefore, $T_J$ is also symmetric. Moreover, by \eqref{eq:antisymSqf},
\[f_J(\I(T_Jx))=S_{r_J(\I(T_Jx))}f(\I(T_Jx))=S_{r_J(x)}f(\I(T^{r_J(x)}x))=-S_{r_J(x)}f(x)=-f_J(x).\]
By the first part of the lemma, this gives $m(f_J, J_\alpha(\delta))=0$ for every $\alpha\in\mathcal{A}$.
\end{proof}

Suppose that
\begin{align}\label{prop:theta tau}
\begin{aligned}
\text{$\theta:[1,+\infty)\to\R_{>0}$ is an increasing $C^1$-map such that $\int_{1}^{+\infty}\frac{dx}{x\theta(x)}=+\infty$,} \\
\text{ and the map $\tau:(0,1]\to\R_{>0}$ given by $\tau(s)=\frac{s^2}{\theta'(1/s)}$ is increasing.}\quad
\end{aligned}
\end{align}
Recall that $\Upsilon_\theta(\bigsqcup_{\alpha\in \mathcal{A}}
I_{\alpha})$ is the space of functions $f:[0,1)\to\R$ which are $C^1$ on the interior of the intervals $I_{\alpha}$, $\alpha\in \mathcal{A}$, and
\begin{equation}\label{def:Zf}
Z_\theta(f):=\max_{\alpha\in\mathcal{A}}\left\{\sup_{x\in(l_\alpha,m_\alpha]}|f'(x)\tau(x-l_\alpha)|,\sup_{x\in[m_\alpha,r_\alpha)}|f'(x)\tau(r_\alpha-x)|\right\}<+\infty.
\end{equation}

The following proposition shows how to verify the assumption \eqref{eq: derivative of f} in Theorem~\ref{thm:ergcr} for functions in $\Upsilon_\theta(\bigsqcup_{\alpha\in \mathcal{A}}
I_{\alpha})$ with $\theta$ as in \eqref{prop:theta tau}.

\begin{proposition}\label{prop:controlder}
Let $T:[0,1)\to[0,1)$ be an IET and assume that there exist a sequence $\{q_n\}_{n \in \N} \subseteq \N$ and a sequence of Rokhlin towers $\{\{T^i(\Delta^n) \mid 0\leq i<q_n\}\}_{n\in\N}$ satisfying \eqref{eq: qn3}-\eqref{eq: qn5}. Let $\theta$ and $\tau$ as in \eqref{prop:theta tau}.
Then, for any $f\in \Upsilon_\theta(\bigsqcup_{\alpha\in \mathcal{A}}I_{\alpha})$, there exists $E>1$ such that
\begin{align}\label{eq: Sqleq}
|S_{q_n}(f')(x)|\le Eq_n\theta(E q_n)\qquad \text{ for any }x\in T^i\Delta^n(\tfrac{1}{4})\text{ and any }0\leq i<q_n.
\end{align}
If additionally $f$ is piecewise monotonic and
\begin{equation}\label{def:zf}
z_\theta(f):=\max_{\alpha\in\mathcal{A}}\left\{\inf_{x\in(l_{\alpha},m_{\alpha}]}|f'(x)\tau(x-l_{\alpha})|,
\inf_{x\in[m_{\alpha},r_{\alpha})}|f'(x)\tau(r_{\alpha}-x)|\right\}>0,
\end{equation}
then
\begin{align}\label{eq: Sqgeq}
|S_{q_n}(f')(x)|\ge E^{-1}q_n\theta(E^{-1} q_n)\qquad \text{ for any }x\in T^i\Delta^n(\tfrac{1}{4})\text{ and any }0\leq i<q_n.
\end{align}
\end{proposition}

\begin{proof}
By \eqref{eq: qn3} and \eqref{eq: qn4}, there exists $\bar{C}>1$ such that for any $x\in T^i\Delta^n(\tfrac{1}{4})$ with $0\leq i<q_n$, we have that
the distance between two consecutive elements of the set
\[\{T^jx:0\leq j<q_n\}\cup End(T),\]
where $End(T)$ denotes the set of endpoints of the exchanged intervals of $T$, is always between $\bar{C}^{-1}/q_n$ and $\bar{C}/q_n$. As $1/\tau$ is decreasing, this observation together with \eqref{prop:theta tau} yield
\[|S_{q_n}f'(x)|\leq \sum_{0\leq j<q_n}|f'(T^jx)|\leq 2dZ_\theta(f)\sum_{1\leq j\leq q_n}\frac{1}{\tau(\frac{1}{\bar{C}q_n}j)}\leq 4dZ_\theta(f)\bar{C}q_n\int_{\frac{1}{2\bar{C}q_n}}^1\frac{ds}{\tau(s)}.\]
Moreover, by the definition of $\tau$,
\[\int_{\frac{1}{2\bar{C}q_n}}^1\frac{ds}{\tau(s)}=\int_{\frac{1}{2\bar{C}q_n}}^1\frac{\theta'(1/s)ds}{s^2}=
-\int_{\frac{1}{2\bar{C}q_n}}^1\frac{d}{ds}\theta(1/s)ds=\theta(2\bar{C}q_n)-\theta(1)\leq \theta(2\bar{C}q_n),\]
which gives \eqref{eq: Sqleq}.

Suppose now that $f$ is piecewise monotonic and that $z_\theta(f)$ given by \eqref{def:zf} is positive. Let us assume without loss of generality that $f$ is increasing and that, for some $\alpha_0 \in \A$, $z_\theta(f)=\inf_{x\in(l_{\alpha_0},m_{\alpha_0}]}f'(x)\tau(x-l_{\alpha_0})>0$. In all other cases, the arguments go in the same way.

Then, for any $x\in T^i\Delta^n(\tfrac{1}{4})$ with $0\leq i<q_n$, we have
\begin{align*}
|S_{q_n}f'(x)|&=\sum_{0\leq j<q_n}f'(T^jx)\geq \sum_{\substack{0\leq j<q_n\\T^jx\in(l_{\alpha_0},m_{\alpha_0}]}}f'(T^jx)
\geq z_\theta(f)\sum_{1\leq j\leq \frac{\lambda_{\alpha_0}}{2\bar{C}}q_n}\frac{1}{\tau(\frac{\bar{C}}{q_n}j)}\\
&\geq \frac{z_\theta(f)}{\bar{C}}q_n\int_{\frac{\bar{C}}{q_n}}^{{\lambda_{\alpha_0}}}\frac{ds}{\tau(s)}=
\frac{z_\theta(f)}{\bar{C}}q_n\left(\theta(\tfrac{q_n}{\bar{C}})-\theta(\tfrac{1}{\lambda_{\alpha_0}})\right),
\end{align*}
which gives \eqref{eq: Sqgeq}.
\end{proof}

We say that an increasing function $\theta:[x_0,+\infty)\to\R_{>0}$, with $x_0 > 0$, is \emph{slowly varying} if for any $\eta>0$ there exists $C_\eta>1$ such that
\[C^{-1}_\eta\leq\frac{\theta(\eta x)}{\theta(x)}\leq C_\eta, \qquad\text{for any } x\geq x_0.\]

Notice that if $\theta:[1,+\infty)\to\R_{>0}$ in Proposition \ref{prop:controlder} is slowly varying, then \eqref{eq: Sqleq} and \eqref{eq: Sqgeq} imply \eqref{eq: derivative of f}.

\begin{proof}[Proof of Theorem~\ref{thm:anti}]
By Theorem~\ref{prop:full}, for a.e.\ $T$ there exists an increasing sequence $\{q_n\}_{n\in\N}$ of natural numbers and a sequence $\{\{T^i(\Delta^n) \mid 0\leq i<q_n\}\}_{n\in\N}$ of Rokhlin towers satisfying \eqref{eq: qn1}--\eqref{eq: qn5}. Moreover, it follows from the proof of Theorem \ref{prop:full} that we can take $\Delta^n=I_{\pi_0^{-1}(d)}^{k_n}$ and $q_n=q^{k_n}_{\pi_0^{-1}(d)}$ for some increasing sequence $\{k_n\}_{n\in\N}$.
In view of \cite[Lemma~3.12]{Be-Tr-Ul}, there exists $\alpha\in\mathcal{A}\cup\{I\}$ and $0\leq s_n< q_n$ such that $T^{-s_n}m_{\alpha}=m_n$ is the center of $I_{\pi_0^{-1}(d)}^{k_n}$.

As $T^{q_n}m_n=m_n+\delta_n$, for some $|\Delta^n|/16<|\delta_n|<|\Delta^n|/8$, the interval {$[m_\alpha-\vep_n,m_\alpha+\vep_n]=T^{s_n}[m_m-\vep_n,m_n+\vep_n]$} with $\vep_n=\tfrac{3}{8}|\Delta^n|$ satisfies the assumptions of Lemma~\ref{lem:anti}.
Hence for any $0\leq i<q_n$ there exists
{\[\xi^n_i\in T^{i-s_n}[m_\alpha-\vep_n,m_\alpha+\vep_n]=T^i[m_n-\vep_n/2,m_n+\vep_n/2]=T^i\Delta^n(\tfrac{5}{16})\]} such that
\[ S_{q_n}f(\xi^n_i) = v_n := \left\{ \begin{array}{lcl} f(m_\alpha-\delta_n/2) & \text{ if } & \alpha\in\mathcal{A}, \\
0 & \text{ if } & \alpha = I. \end{array} \right.\]
By the anti-symmetricity of $f$, we have, $v_n\to f(m_\alpha) = 0$ as $n\to\infty$.

As $f\in \Upsilon_\theta(\bigsqcup_{\alpha\in \mathcal{A}}I_{\alpha})$ and $\theta$ is slowly varying, by Proposition~\ref{prop:controlder}, there exists $E>1$ such that
\[E^{-1}q_n\theta(q_n)\leq|S_{q_n}(f')(x)|\leq Eq_n\theta(q_n)\text{ for any }x\in T^i\Delta^n(\tfrac{1}{4})\text{ and }0\leq i<q_n.
\]

 Notice that $[\xi^n_i-|\Delta^n|/16,\xi^n_i+|\Delta^n|/16]\subseteq \Delta^n(\tfrac{1}{4})$. Let $0<D<C^{-1}/8$, where $C$ is the constant appearing in \eqref{eq: qn1}--\eqref{eq: qn5} and let $\Delta^n_i$ be the interval centred at $\xi^n_i$ such that $|\Delta^n_i|=\frac{D}{q_n\theta(q_n)}<|\Delta_n|/8$.
Then $\Delta^n_i\subseteq \Delta^n(\tfrac{1}{4})$ and $|S_{q_n}f(\xi^n_i)|=v_n\leq \frac{D}{4E}$, for all $n$ large enough. In view of Theorem~\ref{thm:ergcr}, this gives the ergodicity of $T_f$.
\end{proof}

\section{Applications to skew product with logarithmic singularities}\label{sec:appskew}

Given an IET $T: I \to I$ with exchanged intervals $\{I_\alpha\}_{\alpha \in \A}$, we denote by $\mathrm{LG}(\sqcup_{\alpha\in \mathcal{A}} I_{\alpha})$ the set of functions (or cocycles) $\varphi: I\to \R$ with \emph{logarithmic singularities of geometric type} associated with $T$, that is, the set of functions for which there exist
constants $C_\alpha^\pm=C_\alpha^\pm(\varphi) %
\in \mathbb{R}$, for every $\alpha \in \mathcal{A}$, and a function $g_\varphi: I\to\R$ absolutely continuous on the interior of the exchanged intervals, such that
\begin{equation}\label{def:logsing}
\begin{aligned}
\varphi(x)=&-\sum_{\alpha\in\mathcal{A}}C^+_\alpha\log\Big(|I|\Big\{\frac{x-l_\alpha}{|I|}\Big\}\Big)
- \sum_{\alpha\in\mathcal{A}} C^-_\alpha\log\Big(|I|\Big\{\frac{r_\alpha-x}{|I|}\Big\}\Big)+g_\varphi(x),
\end{aligned}
\end{equation}
where $\{ \cdot \}$ denotes the fractional part, and
 \[ C_{\pi_0^{-1}(d)}^{-} \cdot C_{\pi_1^{-1}(d)}^{-}=0 \quad\text{and}\quad C_{\pi_0^{-1}(1)}^{+}\cdot C_{\pi_1^{-1}(1)}^{+}=0.\]
For every $\varphi\in \mathrm{LG}(\sqcup_{\alpha\in \mathcal{A}}
I_{\alpha})$ we set
\[\mathcal{L}(\varphi) := \sum_{\alpha\in\mathcal{A}}(|C^+_\alpha|+|C^-_\alpha|)\quad\text{and}\quad
\mathcal{LV}(\varphi):=\mathcal{L}(\varphi)+\operatorname{Var} g_\varphi.\]
For any $\varphi \in\mathrm{LG}(\sqcup_{\alpha\in \mathcal{A}} I_{\alpha})$ and $\mathcal{O}\in\Sigma(\pi)$ let
\[
\Delta_\mathcal{O}(\varphi):=\sum_{\alpha\in\mathcal{A}_\mathcal{O}^-}C^-_{\alpha}-
\sum_{\alpha\in\mathcal{A}_\mathcal{O}^+}C^+_{\alpha}\quad\text{and let}\quad
\mathcal{AS}(\varphi):=\sum_{\mathcal{O}\in\Sigma(\pi)}|\Delta_\mathcal{O}(\varphi)|.\]
If $T = (\pi, \lambda)$ comes from a locally Hamiltonian flow $\psi_\R$ on a surface $M$ and $\varphi=\varphi_f\in \mathrm{LG}(\sqcup_{\alpha\in \mathcal{A}} I_{\alpha})$ for some smooth $f:M\to \R$ (cf.\ \eqref{def:phif}), then $\Delta_\mathcal{O}(\varphi)$ measures the asymmetry of the Birkhoff integrals of $f$ when the orbits pass close to the saddle associated with $\mathcal{O}\in\Sigma(\pi)$.

Recall that the genus $g = g(T)$ of any suspension $(\pi, \lambda, \tau)$ is determined by $\pi$ (see \eqref{eq:genus}). In view of Theorem~6.1 in \cite{Fr-Ul2} and the definition of distributions $(d_i)_{1 \leq i \leq g}$ in \cite[Section~7.2.1]{Fr-Ul2} (recall that $D_i(f)=d_i(\varphi_f)$, for $1\leq i\leq {g}$), we have the following.
\begin{theorem}\label{operatorcorrection}
For any IET $T : I \to I$ satisfying $\mathrm{UDC}$ there exist $g$ bounded linear operators $d_i : \mathrm{LG}(\sqcup_{\alpha\in \mathcal{A}} I_{\alpha}) \to \R$, $1 \leq i \leq g$, such that, for every $\varphi\in\mathrm{LG}(\sqcup_{\alpha\in \mathcal{A}} I_{\alpha})$ satisfying $d_i(\varphi)=0$, for all $1\leq i\leq g$, there exists a sequence of renormalizing intervals $\{I^{(k)}\}_{k\geq 0}$ such that%
\begin{equation}\label{eq:thmcorr1}
\limsup_{k\to+\infty}\frac{\log\left(\frac{1}{|I^{(k)}|}\|S(k)\, {\varphi}\|_{L^1(I^{(k)})}\right)}{k}\leq 0.
\end{equation}
Furthermore, if additionally $T$ satisfies $\mathrm{SUDC}$ and $\mathcal{AS}(\varphi)=0$, then the sequence
\begin{equation}\label{eq:thmcorr2}
\left\{\frac{1}{|I^{(k)}|}\|S(k)\, {\varphi}\|_{L^1(I^{(k)})}\right\}_{k\geq 0}\quad\text{is bounded.}
\end{equation}
\end{theorem}
Recall that UDC and SUDC are the two Diophantine conditions already mentioned in Section~\ref{sc:Diophantine}. They were introduced in \cite{Fr-Ul2}, where it was proven that almost every IET satisfies both conditions.

In the following, for any IET $T$ satisfying UDC and for any $\varphi \in \mathrm{LG}(\sqcup_{\alpha\in \mathcal{A}} I_{\alpha})$ we denote
\[\overline{d}(\varphi):= (d_i(\varphi))_{i = 1}^g \in \R^g,\]
where $d_i $, for $1 \leq i \leq g$, are the linear operators given by Theorem~\ref{operatorcorrection}.

\begin{remark}\label{rmk:defdi}
Suppose that $T$ satisfies UDC.
By the definition of the correction operator $\mathfrak{h}:\mathrm{LG}(\sqcup_{\alpha\in \mathcal{A}} I_{\alpha})\to F \subseteq \R^{d}$, $\dim F=g$ (see the proof of Theorem~6.1 in \cite{Fr-Ul2}) and Lemma~6.8 in \cite{Fr-Ul2},
for any $\varphi\in \mathrm{LG}(\sqcup_{\alpha\in \mathcal{A}} I_{\alpha})$, $\mathfrak{h}(\varphi)$ is the unique vector (piecewise constant function) in $F$ such that
\[\lim_{k\to\infty}\frac{\log\frac{1}{|I^{(k)}|}\|S(k)(\varphi-\mathfrak{h}(\varphi))\|_{L^1(I^{(k)})}}{k}\leq 0.\]
Thus \eqref{eq:thmcorr1} implies $\mathfrak{h}(\varphi)=0$. Moreover, $\mathfrak{h}(\varphi)=\sum_{i=1}^gd_i(\varphi)h_i$, where $h_1,\ldots, h_g$ is a basis of $F$, see Section~7.2.1 in \cite{Fr-Ul2}. Hence, $\overline{d}(\varphi) = 0$. Therefore, \eqref{eq:thmcorr1} is equivalent to $\overline{d}(\varphi) = 0$.
\end{remark}
Moreover, we have the following criterion for ergodicity.
\begin{theorem}[Theorem~8.1 in \cite{Fr-Ul2}]\label{thm:erg}
For any IET $T : I \to I$ satisfying $\mathrm{SUDC}$ and any cocycle $\varphi\in \mathrm{LG}(\sqcup_{\alpha\in \mathcal{A}} I_{\alpha})$ satisfying
\[
\mathcal{L}(\varphi)>0,\qquad \mathcal{AS}(\varphi)=0,\qquad \overline{d}(\varphi) = 0, \qquad g'_{\varphi}\in \mathrm{LG}(\sqcup_{\alpha\in \mathcal{A}} I_{\alpha}),\]
the skew product $T_{\varphi}$ on $I\times \R$ is ergodic.
\end{theorem}

\begin{proposition}\label{prop:van}
For a.e. symmetric IET $T = (\pi,\lambda)$ and any $\varphi\in\mathrm{LG}(\sqcup_{\alpha\in \mathcal{A}} I_{\alpha})$ anti-symmetric we have $\overline{d}(\varphi) = 0$.
\end{proposition}

\begin{proof}
Note that, for a.e. symmetric IET $T = (\pi, \lambda)$ there exists $0<\kappa<1$ and a sequence of renormalizations $\{(\pi^{(k)},\lambda^{(k)})\}_{k\geq 1}$ such that:
\begin{itemize}
\item[(i)] $\pi^{(k)}=\pi$ for $k\geq 0$;
\item[(ii)] $\lambda^{(k)}_\alpha>\kappa|I^{(k)}|$ for all $\alpha\in\mathcal{A}$;
\item[(iii)] $T$ satisfies UDC for the sequence of renormalizations $\{(\pi^{(k)},\lambda^{(k)})\}_{k\geq 1}$;
\item[(iv)] $|I^{(k)}|p^{(k)}\geq \kappa>0$, where $p^{(k)}$ is the maximal number for which $T^j(\mathrm{Int}I^{(k)})$, $0\leq j<p^{(k)}$ avoid the discontinuities of $T$;
\item[(v)] for every $\alpha\in\mathcal{A}\cup\{I\}$ there exists $\beta_\alpha\in\mathcal{A}\cup\{I\}$ and $0\leq s^{(k)}_\alpha< q_\alpha^{(k)}$ such that $T^{s^{(k)}_\alpha}m^{(k)}_\alpha=m_{\beta_\alpha}$, where $q_I^{(k)}= p^{(k)}$.
\end{itemize}
Let $J^{(k)}:=I^{(k)}_{\alpha_0}$ if $\beta_{\alpha_0}=I$, and $J^{(k)}:=I^{(k)}$ if $\beta_I=I$. We require an additional condition for $T$:
\begin{itemize}
\item[(vi)] the length of all intervals exchanged by $T_{J^{(k)}}$ are greater than $\kappa|J^{(k)}|$.
\end{itemize}
To show that the a.e.\ symmetric IET $T$ satisfies the above conditions, one must construct a set $Y\subset\mathcal{M}$ of positive measure such that for any $(\pi',\lambda,\tau)\in Y$ we have,
\begin{itemize}
\item $\pi'=\pi$ and $\lambda_\alpha\geq \kappa|\lambda|$ for $\alpha\in\mathcal{A}$ (this is responsible for (i) and (ii));
\item $(\pi,\lambda,\tau)$ satisfies \eqref{eq:sym1} and \eqref{eq:sym2} (in view of Remark~\ref{rem: sym_midpoints}, this is responsible for (v));
 \item $(\pi,\lambda,\tau)$ satisfies the conditions (i) and (ii) in the proof of Lemma~3.6 (this is responsible for (iv))(this is responsible for (iv));
 \item for any $\alpha\in\mathcal A$ the IET induced from $(\pi,\lambda)$ on $I_\alpha$ is such that the length of all its exchanged intervals are greater than $\kappa\lambda_\alpha$ (this is responsible for (vi)).
\end{itemize}
Finally, we apply the proof of Theorem~3.8 in \cite{Fr-Ul2} (to the set $Y\subset \mathcal{M}$) to obtain a renormalization sequence (given by recurrence to the set $K\subset Y$) for which the Diophantine condition UDC is also met.
\medskip

For every $\alpha\in\mathcal{A}$, let $r^{(k)}_\alpha\in\N$ be the first return time of $J^{(k)}_{\alpha}$ for $T^{(k)}$ to $J^{(k)}$. By assumption, $r^{(k)}_\alpha\leq \kappa^{-2}$, for all $k\geq 1$ and any $\alpha\in\mathcal{A}$. For any $0\leq j<r_\alpha^{(k)}$, let $\gamma^{(k)}(\alpha,j)\in\mathcal{A}$ be such that $(T^{(k)})^j J^{(k)}_{\alpha} \subseteq I^{(k)}_{\gamma^{(k)}(\alpha,j)}$. As the sequences $\{r^{(k)}_\alpha\}_{k\geq 1}$ are bounded, we may pass to a subsequence and assume that $r^{(k)}_\alpha$ and $\gamma^{(k)}(\alpha,j)$ do not depend on $k$, so we will write $r_\alpha$ and $\gamma(\alpha,j)$ instead.
Let us consider the appearance matrix $D\in SL(d,\Z)$ given by $D_{\alpha\beta}=\#\{0\leq j<r_\alpha:\gamma(\alpha,j)=\beta\}$.

Fix $ k \geq 1$. Let $J:=T^{s}J^{(k)}$, where $s:=s^{(k)}_{\alpha_0}$ if $\beta_{\alpha_0}=I$, and $s:=s^{(k)}_{I}$ if $\beta_I=I$. Then $1/2=m_I$ is the midpoint of $J$. It follows directly from the assumptions that $\{T^j\mathrm{Int}J^{(k)}\mid0\leq j\leq s\}$ is a tower of intervals avoiding the discontinuities of $T$.

In view of Lemma~\ref{lem:ind}, for any $0<\delta<1/2$, we have
\[m(\varphi_J,J_\alpha(\delta))=0\text{ for all }\alpha\in\mathcal{A}.\]
Note that
\begin{equation}\label{eq:diffmv}
|m(\varphi_J,J_\alpha(\delta))-m(\varphi_{J^{(k)}},J^{(k)}_\alpha(\delta))|\leq \mathcal{LV}(\varphi)\frac{1+\log\|Q(k)\|}{\kappa\delta}.
\end{equation}
Indeed, as $J_\alpha=T^sJ^{(k)}_\alpha$, we have
\begin{equation}\label{eq:phiJ}
\varphi_{J^{(k)}}=\sum_{0\leq j<r_\alpha}S(k)\varphi\circ (T^{(k)})^j= S_{u^{(k)}_\alpha}\varphi\quad\text{on } J^{(k)}_\alpha, \quad\text{and}\quad
\varphi_{J}=S_{u^{(k)}_\alpha}\varphi\quad\text{on } J_\alpha,
\end{equation}
where $u^{(k)}_\alpha:=\sum_{0\leq j<r_\alpha}q^{(k)}_{\gamma(\alpha,j)}$. Notice that $T^{u^{(k)}_\alpha}x\in J^{(k)}$, for any $x\in J^{(k)}_\alpha$. It follows that
\begin{align*}
|m(\varphi_J,J_\alpha(\delta))-m(\varphi_{J^{(k)}},J^{(k)}_\alpha(\delta))|&=|m(S_{u^{(k)}_\alpha}\varphi,T^sJ^{(k)}_\alpha(\delta))-m(S_{u^{(k)}_\alpha}\varphi,J^{(k)}_\alpha(\delta))|\\
&= \frac{1}{|J^{(k)}_\alpha(\delta)|}\left|\int_{J^{(k)}_\alpha(\delta)}(S_{u^{(k)}_\alpha}\varphi(T^sx)-S_{u^{(k)}_\alpha}\varphi(x))dx\right|\\
&= \frac{1}{|J^{(k)}_\alpha(\delta)|}\left|\int_{J^{(k)}_\alpha(\delta)}(S_s\varphi(T^{u_\alpha^{(k)}}x)-S_s\varphi(x))dx\right|\\
& \leq \frac{1}{|J^{(k)}_\alpha(\delta)|}\int_{J^{(k)}_\alpha(\delta)}\left|S_s\varphi(T^{u_\alpha^{(k)}}x)-S_s\varphi(x)\right|dx.
\end{align*}
Moreover, for any $x\in J^{(k)}_\alpha(\delta)$, we have $T^{u_\alpha^{(k)}}x\in J^{(k)}(\kappa\delta)$ and
\[\left|S_s\varphi(x)-S_s\varphi(T^{u_\alpha^{(k)}}x)\right|=|x-T^{u_\alpha^{(k)}}x||S_s\varphi'(y)|,\]
for some $y\in J^{(k)}(\kappa\delta)$. Since $\{T^j\mathrm{Int}J^{(k)} \mid 0\leq j\leq s\}$ is a tower of intervals avoiding discontinuities of $T$, standard arguments (cf.\ the proof of Lemma~5.10 in \cite{Fr-Ul2}) show that
\[|S_s\varphi'(y)|\leq \mathcal{LV}(\varphi)\left(\frac{1}{\kappa\delta|J^{(k)}|}+\frac{1+\log\|Q(k)\|}{\kappa|I^{(k)}|}\right).\]
This gives \eqref{eq:diffmv}.

In view of \eqref{eq:phiJ}, we have
\begin{equation}\label{eq:mJk}
m(\varphi_{J^{(k)}},J^{(k)}_\alpha(\delta))=\sum_{0\leq j<r_\alpha}m\big(S(k)\varphi,(T^{(k)})^jJ_\alpha^{(k)}(\delta)\big).
\end{equation}
Let $0 \leq j < r_\alpha$. As $ (T^{(k)})^jJ_\alpha^{(k)}(\delta) \subseteq I_{\gamma(\alpha,j)}^{(k)}$ and
\[|(T^{(k)})^jJ_\alpha^{(k)}(\delta)|=|J_\alpha^{(k)}(\delta)|=(1-2\delta)|J_\alpha^{(k)}|\geq (1-2\delta)\kappa|I_{\gamma(\alpha,j)}^{(k)}|,\]
in view of Proposition~2.5 in \cite{Fr-Ul1} and Corollary~5.12 in \cite{Fr-Ul2}, we have
\[\big|m\big(S(k)\varphi,(T^{(k)})^jJ_\alpha^{(k)}(\delta)\big)-m\big(S(k)\varphi,I_{\gamma(\alpha,j)}^{(k)}\big)\big|\leq\mathcal{LV}(S(k)\varphi)\Big(4+\frac{1}{(1-2\delta)\kappa}\Big)=O(\log\|Q(k)\|).\]
In view of \eqref{eq:phiJ} and \eqref{eq:mJk}, it follows that for every $\alpha\in\mathcal{A}$, we have
\begin{align*}
\sum _{\beta\in\mathcal{A}}D_{\alpha\beta}m(S(k)\varphi,I_{\beta}^{(k)})=\sum_{0\leq j<r_\alpha}m\big(S(k)\varphi,I_{\gamma(\alpha,j)}^{(k)}\big)=O(\log\|Q(k)\|).
\end{align*}
Since the matrix $D$ is invertible, this gives
\[m(S(k)\varphi,I_{\alpha}^{(k)})=O(\log\|Q(k)\|),\qquad\text{for all}\quad\alpha\in\mathcal{A}.\]
In view of Proposition~2.5 in \cite{Fr-Ul1} and Corollary~5.12 in \cite{Fr-Ul2}, for any $\alpha\in\mathcal{A}$, we have
\begin{align*}
\left|\frac{1}{|I^{(k)}_\alpha|}\int_{I^{(k)}_\alpha}|S(k)\varphi(x)|dx-|m(S(k)\varphi,I_{\alpha}^{(k)})|\right|&\leq
\frac{1}{|I^{(k)}_\alpha|}\int_{I^{(k)}_\alpha}|S(k)\varphi(x)-m(S(k)\varphi,I_{\alpha}^{(k)})|dx\\
&\leq 8\mathcal{LV}(S(k)\varphi)=O(\log\|Q(k)\|).
\end{align*}
Hence
\[\frac{\|S(k)\varphi\|_{L^1(I^{(k)})}}{|I^{(k)}|}=O(\log\|Q(k)\|).\]
As $T$ satisfies UDC, in view of Remark~\ref{rmk:defdi}, this gives $\overline{d}(\varphi) = 0$.
\end{proof}

\begin{theorem}\label{thm:ergasym}
Let $d \geq 2$ even. For a.e.\ symmetric IET $T : I \to I$ on $d$ intervals and any cocycle $\varphi\in \mathrm{LG}(\sqcup_{\alpha\in \mathcal{A}} I_{\alpha})$ satisfying
\[
\mathcal{L}(\varphi)>0, \qquad \overline{d}(\varphi) = 0, \qquad \textup{Var}(g_\varphi') < +\infty,\]
the skew product $T_{\varphi}$ on $I\times \R$ is ergodic.

\end{theorem}

\begin{proof}
Since $d$ is even, the set $\Sigma(\pi)$ is one-element, say $\Sigma(\pi)=\{\mathcal{O}\}$. If $\Delta_{\mathcal{O}}(\varphi)=0$, then $\mathcal{AS}(\varphi)=|\Delta_{\mathcal{O}}(\varphi)|=0$ and the ergodicity of $T_{\varphi}$ follows directly from Theorem~\ref{thm:erg}.%

Suppose that $\Delta_{\mathcal{O}}(\varphi)\neq 0$. Let $a:=\Delta_{\mathcal{O}}(\varphi)/(2d)$ and consider $\varphi_a\in\mathrm{LG}(\sqcup_{\alpha\in \mathcal{A}} I_{\alpha})$ given by
\[\varphi_a(x)=\left\{
\begin{array}{cl}
-a\log(r_\alpha-x)+a\log(x-l_\alpha)&\text{ for }x\in I_\alpha\text{ if }\pi_0(\alpha)\neq d,\\
0&\text{ for }x\in I_\alpha\text{ if }\pi_0(\alpha)= d.
\end{array}\right.\]
Then $\varphi_a$ is an anti-symmetric piecewise monotonic map with $\Delta_{\mathcal{O}}(\varphi_a)=2da=\Delta_{\mathcal{O}}(\varphi)$. By Proposition~\ref{prop:van}, $\overline{d}(\varphi_a)=0.$ We will use the following decomposition $\varphi=\varphi_a+\varphi_s$. Then $\varphi_s\in\mathrm{LG}(\sqcup_{\alpha\in \mathcal{A}} I_{\alpha})$ with $\mathcal{AS}(\varphi_s)=|\Delta_{\mathcal{O}}(\varphi)-\Delta_{\mathcal{O}}(\varphi_a)|=0$
and $\overline{d}(\varphi_s)=\overline{d}(\varphi)-\overline{d}(\varphi_a)=0$.

By Proposition~\ref{prop:DC} (and its proof), there exist $0<\kappa<1$, $C>0$, and a sequence of renormalizations $\{(\pi^{(k)},\lambda^{(k)})\}_{k\geq 1}$, such that
\begin{itemize}
\item[(i)] $\pi^{(k)}=\pi$ for $k\geq 0$;
\item[(ii)] $\lambda^{(k)}_\alpha>\kappa|I^{(k)}|$ for all $\alpha\in\mathcal{A}$;
\item[(iii)] $|I^{(k)}|p^{(k)}\geq \kappa>0$, where $p^{(k)}$ is the maximal number for which $T^j\mathrm{Int}I^{(k)}$, $0\leq j<p^{(k)}$ avoid the discontinuities of $T$;
\item[(iv)] for every $\alpha\in\mathcal{A}\cup\{I\}$ there exists $\beta_\alpha\in\mathcal{A}\cup\{I\}$ and $0\leq s^{(k)}_\alpha< q_\alpha^{(k)}$ such that $T^{s^{(k)}_\alpha}m^{(k)}_\alpha=m_{\beta_\alpha}$, where $q_I^{(k)}=p^{(k)}$;
\item[(v)] there exists $\alpha_0\in \mathcal{A}$ such that the sequence of towers $\{\{T^jI^{(k)}_{\alpha_0}\mid 0\leq j<q_{\alpha_0}^{(k)}\}\}_{k\in\N}$ satisfies \eqref{eq: qn1}--\eqref{eq: qn4} and $|I^{(k)}_{\alpha_0}|/32<|T^{q_{\alpha_0}^{(k)}}x-x|<|I^{(k)}_{\alpha_0}|/16$ for $x\in I^{(k)}_{\alpha_0}$;
\item[(vi)] the IET $T$ satisfies SUDC.
\end{itemize}
Indeed, (ii) and (iii) are included in SUDC. Property (iv) follows from Remark~\ref{rem: sym_midpoints}. Properties \eqref{eq: qn1}--\eqref{eq: qn4} in (v) hold directly for $\alpha_0=\pi_0^{-1}(d)$
while obtaining the final double inequality requires some modification of the definition \eqref{def:D1} of the set $\mathcal{D}_1$ by halving the constants in its double inequality.

For any $k\geq 1$ let $\Delta^k:=I^{(k)}_{\alpha_0}(1/4)$ and $q_k:=q^{(k)}_{\alpha_0}$. As $|I^{(k)}_{\alpha_0}|/32<|\delta|<|I^{(k)}_{\alpha_0}|/16$ and $|\Delta^k|=|I^{(k)}_{\alpha_0}|/2$, we have
\[ |\Delta^k|/16<|T^{q_k}x-x|<|\Delta^k|/8, \qquad \text{ for any }x\in\Delta^k.\]
It follows that $\{\{T^i\Delta^k \mid 0\leq i<q_k\}\}_{k\in\N}$ is a sequence of Rokhlin towers satisfying \eqref{eq: qn1}--\eqref{eq: qn5}. %
By \eqref{eq: qn1}--\eqref{eq: qn5} and Proposition~\ref{prop:controlder} with $\theta=\log$, there exists $E>1$ such that
\begin{equation}\label{eq:der1}
E^{-1}q_k\log(E^{-1}q_k)\leq |S_{q_k}\varphi'_a(x)|\leq Eq_k\log(Eq_k)\qquad \text{ for any }x\in \bigsqcup_{0\leq j<q_k}T^j\Delta^k(1/4).
\end{equation}

We now use another decomposition $\varphi_s=\varphi_0+g$, where
\[\varphi_0(x)=-\sum_{\alpha\in\mathcal{A}}C^+_\alpha(\varphi_s)\log\Big(|I|\Big\{\frac{x-l_\alpha}{|I|}\Big\}\Big)
- \sum_{\alpha\in\mathcal{A}} C^-_\alpha(\varphi_s)\log\Big(|I|\Big\{\frac{r_\alpha-x}{|I|}\Big\}\Big)\text{ and }g=g_{\varphi_s}.\]
As $g_{\varphi_s} - g_\varphi$ is piecewise $C^\infty$, we have $g=g_{\varphi_s}$ is piecewise absolutely continuous whose derivative is of bounded variation.
 In view of Proposition~5.6 in \cite{Fr-Ul2}, for any $0\leq j\leq q^{(k)}_{\alpha_0}$ and $x\in I^{(k)}_{\alpha_0}(1/16)$, we have
\begin{equation}\label{eq:derg0}
|S_j\varphi'_0(x)|\leq \mathcal{L}(\varphi_0)\Big(\frac{16}{|I^{(k)}_{\alpha_0}|}+Mq^{(k)}_{\alpha_0}\Big)
\leq \mathcal{L}(\varphi_0)(16\kappa^{-2}+M)q^{(k)}_{\alpha_0}.
\end{equation}
As $g'$ is bounded, we also have
\begin{equation}\label{eq:derg}
|S_{j}g'(x)|\leq j\|g'\|_{\sup}, \quad\text{for any } x\in I\text{ and any }j\in\N.
\end{equation}
It follows that
\begin{equation}\label{eq:derj}
|S_j\varphi_s'(x)|\leq (\mathcal{L}(\varphi_0)(16\kappa^{-2}+M)+\|g'\|_{\sup})q_k,\quad \text{for any }x\in I^{(k)}_{\alpha_0}(1/16)\text{ and any }0\leq j\leq q_k.
\end{equation}
Note that
\begin{equation}\label{eq:der2}
|S_{q_k}\varphi'_0(x)|\leq 3\mathcal{L}(\varphi_0)(16\kappa^{-2}+M)q_k, \quad \text{for all }x\in \bigcup_{0\leq j<q_k}T^j\Delta^k(1/4).
\end{equation}
Indeed, if $0\leq j<q^{(k)}_{\alpha_0}$ then for every $x\in \Delta^k(1/4)= I^{(k)}_{\alpha_0}(3/8)$, we have
\[S_{q^{(k)}_{\alpha_0}}\varphi'_0(T^jx)=S_{q^{(k)}_{\alpha_0}}\varphi'_0(x)-S_j\varphi'_0(x)+S_j\varphi'_0(T^{q^{(k)}_{\alpha_0}}x),\qquad \text{ with
 }T^{q^{(k)}_{\alpha_0}}x\in I^{(k)}_{\alpha_0}(5/16).\]
In view of \eqref{eq:derg0}, this gives \eqref{eq:der2}. Next,
by \eqref{eq:der1}, \eqref{eq:der2} and \eqref{eq:derg}, for all $k$ large enough,
\begin{equation}\label{eq:der3}
(2E)^{-1}q_k\log q_k\leq |S_{q_k}\varphi'(x)|\leq 2Eq_k\log q_k, \qquad \text{ for all }x\in \bigcup_{0\leq j<q_k}T^j\Delta^k(1/4).
\end{equation}

Let $\beta:=\beta_{\alpha_0}$. Let us consider the interval $J=[m_\beta-\vep,m_\beta+\vep]$ with $\vep=\frac{1}{8}|I^{(k)}_{\alpha_0}|$. Then $\{T^jJ \mid 0\leq j<q_k\}$ is a Rokhlin tower of intervals included in the exchanged intervals, and $T^{q_k}m_\beta=m_\beta+\delta$ with $|\delta|<\vep/3$. In view of Lemma~\ref{lem:anti}, for any $0\leq i<q_k$ there exists
\[\xi^k_i\in T^{-s_{\alpha_0}^{(k)}+i}[m_\beta-\vep/2,m_\beta+\vep/2]=T^i[m^{(k)}_{\alpha_0}-\vep/2,m^{(k)}_{\alpha_0}+\vep/2] \subseteq T^iI^{(k)}_{\alpha_0}(7/16)\] such that
\begin{equation}\label{eq:choicexi}
|S_{q_k}\varphi_a(\xi_i^k)|\leq 1.
\end{equation}

As $\mathcal{AS}(\varphi_s)=0$ and $\overline{d}(\varphi_s)=0$, by Theorem~\ref{operatorcorrection}, there exists $C>0$ such that
\[\frac{1}{|I^{(k)}|}\|S(k)(\varphi_s)\|_{L^1(I^{(k)})}\leq C, \qquad \text{ for all }k\geq 1.\]
In view of Corollary~4.5 and Proposition~5.8 in \cite{Fr-Ul2}, for every $x\in I^{(k)}_{\alpha_0}(1/4)$ we have
\begin{align}\label{eq:boundint}
\begin{aligned}
|S_{q_k}\varphi_s(x)|=|S(k)\varphi_s(x)|&\leq 2\frac{|I^{(k)}|}{|I^{(k)}_{\alpha_0}|}\frac{\|S(k)(\varphi_s)\|_{L^1(I^{(k)})}}{|I^{(k)}|}+\mathcal{LV}(S(k)\varphi_s)(1+\log 2)\\
&\leq \frac{2C}{\kappa}+\frac{8 M \mathcal{LV}(\varphi_s)}{\kappa}.
\end{aligned}
\end{align}
For any $x\in I^{(k)}_{\alpha_0}(1/4)$ and any $0\leq i<q_k$ there exists $y\in [x,T^{q_k}x] \subseteq I^{(k)}_{\alpha_0}(3/16)$ such that
\begin{align*}
S_{q_k}\varphi_s(T^ix)=S_{q_k}\varphi_s(x)+S_{i}\varphi_s(T^{q_k}x)-S_{i}\varphi_s(x)=S_{q_k}\varphi_s(x)+S_{i}\varphi'_s(y)(T^{q_k}x-x).
\end{align*}
In view of \eqref{eq:boundint} and \eqref{eq:derj}, this gives
\[|S_{q_k}\varphi_s(T^ix)|\leq \frac{2C}{\kappa}+\frac{8 M \mathcal{LV}(\varphi_s)}{\kappa}+\mathcal{L}(\varphi_0)(16\kappa^{-2}+M)+\|g'\|_{\sup}.\]
As $\xi^k_i\in T^iI^{(k)}_{\alpha_0}(7/16) \subseteq T^iI^{(k)}_{\alpha_0}(1/4)$, by \eqref{eq:choicexi}, we have
\begin{equation}\label{eq:Xi}
|S_{q_k}\varphi(\xi_i^k)|\leq \Xi:=\frac{2C}{\kappa}+\frac{8 M \mathcal{LV}(\varphi_s)}{\kappa}+\mathcal{L}(\varphi_0)(16\kappa^{-2}+M)+\|g'\|_{\sup}+1.
\end{equation}
Choose any $D\geq 8E\Xi$. As $\xi^k_i\in T^iI^{(k)}_{\alpha_0}(7/16) \subseteq T^iI^{(k)}_{\alpha_0}(3/8)=T^i\Delta^k(1/4)$,
for any $k$ large enough and for any $0\leq i<q_k$ we can choose an interval $\Delta^k_i \subseteq \Delta^k(1/4)$ centered at $\xi^k_i$
such that $|\Delta^k_i|=\frac{D}{q_k\log q_k}$.

In summary, we have constructed a family of intervals which satisfies \eqref{eq: 0 of f} (it follows from \eqref{eq:Xi} and $D(8E)^{-1}\geq \Xi$) and \eqref{eq: derivative of f} (is given by \eqref{eq:der3}) with $D_1=D_2=D$, $E:=2E$, $v=0$ and $\theta=\log$. In view of Theorem~\ref{thm:ergcr}, this shows the ergodicity of the skew product $T_\varphi$.
\end{proof}

\section{Applications to locally Hamiltonian flows with perfect saddles}
\label{sc:PS}
Let $\psi_\R$ be a locally Hamiltonian flow whose fixed points are centers or perfect saddles. Let $M' \subseteq M$ be a minimal component and $I \subseteq M'$ a transversal curve. Denote by $T:I\to I$ the first return map to $I$, which in standard parameters is a minimal IET $T=(\pi,\lambda)$. Let $g$ be the genus of $M'$ and let $\gamma$ be the number of saddles in $M'$. In the following, we will use freely the notations introduced in Section \ref{sec:introHam}.

For any $\alpha \in \A$ and any closed interval $J \subseteq I_\alpha$, we denote by $J^\tau \subseteq M$ the closure of the set of orbit segments starting from $\mathrm{Int} J$ and running until the first return to $I$. For any $\sigma\in M'\cap\mathrm{PSd}(\psi_\R)$ and $\ell\in \mathcal{L}^{\sigma}_{\sim}$ there exist $\alpha\in\mathcal{A}$ and an interval $J$ of the form $[l_\alpha,l_\alpha+\vep]$ or $[r_\alpha-\vep,r_\alpha]$, for some $\varepsilon > 0$, such that $l_\alpha$ or $r_\alpha$ is
the first backward meeting point of a separatrix incoming to $\sigma\in\mathrm{PSd}(\psi_\R)$ and $J^\tau$ contains all angular sectors $U_{\sigma,l}$ with $l\in \ell$. All parameters (numbers) $l$ in $\ell$ have the same parity which we denote by $\varrho(\ell)$. If $\varrho(\ell)=$~odd then $j=[l_\alpha,l_\alpha+\vep]$ and if $\varrho(\ell)=$~even then $j=[r_\alpha-\vep,r_\alpha]$.

After \cite[Section 7.1]{Fr-Ki3}, for any $0\leq k\leq m_\sigma-2$ let us consider ${\xi}_{(\sigma,\ell,k)}:I\to\R$ a map such that
\begin{itemize}

\item if $\varrho(\ell)=$~odd, then for any $s\in I_\alpha$,
\begin{align*}
{\xi}_{(\sigma,\ell,k)}(s)=
\left\{
\begin{array}{cl}
\frac{1}{m_\sigma^2 k!}(s-l_\alpha)^{\frac{k-(m_\sigma-2)}{m_\sigma}}+const&\text{ if }0\leq k< m_\sigma-2,\\
-\frac{1}{m_\sigma^2 (m_\sigma-2)!}\log(s-l_\alpha)+const &\text{ if }k= m_\sigma-2;
\end{array}
\right.
\end{align*}
\item if $\varrho(\ell)=$~even, then for any $s\in I_\alpha$,
\begin{align*}
{\xi}_{(\sigma,\ell,k)}(s)=
\left\{
\begin{array}{cl}\frac{1}{m_\sigma^2 k!}(r_\alpha-s)^{\frac{k-(m_\sigma-2)}{m_\sigma}}+const&\text{ if }0\leq k< m_\sigma-2,\\
-\frac{1}{m_\sigma^2 (m_\sigma-2)!}\log(r_\alpha-s)+const &\text{ if }k= m_\sigma-2;
\end{array}
\right.
\end{align*}
\item ${\xi}_{(\sigma,\ell,k)}$ is constant on any interval $I_\beta$ with $\beta\neq \alpha$;
\item
\begin{equation}\label{eq:laxi}
\lim_{k\to\infty}\frac{1}{k}\log\Big(\frac{\|S(k)({\xi}_{(\sigma,\ell,k)})\|_{L^1(I^{(k)})}}{|I^{(k)}|}\Big)=\lambda_1\frac{k-(m_\sigma-2)}{m_\sigma}.
\end{equation}
\end{itemize}
Let $m$ be the maximal multiplicity of saddles in $\mathrm{PSd}(\psi_\R)\cap M'$. In view of Theorem~1.3 in \cite{Fr-Ki3} applied to any $0<r<\min\{1/m,\lambda_g/\lambda_1\}$, there exist invariant functionals
\[D_i:C^m(M)\to \R, \text{ for }1\leq i\leq g,\qquad B_s:C^m(M)\to \R, \text{ for }1\leq s<\gamma,\qquad V:C^m(M)\to \R,\]
piecewise constant maps on $I$ (constant on $\{I_\alpha\}_{\alpha\in\mathcal{A}}$)
\[h_i: I \to \R, \text{ for }1\leq i\leq g,\quad c_s:I \to \R, \text{ for }1\leq s<\gamma,\]
and a piecewise linear map (linear on $\{I_\alpha\}_{\alpha\in\mathcal{A}}$)
\[ v: I \to \R,\]
satisfying
\begin{gather}
\nonumber
\label{eq:lah}
\lim_{k\to\infty}\frac{\log\Big(\tfrac{\|S(k)h_i\|_{L^1(I^{(k)})}}{|I^{(k)}|}\Big)}{k}=\lambda_i, \\\label{eq:lacv}
\lim_{k\to\infty}\frac{\log\Big(\tfrac{\|S(k)c_s\|_{L^1(I^{(k)})}}{|I^{(k)}|}\Big)}{k}= 0 =
\lim_{k\to\infty}\frac{\log\Big(\tfrac{\|S(k)v\|_{L^1(I^{(k)})}}{|I^{(k)}|}\Big)}{k},
\end{gather}
for any $1 \leq i \leq g$ and any $1 \leq s \leq \gamma$, such that, for any $f\in C^m(M)$, we have
\begin{align}\label{eq:varphif}
\begin{aligned}
\varphi_f&=\sum_{\sigma\in M'\cap\mathrm{PSd}(\psi_\R)}\sum_{\ell\in\mathcal{L}_\sim^\sigma}\sum_{0\leq k\leq m_\sigma-2}\mathfrak{C}^k_{(\sigma,\ell)}(f)\xi_{(\sigma,\ell,k)}\\
&\quad
+\sum_{1\leq i\leq g}D_i(f)h_i+\sum_{1\leq s<\gamma}B_s(f)c_s+V(f)v+\mathfrak{r}(f),
\end{aligned}
\end{align}
where $\mathfrak{r}(f):I\to\R$ is a piecewise continuous map (continuous on $\{I_\alpha\}_{\alpha\in\mathcal{A}}$) such that
\begin{equation}
\label{eq:larf}
\limsup_{k\to\infty}\frac{\log\|S(k)\mathfrak{r}(f)\|_{\sup}}{k}\leq -\lambda_1 r.
\end{equation}
Then, standard Gottschalk-Hedlund type arguments (c.f.\ \cite[Section 3.4]{Ma-Mo-Yo}) show that $\mathfrak{r}(f)$ is a coboundary with continuous transfer map, i.e.,\ there exists a continuous function $g:I\to\R$ such that $\mathfrak{r}(f)=g\circ T-g$.

\begin{proof}[Sketch of the proof of Theorem~\ref{thm:devspect}]

Let $\tau:I\to\R_{>0}\cup\{+\infty\}$ be the first return time map for the flow $\psi_\R$ on $M'$. %
Using the notations introduced in Section \ref{sec:introHam} (in particular, from \eqref{def:phif}), we have $\tau=\varphi_1$. We will make use of the special representation $T^\tau_\R$ of $\psi_\R$ on $M'$.
Recall that $T^\tau_\R$ acts on $I^\tau:=\{(x,r):x\in I,0\leq r<\tau(x)\}$ so that $T^\tau_t(x,r)=(x,r+t)$ until it reaches the roof point $(x,\tau(x))$ which is identified with $(Tx,0)$ on the bottom.
For any map $\varphi:I\to\R$ let $f_{\varphi}:I^\tau\to\R$ be given by $f_{\varphi}(x,r)=\frac{\varphi(x)}{\tau(x)}$. Then $\varphi_{f_{\varphi}}=\varphi$. Moreover, if $\varphi$ is smooth of intervals $I_\alpha$, $\alpha\in \mathcal{A}$, then $f_{\varphi}:M\to\R$ is piecewise smooth, i.e.\ is smooth of the interior of $I^\tau_\alpha$, $\alpha\in \mathcal{A}$.

For any $1\leq i\leq g$, let $\zeta_i=f_{h_i}$, and for any $\sigma\in\mathrm{PSd}(\psi_R)$, any $\ell\in\mathcal{L}^\sigma_\sim$, and any $0\leq k\leq m_\sigma$, let $\zeta^k_{(\sigma,\ell)}=f_{\xi_{(\sigma,\ell,k)}}$ (with $h_i$ and $\xi_{(\sigma,\ell,k)}$ as defined at the beginning of this section). As each $h_i$ is bounded and the singularities of $\xi_{(\sigma,\ell,k)}$ are dominated by the singularities of $g=\varphi_1$ (see Theorem~9.1 in \cite{Fr-Ki1}), the maps $\zeta_i$ and $\zeta^k_{(\sigma,\ell)}$ are piecewise smooth and bounded.

Now we can define the cocycles $u_i(T,x)$ and $c^k_{(\sigma,\ell)}(T,x)$ by
\[u_i(T,x):=\int_0^T\zeta_i(\psi_tx)\,dt\ \text{ and } \ c^k_{(\sigma,\ell)}(T,x):=\int_0^T\zeta^k_{(\sigma,\ell)}(\psi_tx)\,dt,\ \text{ for any } T\in\R \text{ and any } x\in M'.\]
For any $f\in C^m(M)$ let us consider the decomposition
\[f=\sum_{\sigma\in M'\cap\mathrm{PSd}(\psi_\R)}\sum_{\ell\in\mathcal{L}_\sim^\sigma}\sum_{0\leq k< m_\sigma-2}\mathfrak{C}^k_{(\sigma,\ell)}(f)\zeta^k_{(\sigma,\ell)}+\sum_{1\leq i\leq g}D_i(f)\zeta_i+f_e,\]
where the distributions $\mathfrak{C}^k_{(\sigma,\ell)}$ are given by \eqref{def:gothc} and $f_e$ is a bounded piecewise smooth map (as the difference of a smooth map and a bounded piecewise smooth map) treated as an error term. Then
 \begin{align*}
 \int_0^Tf(\psi_tx)\,dt&=\sum_{\sigma\in M'\cap\mathrm{PSd}(\psi_\R)}\sum_{\ell\in\mathcal{L}_\sim^\sigma}\sum_{0\leq k< m_\sigma-2}\mathfrak{C}^k_{(\sigma,\ell)}(f)c^k_{(\sigma,\ell)}(T,x)\\
 &\quad+\sum_{1\leq i\leq g}D_i(f)u_i(T,x)+err(f,T,x)
 \end{align*}
with $err(f,T,x)=\int_0^Tf_e(\psi_tx)\,dt$. Taking any $x\in I$ and $T=\tau(x)$, we obtain
\[\varphi_f=\sum_{\sigma\in M'\cap\mathrm{PSd}(\psi_\R)}\sum_{\ell\in\mathcal{L}_\sim^\sigma}\sum_{0\leq k< m_\sigma-2}\mathfrak{C}^k_{(\sigma,\ell)}(f)\xi_{(\sigma,\ell,k)}
+\sum_{1\leq i\leq g}D_i(f)h_i+\varphi_{f_e},\]
so, by \eqref{eq:varphif},
\begin{equation}\label{eq:fe}
\varphi_{f_e}=\sum_{\sigma\in M'\cap\mathrm{PSd}(\psi_\R)}\sum_{\ell\in\mathcal{L}_\sim^\sigma}\mathfrak{C}^k_{(\sigma,\ell)}(f)\xi_{(\sigma,\ell,m_\sigma-2)}+\sum_{1\leq s<\gamma}B_s(f)c_s+V(f)v+\mathfrak{r}(f).
\end{equation}
In view of \eqref{eq:laxi}, \eqref{eq:lacv} and \eqref{eq:larf}, this gives
\[\limsup_{k\to\infty}\frac{\log\Big(\tfrac{\|S(k)\varphi_{f_e}\|_{L^1(I^{(k)})}}{|I^{(k)}|}\Big)}{k}\leq 0.\]
By Proposition~7.13 in \cite{Fr-Ki1}, we have
\[\limsup_{T\to\infty}\frac{\log |err(f,T,x)|}{\log T}=\limsup_{T\to\infty}\frac{\log |\int_0^Tf_e(\psi_tx)\,dt|}{\log T}\leq 0, \qquad \text{ for a.e. }x\in M'.\]
As $\varphi_{\zeta_i}=h_i$, by \eqref{eq:lah} and Proposition~7.13 in \cite{Fr-Ki1},
\[\limsup_{T\to\infty}\frac{\log |u_i(T,x)|}{\log T}=\limsup_{T\to\infty}\frac{\log |\int_0^T\zeta_i(\psi_tx)\,dt|}{\log T}=\frac{\lambda_i}{\lambda_1}, \qquad \text{ for a.e. }x\in M'.\]

As $\xi_{(\sigma,\ell,k)}=\varphi_{\zeta^k_{(\sigma,\ell)}}$, in view of \eqref{eq:laxi}, applying Theorem~7.10 in \cite{Fr-Ki1} to $f=\zeta^k_{(\sigma,\ell)}$ and $b=\frac{(m_\sigma-2)-k}{m_\sigma}$, we obtain
\[\limsup_{T\to\infty}\frac{\log |c^k_{(\sigma,\ell,k)}(T,x)|}{\log T}=\limsup_{T\to\infty}\frac{\log |\int_0^T\zeta^k_{(\sigma,\ell)}(\psi_tx)\,dt|}{\log T}\leq\frac{(m_\sigma-2)-k}{m_\sigma},\quad\text{ for a.e. }x\in M'.\]
As $|(x-l_\alpha)^{1+b}(\xi_{(\sigma,\ell,k)})'(x)|=\frac{b}{m_\sigma^2 k!}>0$ if $\varrho(\ell)=$~odd, and
$|(r_\alpha-x)^{1+b}(\xi_{(\sigma,\ell,k)})'(x)|=\frac{b}{m_\sigma^2 k!}>0$ if $\varrho(\ell)=$~even, for any $x\in I_\alpha$,
by Proposition~5.6 in \cite{Fr-Ki1} and applying the arguments in Part V of the proof of Theorem~1.1 in \cite{Fr-Ki1},
we obtain the lower bound
\[\limsup_{T\to\infty}\frac{\log |c^k_{(\sigma,\ell,k)}(T,x)|}{\log T}=\limsup_{T\to\infty}\frac{\log |\int_0^T\zeta^k_{(\sigma,\ell)}(\psi_tx)\,dt|}{\log T}\geq\frac{(m_\sigma-2)-k}{m_\sigma},\quad \text{ for a.e. }x\in M'.\]
\end{proof}

Now we are ready to prove Conjecture~\ref{conj} in the %
case where the flow is minimal on the whole surface.

\begin{theorem}\label{mainthm:mini}
For almost every minimal locally Hamiltonian flow $\psi_\mathbb{R}$ on a compact surface $M$ of genus $g \geq 1$, the following holds. Let $m:=\max\{m_\sigma:\sigma\in \mathrm{Fix}(\psi_\mathbb{R})\}$. For every $f\in C^m(M)$ such that
\begin{itemize}
\item $D_i(f)=0$ for all $1\leq i\leq g$,
\item $\mathfrak{C}^k_{\sigma,l}(f)=0$ for all $\sigma\in\mathrm{Fix}(\psi_\mathbb{R})$, $0\leq k<m_\sigma-2$ and $0\leq l<2m_\sigma$, and
\item $\mathfrak{C}^{m_\sigma-2}_{\sigma,l}(f)\neq 0$ for some $\sigma\in\mathrm{Fix}(\psi_\mathbb{R})$ and $0\leq k<m_\sigma-2$,
\end{itemize}
the skew product flow $\psi^f_\R$ is ergodic.
\end{theorem}

\begin{proof}
As we noted earlier, it is enough to show that the skew product map $T_{\varphi_f}$ is ergodic if $T$ satisfies SUDC.
In view of \eqref{eq:varphif}, we have
\[\varphi_f=\sum_{\sigma\in \mathrm{Fix}(\psi_\R)}\sum_{0\leq l<2m_\sigma}\mathfrak{C}^{m_\sigma-2}_{(\sigma,l)}(f)\xi_{(\sigma,l,m_\sigma-2)}
+\sum_{1\leq s<\gamma}B_s(f)c_s+V(f)v+\mathfrak{r}(f).\]
Let
\[\widetilde{\varphi}_f=\sum_{\sigma\in \mathrm{Fix}(\psi_\R)}\sum_{0\leq l<2m_\sigma}\mathfrak{C}^{m_\sigma-2}_{(\sigma,l)}(f)\xi_{(\sigma,l,m_\sigma-2)}
+\sum_{1\leq s<\gamma}B_s(f)c_s+V(f)v.\]
Since $\mathfrak{r}(f)$ is a coboundary, the skew products $T_{\varphi_f}$ and $T_{\widetilde{\varphi}_f}$ are isomorphic. Therefore, it is enough to show the ergodicity of $T_{\widetilde{\varphi}_f}$.
Note that $\widetilde{\varphi}_f\in \mathrm{LG}(\sqcup_{\alpha\in \mathcal{A}} I_{\alpha})$ and, by \eqref{eq:laxi} and \eqref{eq:lacv}, we have
\[\lim_{k\to\infty}\frac{1}{k}\log\Big(\frac{\|S(k)(\widetilde{\varphi}_f)\|_{L^1(I^{(k)})}}{|I^{(k)}|}\Big)\leq 0.\]
In view of Remark~\ref{rmk:defdi}, this gives $\overline{d}(\widetilde{\varphi}_f)=0$.%

Recall that every $\mathcal{O}\in\Sigma(\pi)$ represents a saddle $\sigma$. Then, by the definition of $\xi_{(\sigma,l,m_\sigma-2)}$, we have
\begin{gather*}
\mathcal{L}(\widetilde{\varphi}_f)=\sum_{\alpha\in\mathcal{A}}(|C^-_{\alpha}(\widetilde{\varphi}_f)|+ |C^+_{\alpha}(\widetilde{\varphi}_f)|)=\tfrac{1}{m_\sigma^2 (m_\sigma-2)!}\sum_{\sigma\in\mathrm{Fix}(\psi_\R)}\sum_{0\leq l<m_\sigma}|\mathfrak{C}^{m_\sigma-2}_{(\sigma,l)}(f)|>0,\\
\Delta_\mathcal{O}(\widetilde{\varphi}_f)=\sum_{\alpha\in\mathcal{A}_\mathcal{O}^-}C^-_{\alpha}(\widetilde{\varphi}_f)-
\sum_{\alpha\in\mathcal{A}_\mathcal{O}^+}C^+_{\alpha}(\widetilde{\varphi}_f)=\tfrac{1}{m_\sigma^2 (m_\sigma-2)!}\Big(\sum_{0\leq l<m_\sigma}\mathfrak{C}^{m_\sigma-2}_{(\sigma,2l)}(f)-\sum_{0\leq l<m_\sigma}\mathfrak{C}^{m_\sigma-2}_{(\sigma,2l+1)}(f)\Big).
\end{gather*}
By the definition of functionals $\mathfrak{C}^k_{(\sigma,\ell)}$, we have
\begin{align*}
\sum_{0\leq l<m_\sigma}\mathfrak{C}^{m_\sigma-2}_{(\sigma,2l)}(f)
 & =\sum_{0\leq l<m_\sigma}\sum_{0\leq i\leq m_\sigma-2}
\theta_\sigma^{2l(2i+2-m_\sigma)}\tbinom{m_\sigma-2}{i}\mathfrak{B}(\tfrac{(m_\sigma-1)-i}{m_\sigma},\tfrac{i+1}{m_\sigma})\tfrac{\partial^{m_\sigma-2}(f\cdot V)}{\partial z^i\partial\overline{z}^{m_\sigma-2-i}}(0,0)\\
&=\sum_{0\leq i\leq m_\sigma-2}\Big(\sum_{0\leq l<m_\sigma}\theta_\sigma^{2l(2i+2-m_\sigma)}\Big)
\tbinom{m_\sigma-2}{i}\mathfrak{B}(\tfrac{(m_\sigma-1)-i}{m_\sigma},\tfrac{i+1}{m_\sigma})\tfrac{\partial^{m_\sigma-2}(f\cdot V)}{\partial z^i\partial\overline{z}^{m_\sigma-2-i}}(0,0),
\end{align*}
and
\begin{align*}
\sum_{0\leq l<m_\sigma}\!\mathfrak{C}^{m_\sigma-2}_{(\sigma,2l+1)}(f) &=\!\sum_{0\leq i\leq m_\sigma-2}\!\Big(\!\sum_{0\leq l<m_\sigma}\!\theta_\sigma^{(2l+1)(2i+2-m_\sigma)}\Big)
\tbinom{m_\sigma-2}{i}\mathfrak{B}(\tfrac{(m_\sigma-1)-i}{m_\sigma},\tfrac{i+1}{m_\sigma})\tfrac{\partial^{m_\sigma-2}(f\cdot V)}{\partial z^i\partial\overline{z}^{m_\sigma-2-i}}(0,0).
\end{align*}
Note that
\[\sum_{0\leq l<m_\sigma}\theta_\sigma^{2l(2i+2-m_\sigma)}=\left\{
\begin{array}{cl}
m_\sigma&\text{if }i=\frac{m_\sigma-2}{2},\\
0 &\text{otherwise.}
\end{array}
\right.\]
It follows that $\Delta_\mathcal{O}(\widetilde{\varphi}_f)=0$ for every $\mathcal{O}\in\Sigma(\pi)$. As the IET $T:I\to I$ satisfies SUDC,
$\widetilde{\varphi}_f\in \mathrm{LG}(\sqcup_{\alpha\in \mathcal{A}} I_{\alpha})$, $\mathcal{L}(\widetilde{\varphi}_f)>0$, $\mathcal{AS}(\widetilde{\varphi}_f)=0$, $g'_{\widetilde{\varphi}_f}$ is piecewise constant and $\overline{d}(\widetilde{\varphi}_f)=0$, %
 by Theorem~\ref{thm:erg}, the skew product $T_{\widetilde{\varphi}_f}$ on $I\times \R$ is ergodic.
\end{proof}

\begin{proof}[Proof of Theorem~\ref{mainthm:onesaddle}]
To conclude the first part of the result, as in the proof of Theorem~\ref{mainthm:mini}, it is enough to show that the skew product map $T_{\varphi_f}$ is ergodic if $T$ is a.e.\ IET for which Theorem~\ref{thm:ergasym} holds. In fact, in the beginning, we will repeat the reasoning from that proof.
Only the final argument is completely different. Denote by $\sigma$ the only saddle of $\psi_\R$ on $M'$. Then $\gamma=1$.
Recall that the transversal $I \subseteq M'$ can be chosen so that $T=(\pi,\lambda) :I\to I$ is a symmetric IET exchanging $d=2g+1$ intervals.
As $\mathfrak{C}^{k}_{(\sigma,\ell)}(f)=0$, for all $\ell\in\mathcal{L}^\sigma_{\sim}$, $0\leq k<m_\sigma-2$, and $D_i(f)=0$, for $1\leq i\leq g$, in view of \eqref{eq:varphif}, we have
\[\varphi_f=\sum_{\ell\in\mathcal{L}^\sigma_{\sim}}\mathfrak{C}^{m_\sigma-2}_{(\sigma,\ell)}(f)\xi_{(\sigma,l,m_\sigma-2)}
+V(f)v+\mathfrak{r}(f).\]
Let
\[\widetilde{\varphi}_f=\sum_{\ell\in\mathcal{L}^\sigma_{\sim}}\mathfrak{C}^{m_\sigma-2}_{(\sigma,\ell)}(f)\xi_{(\sigma,l,m_\sigma-2)}
+V(f)v.\]
Since $\mathfrak{r}(f)$ is a coboundary, the skew products $T_{\varphi_f}$ and $T_{\widetilde{\varphi}_f}$ are isomorphic. Therefore, it is enough to show the ergodicity of $T_{\widetilde{\varphi}_f}$.
Note that $\widetilde{\varphi}_f\in \mathrm{LG}(\sqcup_{\alpha\in \mathcal{A}} I_{\alpha})$ and, by \eqref{eq:laxi} and \eqref{eq:lacv}, we have
\[\lim_{k\to\infty}\frac{1}{k}\log\Big(\frac{\|S(k)(\widetilde{\varphi}_f)\|_{L^1(I^{(k)})}}{|I^{(k)}|}\Big)\leq 0.\]
In view of Remark~\ref{rmk:defdi}, this gives $\overline{d}(\widetilde{\varphi}_f)=0$. %
Moreover,
\[\mathcal{L}(\widetilde{\varphi}_f)=\sum_{\alpha\in\mathcal{A}}(|C^-_{\alpha}(\widetilde{\varphi}_f)|+ |C^+_{\alpha}(\widetilde{\varphi}_f)|)=\tfrac{1}{m_\sigma^2 (m_\sigma-2)!}\sum_{\ell\in\mathcal{L}^\sigma_\sim}|\mathfrak{C}^{m_\sigma-2}_{(\sigma,\ell)}(f)|>0.\]
As
$\widetilde{\varphi}_f\in \mathrm{LG}(\sqcup_{\alpha\in \mathcal{A}} I_{\alpha})$, $\mathcal{L}(\widetilde{\varphi}_f)>0$, $g'_{\widetilde{\varphi}_f}$ is piecewise constant and $\overline{d}(\widetilde{\varphi}_f)=0$, %
 by Theorem~\ref{thm:ergasym}, the skew product $T_{\widetilde{\varphi}_f}$ on $I\times \R$ is ergodic, thus showing the ergodicity of $T_{{\varphi}_f}$ and $\psi^f_\R$ on $M'\times\R$.

To prove the second part of the result, and in view of \eqref{eqn;dev-complete}, it suffices to show that $err(f, T, x)$ has sub-polynomial growth and that it equidistributes on $\R$. The sub-polynomial oscillation follows directly from Theorem~\ref{thm:devspect}. Let us assume that $f:M\to\R$ is an arbitrary $C^{m_\sigma}$-map. As $\gamma=1$, in view of \eqref{eq:fe},
\[\varphi_{f_e}=\sum_{\ell\in\mathcal{L}^\sigma_{\sim}}\mathfrak{C}^{m_\sigma-2}_{(\sigma,\ell)}(f)\xi_{(\sigma,l,m_\sigma-2)}
+V(f)v+\mathfrak{r}(f).\]
Therefore, $\varphi_{f_e}$ has the same form as $\varphi_{f}$ in the first part of the proof. It follows that the skew product $T_{\varphi_{f_e}}$ and the skew product flow $\psi^{f_e}_\R$ on $M'\times \R$ is ergodic. Then
\[\psi^{f_e}_T(x,r)=\Big(\psi_Tx,r+\int_0^Tf_e(\psi_t x)\,dt\Big)=\big(\psi_Tx,r+err(f,T,x)\big).\]
We now can apply the ratio ergodic theorem to the ergodic flow $\psi^{f_e}_\R$ on $M'\times\R$ to obtain the equidistribution of the error term for a.e.\ $x\in M'$. For all the details of this step, we refer the reader to the proof of Theorem~1.3 in \cite{Fr-Ul2}.
\end{proof}

\section*{Acknowledgements}
\noindent This research was partially supported by the Narodowe Centrum Nauki Grant 2022/45/B/ST1/00179. The third author was supported by the Spanish State Research Agency through the Severo Ochoa and María de Maeztu Program for Centers and Units of Excellence in R$\&$D (CEX2020-001084-M) and by the European Union's Horizon 2020 research and innovation program under the Marie Skłodowska-Curie grant agreement No. 101154283.

\appendix
\section{Locally Hamiltonian flows with imperfect saddles}\label{sec:imper}
The main purpose of this section is to build a natural class of examples of ergodic locally Hamiltonian flows $\psi_\R$ for which $\varphi_f$ have singularities other than logarithmic and power.
More precisely, we find a natural family of functions $\theta$ satisfying \eqref{prop:theta tau0} for which $\varphi_f\in\Upsilon_\theta(\bigsqcup_{\alpha\in \mathcal{A}}
I_{\alpha})$. Finally, for this type of flows $\psi_\R$ which are of hyper-elliptic type (there exists an involution $\mathcal{I}$ such that $\psi_t\circ\mathcal{I}=\mathcal{I}\circ\psi_{-t}$), we construct anti-symmetric observables $f$ (i.e.\ $f\circ \mathcal{I}=-f$) for which the skew extension $\psi^f_\R$ is ergodic, using Theorem~\ref{thm:anti} for $\varphi_f\in\Upsilon_\theta(\bigsqcup_{\alpha\in \mathcal{A}}
I_{\alpha})$.

To implement this idea into practice, in Sections~\ref{sec:simple} and \ref{sec:impmulti} we consider some local Hamiltonians having degenerate saddle points that are not perfect. Such non-perfect saddles have pairs of tangent separatrices, unlike the perfect situation, see Figures~\ref{fig:RR} and \ref{fig:QQ}. In Section~\ref{sec:simple}, we perform an analysis of transition times near simple imperfect saddles, with only four separatrices, see Figure~\ref {fig:RR}. Using the results obtained in Section~\ref{sec:impmulti}, we examine imperfect multi-saddles. These, in turn, are used to construct hyper-elliptic locally Hamiltonian flows with saddle loops and their ergodic skew extensions in Section~\ref{sec:impfinal}.

\subsection{Degenerate simple saddles}\label{sec:simple}
Let $H:[-1,1]^2\to\R$ be a $C^2$-hamiltonian such that $H(x,y)=-H_1(x)+H_2(y)$ and $H_1,H_2:[-1,1]\to\R_{\geq 0}$ are unimodal functions with the minimum at zero, i.e.
\[H_1(0)=H_2(0)=0,\quad H'_1(0)=H'_2(0)=0,\text{ and }xH_1'(x)>0,\quad xH_2'(x)>0\text{ for }x\neq 0.\]
Most often, we will assume that $H_1,H_2$ are even, i.e.
\begin{equation}\label{eq: symmetricH}
	 H_1(-x)=H_1(x)\ \text{and}\ H_2(-x)=H_2(x).
	 \end{equation}
Then, the corresponding Hamiltonian equation is of the form
\[x'=H_2'(y),\quad y'=H_1'(x).\]
We denote the corresponding local Hamiltonian flow on $[-1,1]^2$ by $(\Psi_t)$. It has four separatrixes, which together form the set of solutions to the equation $H_1(x)=H_2(y)$. The set of solutions is the union of two graphs of the maps $H_{2,+}^{-1}\circ H_1$ and $H_{2,-}^{-1}\circ H_1$, where $H_{i,\pm}^{-1}$ are two inverse branches of $H_i$ for $i=1,2$. Both maps $H_{2,\pm}^{-1}\circ H_1$ are continuous, and $C^2$ except perhaps zero. We will deal also with level sets $L_s:=\{(x,y)\in [-1,1]^2:-H_1(x)+H_2(y)=s\}$ for $s\in\R\setminus \{0\}$. Each of them is the union of two disjoint $C^2$-curves. Let $\gamma$ be a compact smooth transversal curve. Let us consider its parametrization $\gamma:[a,b]\to \R^2$ such that $H(\gamma(s))=s$ for $s\in[a,b]$. This parametrization is of class $C^2$ and is standard in the sense that $\int_{\gamma([a,s])}\eta=s-a$, where $\eta=\frac{\partial H}{\partial x}dx+\frac{\partial H}{\partial y}dy$. Indeed,
\begin{gather*}
\int_{\gamma([a,s])}\eta=\int_a^s\Big(\frac{\partial H}{\partial x}(\gamma(t))\gamma'_1(t)+\frac{\partial H}{\partial y}(\gamma(t))\gamma'_2(t)\Big)dt\\
=\int_a^s\frac{d}{dt}H(\gamma(t))dt=H(\gamma(s))-H(\gamma(a))=s-a.
\end{gather*}

\begin{figure}[h!]
 \includegraphics[width=0.4\textwidth]{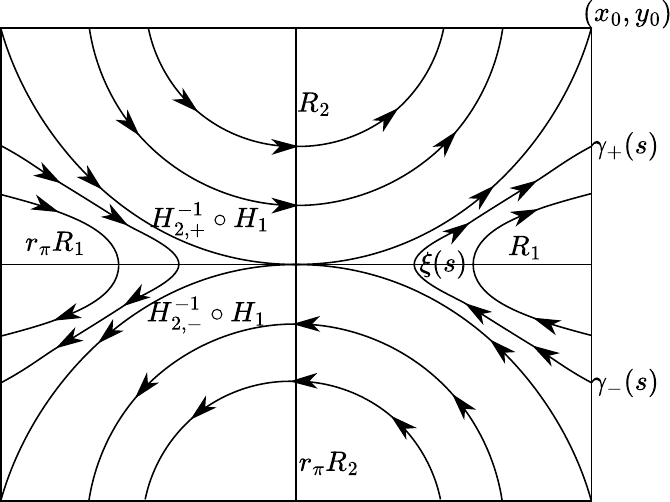}
 \caption{The phase portrait of the local Hamiltonian on the rectangle $R$. {The map $r_{\pi}$ is the rotation by $\pi$. The symmetry of the picture is due to \eqref{eq: symmetricH}.} } \label{fig:RR}
\end{figure}

For any $0<x_0\leq 1$ and $0<\vep<s_0:=H_1(x_0)$ (we will constantly assume that $s_0=H_1(x_0)\leq 1$), let $\gamma^+,\gamma^-$ are two curves (lines) given by
\begin{equation}\label{def:g+-}
\gamma^+(s)=(x_0,H_{2,+}^{-1}(H_1(x_0)-s))\text{ and }\gamma^-(s)=(x_0,H_{2,-}^{-1}(H_1(x_0)-s))\text{ for }s\in [0,\vep].
\end{equation}
Then,
\begin{equation}\label{eq:stand}
H(\gamma^\pm(s))=-H_1(x_0)+H_2(H_{2,\pm}^{-1}(H_1(x_0)-s))=-s,
\end{equation}
so their reverse parametrizations are standard. Since $\gamma^+(s),\gamma^-(s)\in L_{-s}$ for $s\in(0,\vep]$, they lay on the same orbit of $(\Psi_t)$. We denote the time it takes for the flow to get from $\gamma^-(s)$ to $\gamma^+(s)$ by $r(s)>0$. Since the orbit starts from the lower half-plane (from $\gamma^-(s)$), ends in the upper half-plane (at $\gamma^+(s)$), and goes up all the time, there is exactly one its intersection point with the horizontal axis, that is $(\xi(s),0)=\Psi_{t(s)}(\gamma^-(s))$ with $0<t(s)<r(s)$. Moreover,
\[-H_1(\xi(s))=H(\xi(s),0)=H(\gamma^-(s))=-s,\text{ so }\xi(s)=H^{-1}_{1,+}(s).\]

Let $f:\R^2\to\R$ be a smooth function. We are interested in studying the behavior of ergodic integrals
\[\varphi_f(s)=\int_0^{r(s)}f(\Psi_t(\gamma^-(s)))dt=\int_0^{t(s)}f(\Psi_t(\gamma^-(s)))dt+\int_{t(s)}^{r(s)}f(\Psi_t(\gamma^-(s)))dt\]
when $s$ converges to zero. Let $(x(t),y(t))=\Psi_t(\gamma^-(s))$. Then $H_1(x(t))-H_2(y(t))=s$, so
\[y(t)=\left\{
\begin{array}{rl}
H_{2,-}^{-1}(H_1(x(t))-s)&\text{for }t\in[0,t(s)]\\
H_{2,+}^{-1}(H_1(x(t))-s)&\text{for }t\in[t(s),r(s)].
\end{array}
\right.\]
Moreover, $x'(t)=H'_2(y(t))=H'_2(H_{2,+}^{-1}(H_1(x(t))-s))$. Substituting $u= u(t)=H_1(x(t))$, for $t\in[0,t(s)]$, we get
\[du=H'_1(x(t))x'(t)dt=H'_1(H^{-1}_{1,+}(u))H'_2(H_{2,-}^{-1}(u-s))dt\]
with
\[u(0)=H_1(x(0))=H_1(\gamma^-_1(s))=H_1(x_0)=s_0\text{ and }u(t(s))=H_1(x(t(s)))=H_1(\xi(s))=s,\]
and for $t\in[t(s),r(s)]$, we get
\[du=H'_1(x(t))x'(t)dt=H'_1(H^{-1}_{1,+}(u))H'_2(H_{2,+}^{-1}(u-s))dt\]
with
\[u(r(s))=H_1(x(r(s)))=H_1(\gamma^+_1(s))=H_1(x_0)=s_0\text{ and }u(t(s))=s.\]
Therefore, integrating by substitution, we have
\begin{align*}
\int_0^{t(s)}f(\Psi_t(\gamma^-(s)))dt&=-\int_s^{s_0}\frac{f(H^{-1}_{1,+}(u),H_{2,-}^{-1}(u-s))}{H'_1(H^{-1}_{1,+}(u))H'_2(H_{2,-}^{-1}(u-s))}du\\
\int_{t(s)}^{r(s)}f(\Psi_t(\gamma^-(s)))dt&=\int_s^{s_0}\frac{f(H^{-1}_{1,+}(u),H_{2,+}^{-1}(u-s))}{H'_1(H^{-1}_{1,+}(u))H'_2(H_{2,+}^{-1}(u-s))}du.
\end{align*}
Recall that $H_1$ and $H_2$ are even. Then, $-H^{-1}_{i,-}=H^{-1}_{i,+}=:H^{-1}_{i}$ and
\begin{equation}\label{eq:phif}
\varphi_f(s)=\int_s^{s_0}\frac{f(H^{-1}_{1}(u),H_{2}^{-1}(u-s))+f(H^{-1}_{1}(u),-H_{2}^{-1}(u-s))}{H'_1(H^{-1}_{1}(u))H'_2(H_{2}^{-1}(u-s))}du.
\end{equation}

Let $g:\R_{\geq 0}\to\R_{> 0}$ be a $C^1$-map such that
\begin{equation}\label{prop:g}
g\text{ is increasing and concave with}\quad \lim_{x\to+\infty}g(x)=+\infty\text{ and }\int_2^\infty\frac{1}{xG(\log x)}\,dx=+\infty,
\end{equation}
where $G$ is the primitive of $g$ with $G(0)=0$. {Observe that $g(x)=\log(x)$ satisfies \eqref{prop:g}.}

\begin{remark}
Note that
\begin{equation}\label{eq:ga}
\lim_{x\to+\infty}\frac{g(x+a)}{g(x)}=1\text{ for any }a>1.
\end{equation}
Indeed, by \eqref{prop:g}, $\frac{d}{dx}\log g(x)=\frac{g'(x)}{g(x)}\to 0$ as $x\to+\infty$. Therefore, for any $a>0$,
\[\log 1\leq\liminf_{x\to+\infty}\log \frac{g(x+a)}{g(x)}\leq \limsup_{x\to+\infty}\log \frac{g(x+a)}{g(x)}\leq\limsup_{x\to+\infty}a\frac{d}{dx}\log g(x)=0,\]
which gives \eqref{eq:ga}.
\end{remark}

\begin{lemma}
Let $k\geq 1$. For any $C^k$-map $g:\R_{\geq 0}\to\R_{> 0}$ satisfying \eqref{prop:g}, there exists an even $C^2$-map $H:[-x_0,x_0]\to[0,s_0]$ with $0<x_0\leq 1$ and $s_0:=H(x_0)\leq 1$, such that
\[(H^{-1})'(u)=\frac{g(-\log u)}{\sqrt{u}}\text{ for }u\in(0,s_0], \text{ and } H(0)= H'(0)=H''(0)=0.\]
Moreover, $H$ is of class $C^{k+1}$ on $(0,x_0]$.
\end{lemma}

\begin{proof}
Let us start by observing that, by the concavity of $g$, the function $\frac{g(-\log u)}{\sqrt{u}}=O( \frac{-\log u}{\sqrt{u}})$ is integrable. Let us define the inverse $H^{-1}:[0,s_0]\to[0,x_0]$, which is given by
\[H^{-1}(x)=\int_0^x \frac{g(-\log u)du}{\sqrt{u}}\]
which is a continuous increasing function and is $C^{k+1}$ on $(0,s_0]$ {and $H(0)=0$}. More precisely, we choose $0<s_0\leq 1$ such that the integral $x_0=H^{-1}(s_0)$ is at most $1$.
Then, its inverse can be extended to an even continuous function $H:[-x_0,x_0]\to [0,s_0]$ such that $H$ is of class $C^{k+1}$ on $(0,x_0]$, $H(0)=0$, and $(H)'(x)=\frac{\sqrt{H(x)}}{g(-\log H(x))}$ for $x\in (0,x_0]$.
To show that $H$ on $[-1,1]$ is class $C^2$, it is enough to show that
\[\lim_{x\to 0}H'(x)=\lim_{x\to 0}H''(x)=0.\]
As $H(0)=0$, by \eqref{prop:g}, we have
\[(H)'(x)=\frac{\sqrt{H(x)}}{g(-\log H(x))}\to 0\text{ as }x\to 0,\]
and
\begin{align*}
(H)''(x)&=\frac{(H)'(x)}{2\sqrt{H(x)}g(-\log H(x))}+\frac{(H)'(x)g'(-\log H(x))}{\sqrt{H(x)}g^2(-\log H(x))}\\
&=\frac{1}{2g^2(-\log H(x))}+\frac{g'(-\log H(x))}{g^3(-\log H(x))}\to 0\text{ as }x\to 0.
\end{align*}
\end{proof}

{We now provide a construction of the Hamiltonian in the simple saddle case, where one of the components $H_i$ of the Hamiltonian is simply a square, while the other is a perturbation of a square so that finally, the ergodic integral obtained via approaching the saddle, is of the form for which the Theorem \ref{mainthm:mini} holds in the simple case. Moreover, this construction will also be crucial in the general (non-simple saddle) case, where the picture seen on some of the pairs of neighboring sectors can be viewed as a pre-image of the simple saddle picture through the map $z\mapsto z^m$, where $2m$ is the multiplicity of the saddle. The cases depend on the coordinate on which we see the perfect square. It is worth mentioning that in the final construction, we use Case 2. However, Case 1 (depicted in Figure \ref{fig:RR}) corresponds to the picture seen in Case 2 rotated by $\tfrac{\pi}{2}$ and allows for some simplification in the computation.}

\noindent
\textbf{Case 1.}
Suppose that $H_2(y)=y^2$, $H_1$ is an even $C^2$-map such that $(H^{-1}_{1})'(u)=\frac{g(-\log u)}{\sqrt{u}}$ for $u\in(0,H_1(1)]$, and $f\equiv 1$.
Then, by \eqref{eq:phif}, we have
\[\varphi_1(s)=\int_s^{s_0}\frac{g(-\log u)}{\sqrt{u}\sqrt{u-s}}du=\int_1^{\frac{s_0}{s}}\frac{g(-\log su)}{\sqrt{u}\sqrt{u-1}}du\]
and, {by the Leibniz integration formula, we have}
\[\varphi'_1(s)=-\frac{s_0}{s}\frac{g(-\log s_0)}{\sqrt{s_0}\sqrt{s_0-s}}-\int_1^{\frac{s_0}{s}}\frac{g'(-\log su)}{s\sqrt{u}\sqrt{u-1}}du.\]
Moreover,
\begin{align*}
-s\varphi'_1(s)&=\frac{g(-\log s_0)\sqrt{s_0}}{\sqrt{s_0-s}}+\int_1^{\frac{s_0}{s}}\frac{g'(-\log su)}{\sqrt{u}\sqrt{u-1}}du\\
&\geq g(-\log s_0)+\int_1^{\frac{s_0}{s}}\frac{g'(-\log su)}{u}du=g(-\log s),
\end{align*}
and for any $n\geq 2$ and any $s\in(0,s_0/n]$, we have
\begin{align*}
-s\varphi'_1(s)&=\frac{g(-\log s_0)\sqrt{s_0}}{\sqrt{s_0-s}}+\int_1^{\frac{s_0}{s}}\frac{g'(-\log su)}{\sqrt{u}\sqrt{u-1}}du\\
&\leq \sqrt{\frac{n}{n-1}}g(-\log s_0)+\int_1^{n}\frac{g'(-\log su)}{\sqrt{u}\sqrt{u-1}}du+\sqrt{\frac{n}{n-1}}\int_n^{\frac{s_0}{s}}\frac{g'(-\log su)}{u}du\\
&\leq \sqrt{\frac{n}{n-1}}g(-\log s)+\int_1^{n}\frac{g'(0)}{\sqrt{u}\sqrt{u-1}}du.
\end{align*}
It follows that there exists $C>1$ such that
\begin{gather}\label{eq:ups1}
\frac{g(-\log s)}{s}\leq |\varphi'_1(s)|\leq C\frac{g(-\log s)}{s}\text{ for }s\in(0,s_0/2],\\
\nonumber 1\leq\liminf_{s\to 0}\frac{-\varphi'_1(s)}{\frac{g(-\log s)}{s}}\leq \limsup_{s\to 0}\frac{-\varphi'_1(s)}{\frac{g(-\log s)}{s}}\leq\sqrt{\frac{n}{n-1}}.
\end{gather}
By passing with $n$ to infinity, we get
\begin{equation}\label{eq:ups2}
\lim_{s\to 0}{\frac{s}{g(-\log s)}{\varphi'_1(s)}}=-1.
\end{equation}

\begin{remark}
Note that, in this case, the map $H^{-1}_2\circ H_1=\sqrt{H_1}:[-x_0,x_0]\to\R$ is of class $C^1$ with the derivative at zero equal zero. Indeed, this follows from
\[(\sqrt{H_1})'(x)=\frac{(H_1)'(x)}{2\sqrt{H_1(x)}}=\frac{1}{2g(-\log H_1(x))}\to 0\text{ as }x\to 0.\]
Therefore, the level set $L_0$ is the union of the graphs of two $C^1$-maps $\pm\sqrt{H_1}$, and the phase portrait of the local Hamiltonian flow is as shown in Figure~\ref{fig:RR}.
\end{remark}
\medskip

\noindent
\textbf{Case 2.} Suppose that $H_1(x)=x^2$, $H_2$ is an even $C^2$-map such that $(H^{-1}_{2})'(u)=\frac{g(-\log u)}{\sqrt{u}}$, and $f\equiv 1$. Then,
\[\varphi_1(s)=\int_s^{s_0}\frac{g(-\log (u-s))}{\sqrt{u}\sqrt{u-s}}du=\int_1^{\frac{s_0}{s}}\frac{g(-\log s(u-1))}{\sqrt{u}\sqrt{u-1}}du,\]
and, for any $n\geq 2$ and any $s\in(0,s_0/n]$, {the Leibniz integration formula yields}
\begin{align*}
-s\varphi'_1(s)&=\frac{g(-\log (s_0-s))\sqrt{s_0}}{\sqrt{s_0-s}}+\int_1^{\frac{s_0}{s}}\frac{g'(-\log s(u-1))}{\sqrt{u}\sqrt{u-1}}du\\
&=\frac{g(-\log (s_0-s))\sqrt{s_0}}{\sqrt{s_0-s}}+\int_1^{n}\frac{g'(-\log s(u-1))}{\sqrt{u}\sqrt{u-1}}du
+\int_{n-1}^{\frac{s_0}{s}-1}\frac{g'(-\log su)}{\sqrt{u+1}\sqrt{u}}du.
\end{align*}
Moreover, {since $s\le s_0/n$,} we have
\begin{gather*}
0\leq\frac{g(-\log (s_0-s))\sqrt{s_0}}{\sqrt{s_0-s}}\leq\sqrt{\frac{n}{n-1}}g(-\log (s_0(1-\tfrac{1}{n})))=:C_n,\\
0\leq \int_1^{n}\frac{g'(-\log s(u-1))}{\sqrt{u}\sqrt{u-1}}du\leq \int_1^{n}\frac{g'(0)}{\sqrt{u}\sqrt{u-1}}du=:D_n<+\infty.
\end{gather*}
It follows that for $s\in(0,s_0/2]$, we have
\begin{align*}
-s\varphi'_1(s)&
\leq C_2+ D_2+\int_1^{\frac{s_0}{s}-1}\frac{g'(-\log su)}{\sqrt{u+1}\sqrt{u}}du \\
&\leq
C_2+ D_2+\int_1^{\frac{s_0}{s}}\frac{g'(-\log su)}{u}du\leq C_2+ D_2+g(-\log s),
\end{align*}
and any $s\in(0,s_0/n]$, we have
\begin{align*}
-s\varphi'_1(s)&
\geq \int_{n-1}^{\frac{s_0}{s}-1}\frac{g'(-\log su)}{\sqrt{u+1}\sqrt{u}}du\geq \sqrt{1-\frac{1}{n}}\int_{n-1}^{\frac{s_0}{s}(1-\frac{1}{n})}\frac{g'(-\log su)}{u}du\\
&\geq \sqrt{1-\frac{1}{n}}(g(-\log (n-1)s)-g(-\log (s_0(1-\tfrac{1}{n}))).
\end{align*}
It follows that there are $0<c<C$ such that
\begin{gather}\label{eq:ups3}
c\frac{g(-\log s)}{s}\leq |\varphi'_1(s)|\leq C\frac{g(-\log s)}{s}\text{ for }s\in (0,s_0/4],\\
\nonumber \sqrt{1-\frac{1}{n}}\leq\liminf_{s\to 0}\frac{-\varphi'_1(s)}{\frac{g(-\log (n-1)s)}{s}}, \text{ and } \limsup_{s\to 0}\frac{-\varphi'_1(s)}{\frac{g(-\log s)}{s}}\leq 1.
\end{gather}
Using \eqref{eq:ga}, by passing with $n$ to infinity, we get
\begin{equation}\label{eq:ups4}
\lim_{s\to 0}{\frac{s}{g(-\log s)}{\varphi'_1(s)}}=-1.
\end{equation}

\medskip

Let us consider two maps $\tau:(0,1]\to\R_{\geq 0}$ given by $\tau(s)=\frac{s}{g(-\log s)}$, and $\theta:[1,+\infty)\to\R_{\geq 0}$ given by
\[\theta(x)=\int_1^x\frac{du}{u^2\tau(1/u)}=\int_1^x\frac{g(\log u)}{u}du=\int_0^{\log x}g(u)du=:G(\log x).\]

\begin{proposition}
Let $g:\R_{\geq 0}\to \R_{>0}$ be a $C^1$-map satisfying \eqref{prop:g}. Then, $\theta$ and $\tau$ satisfy \eqref{prop:theta tau}, and $\theta$ is slowly varying.

Suppose that $H:[-1,1]^2\to \R$ be a $C^2$-map such that $H(x,y)=H_1(x)+H_2(y)$ with
\begin{itemize}
\item $H_1(x)=x^2$, $H_2$ is an even $C^2$-map such that $(H^{-1}_{2})'(u)=\frac{g(-\log u)}{\sqrt{u}}$, or
\item $H_2(y)=y^2$, $H_1$ is an even $C^2$-map such that $(H^{-1}_{1})'(u)=\frac{g(-\log u)}{\sqrt{u}}$.
\end{itemize}
Then, $\varphi_1\in {C^1((0,s_0])}$%
and there are $0<c<C$ such that
\begin{equation}\label{eq:ups}
c\leq\tau(s)|\varphi_1'(s)|\leq C\text{ for }s\in(0,s_0/4]\text{ with }\lim_{s\to 0}\tau(s)\varphi_1'(s)=-1.
\end{equation}
\end{proposition}

\begin{proof}
By definition, $\theta$ is increasing and
\[\tau'(s)=\frac{1}{g(-\log s)}+\frac{g'(-\log s)}{g^2(-\log s)}>0,\]
so $\tau$ is also increasing. By \eqref{prop:g}, we also have $\int_{2}^{+\infty}\frac{dx}{x\theta(x)}=+\infty$, so $\theta$ and $\tau$ satisfy \eqref{prop:theta tau}.

In view of \eqref{eq:ga}, for any $\alpha>1$, we have
\[\lim_{x\to+\infty}\frac{\theta(\alpha x)}{\theta(x)}=\alpha\lim_{x\to+\infty}\frac{\theta'(\alpha x)}{\theta'(x)}=\lim_{x\to+\infty}\frac{g(x+\log \alpha)}{g(x)}=1,\]
the map $\theta$ is also slowly varying.

Finally, \eqref{eq:ups} (in particular, $\varphi_1\in \Upsilon_\theta((0,s_0])$) follows directly from \eqref{eq:ups1}, \eqref{eq:ups2}, \eqref{eq:ups3}, and \eqref{eq:ups4}.
\end{proof}

\subsection{Imperfect multi-saddles}\label{sec:impmulti}
{We are now ready to consider flows with saddles of higher multiplicities. A large part of what follows was inspired by techniques used in \cite{Fr-Ki2}, where a very thorough local analysis of flow with perfect saddles was performed. }

Let us consider the standard Hamiltonian equation on the complex plane $z'=-2i\frac{\partial H}{\partial\bar z}(z,\bar{z})$.
Suppose that $\widetilde{H}:[-1,1]^2\to\R$ is a $C^2$-Hamiltonian which is $C^\infty$ outside $(0,0)$. Later on, it will simply be the Hamiltonian considered in Section~\ref{sec:simple}. Let $m\geq 2$. We denote by $\mathcal{D}=\mathcal{D}_m$ the pre-image of the square $[-1,1]\times[-1,1]$ by the map $\C\ni\omega\mapsto \omega^m\in\C$. Let $H:\mathcal{D}\to\R$ be the new Hamiltonian given by $H(z,\bar{z})=\widetilde H(z^m,\bar{z}^m)$. The map $H$ is of class $C^{2m}$. We analyze the dynamics of trajectories of the Hamiltonian flow derived from the equation $z'=-2i\frac{\partial H}{\partial\bar z}(z,\bar{z})$ by changing variables $\omega=z^m$ and comparing to the flow associated with Hamiltonian $\widetilde H$. We study our new local flow separately on each angular sector of $\mathcal{D}$, which is the image of $[-1,1]^2$ through any branch of the $m$-th root.

Let $G_0:\C\to\C$ be the principal branch of the $m$-th root $G_0(re^{\iota t})=r^{1/m}e^{\iota t/m}$ if $t\in[-\pi,\pi$) and let $\zeta$
be the principal $2m$-th
root of unity.
For any $0\leq l< 2m$ let $G_l:\C\to\C$ be given by $G_l=\zeta^lG_0$.

As $z'=-2i\frac{\partial H}{\partial\bar z}(z,\bar{z})$, in new coordinates $\omega=z^m$, we have
\begin{align}\label{eq:passflows}
\begin{aligned}
\omega'&=mz^{m-1}z'=-2imz^{m-1}\frac{\partial H}{\partial\bar z}(z,\bar{z})=-2imz^{m-1}\frac{d}{d\bar z}\widetilde H(z^m,\bar{z}^m)\\
&=-2im^2|z|^{2(m-1)}\frac{\partial \widetilde H}{\partial\bar \omega}(z^m,\bar{z}^m)
=-2im^2|\omega|^{2\frac{m-1}{m}}\frac{\partial \widetilde H}{\partial\bar \omega}(\omega,\bar{\omega}).
\end{aligned}
\end{align}
Suppose that $\widetilde H(x,y)=-H_1(x)+H_2(y)$, where $H_1,H_2$ are even $C^2$-maps as in Section~\ref{sec:simple}. Choose $x_0,y_0,s_0\in(0,1]$ such that $H_1(x_0)=H_2(y_0)=s_0$ and let $R= R(H_1,H_2):=[-x_0,x_0]\times[-y_0,y_0]\subseteq[-1,1]^2$. Let
\begin{align*}
R_1=R_1(H_1,H_2):&=\{(x,y)\in(0,x_0]\times [-y_0,y_0]:H_1(x)<H_2(y)\},\\
R_2:=R_1(H_2,H_1):&=\{(x,y)\in[-x_0,x_0]\times (0,y_0]:H_1(x)>H_2(y)\}.
\end{align*}
For any $\alpha\in \R$ denote by $r_\alpha:\C\to\C$ the rotation $r_\alpha z=e^{i\alpha}z$.
Then the sets $R_1, R_2, r_{\pi}R_1, r_{\pi}R_2$ give a splitting of $R$ (after removing separatrices) into invariant sets for the local Hamiltonian flow
 $(\widetilde \Psi_t)$ restricted to $R$ and related to the Hamiltonian equation $\omega'=-2i\frac{\partial \widetilde H}{\partial\bar \omega}(\omega,\bar{\omega})$, see Figure~\ref{fig:RR}.

Let $\mathcal{Q}=\mathcal{Q}(s_0) \subseteq \C$ be the pre-image of the rectangle $R$ by the map $\C\ni\omega\mapsto \omega^m\in\C$.
We denote the (local) Hamiltonian flow on $\mathcal{Q}$ related to the Hamiltonian $H$ by $(\Psi_t)$. Let us consider the family of sets $\{\mathcal{Q}_j:1\leq j\leq 4m\}$ given by
\[\mathcal{Q}_{2j+1}:=G_j R_1,\quad \mathcal{Q}_{2j+2}:=G_j R_2\quad\text{for}\quad 0\leq j<2m.\]
Then $\{\mathcal{Q}_j:1\leq j\leq 4m\}$ gives a splitting of $\mathcal{Q}$ (after removing separatrices) into, roughly speaking, invariant sets for the local Hamiltonian flow
 $(\Psi_t)$.

Let $f:\mathcal Q\to\R$ be any smooth map. In this section, we are interested in understanding the singularity of the function that counts ergodic integrals of $f$ from the entrance to the exit of the set $\mathcal{Q}_j$ as an orbit approaches separatrices. We first focus on the set $\mathcal{Q}_1$, but as we will see, this is enough to understand the singularities for all sets $\mathcal{Q}_j$.

Note that $G_0:R_1\to \mathcal{Q}_1$ is an analytic diffeomorphism. We denote by $(\widehat \Psi_t)$ the local flow on $R_1$ conjugated by $G_0$ to $(\Psi_t)$. In view of \eqref{eq:passflows}, $(\widehat \Psi_t)$ is the solution to the equation $\omega'=-2im^2|\omega|^{2\frac{m-1}{m}}\frac{\partial \widetilde H}{\partial\bar \omega}(\omega,\bar{\omega})$, so is a time change of the flow
 $(\widetilde \Psi_t)$ related to the equation $\omega'=-2i\frac{\partial \widetilde H}{\partial\bar \omega}(\omega,\bar{\omega})$.

 Let us consider two curves in $\mathcal{Q}_1$, $G_0\circ \gamma^{\pm}:[0,s_0]\to \mathcal{Q}_1$, where $\gamma^\pm$ are defined in \eqref{def:g+-}. In view of \eqref{eq:stand},
 \[H(G_0\circ\gamma^\pm(s))=\widetilde H((G_0\circ\gamma^\pm(s))^m)=\widetilde H(\gamma^\pm(s))=-s,\]
 so the inverse of $G_0\circ \gamma^{\pm}$ are standard for the Hamiltonian $H$. We denote the time it takes for the flow $(\Psi_t)$ to get from $G_0\circ \gamma^-(s)$ to $G_0\circ\gamma^+(s)$ by $r(s)>0$.
 Then, $r(s)$ is also the time of passing from $\gamma^-(s)$ to $\gamma^+(s)$ for the flow $(\widehat \Psi_t)$. Let $\tilde{r}(s)$ be the time of passing from $\gamma^-(s)$ to $\gamma^+(s)$ for the flow $(\widetilde \Psi_t)$.
 As $(\widehat \Psi_t)$ is a time change of $(\widetilde \Psi_t)$, in view of \eqref{eq:passflows}, for any $f:\mathcal{D}\to\R$, we have
 \begin{equation*}
\varphi^1_f(s):=\int_0^{r(s)}f(\Psi_t(G_0\circ \gamma^-(s)))dt=\int_0^{r(s)}f(G_0\widehat \Psi_t(\gamma^-(s))))dt=\int_0^{\tilde r(s)}\frac{f(G_0\widetilde \Psi_t(\gamma^-(s))))}{m^2|\widetilde \Psi_t(\gamma^-(s))|^{2\frac{m-1}{m}}}dt.
 \end{equation*}
In view of \eqref{eq:phif}, this gives
\begin{equation}\label{eq:phifmulti}
\varphi^1_f(s)=\int_s^{s_0}\frac{f\circ G_0(H^{-1}_{1}(u),H_{2}^{-1}(u-s))+f\circ G_0(H^{-1}_{1}(u),-H_{2}^{-1}(u-s))}{m^2((H^{-1}_{1}(u))^2+(H_{2}^{-1}(u-s))^2)^{\frac{m-1}{m}}H'_1(H^{-1}_{1}(u))H'_2(H_{2}^{-1}(u-s))}du.
\end{equation}

\begin{remark}
The map $\varphi^1_f:(0,s_0]\to\R_{>0}$ counts the ergodic integrals of $f$ from the entrance to the exit of the set $\mathcal{Q}_1$ for the flow $(\Psi_t)$. Similarly, for any $1\leq j\leq 4m$, we define the map $\varphi^j_f:(0,s_0]\to\R_{>0}$ that counts the ergodic integrals of $f$ from the entrance to the exit of the set $\mathcal{Q}_j$. As $\mathcal{Q}_2=G_0 R_2(H_1,H_2)=G_0 R_1(H_2,H_1)$, to determine the map $\varphi^2_f$, it is enough to swap the roles of $H_1$ and $H_2$ in the formula \eqref{eq:phifmulti}. More precisely
\begin{equation}\label{eq:phifmulti2}
\varphi^2_f(s)=\int_s^{s_0}\frac{f\circ G_0(H^{-1}_{2}(u),H_{1}^{-1}(u-s))+f\circ G_0(H^{-1}_{2}(u),-H_{1}^{-1}(u-s))}{m^2((H^{-1}_{2}(u))^2+(H_{1}^{-1}(u-s))^2)^{\frac{m-1}{m}}H'_2(H^{-1}_{2}(u))H'_1(H_{1}^{-1}(u-s))}du.
\end{equation}
As $H_1$ and $H_2$ are even, the Hamiltonian $H(z,\bar{z})=\widetilde{H}(z^m,\bar{z}^m)$ is invariant under the rotation by $\zeta^j$. Therefore, for any $0\leq j<2m$ the flow $(\Psi_t)$ restricted to $\mathcal{Q}_{2j+1}$ (or $\mathcal{Q}_{2j+2}$) is conjugated by the rotation by $\zeta^j$ to the flow $(\Psi_t)$ restricted to $\mathcal{Q}_{1}$ (or $\mathcal{Q}_{2}$ resp.). It follows that
\begin{equation}\label{eq:phifmulti3}
\varphi^{2j+1}_f=\varphi^{1}_{f\circ \zeta^{-j}}\quad \text{and}\quad \varphi^{2j+2}_f=\varphi^{2}_{f\circ \zeta^{-j}}\quad \text{for any}\quad 0\leq j<2m.
\end{equation}
\end{remark}

\subsection*{Singularity analysis for $\varphi^{1}_f$}
In the remainder of this section, we will write $\varphi_f$ instead of $\varphi^{1}_f$ to simplify notation.
Let $f(z,\bar z)=f_{k,l}(z,\bar z)=z^k\bar z^l$ on $\mathcal Q$. Then, by \eqref{eq:phifmulti},
\begin{gather*}
\varphi_{f_{k,l}\circ \zeta^j}(s)
=\frac{2\zeta^{j(k-l)}}{m^2}\Re\int_s^{s_0}
\frac{\big(G_0(H^{-1}_{1}(u)+iH_{2}^{-1}(u-s))\big)^k\big(G_0(H^{-1}_{1}(u)-iH_{2}^{-1}(u-s))\big)^l}{((H^{-1}_{1}(u))^2+(H_{2}^{-1}(u-s))^2)^{\frac{m-1}{m}}H'_1(H^{-1}_{1}(u))H'_2(H_{2}^{-1}(u-s))}du\\
=\frac{2\zeta^{j(k-l)}}{m^2}\Re\int_s^{s_0}
\frac{\big(G_0(H^{-1}_{1}(u)+iH_{2}^{-1}(u-s))\big)^{k-(m-1)}}{\big(G_0(H^{-1}_{1}(u)-iH_{2}^{-1}(u-s))\big)^{(m-1)-l}}\frac{1}{H'_1(H^{-1}_{1}(u))H'_2(H_{2}^{-1}(u-s))}du.
\end{gather*}
We will deal with the case where $k+l=2(m-1)$. Then
\begin{align}\label{eq:kl}
\varphi_{f_{k,l}\circ \zeta^j}(s)
&=\frac{2\zeta^{j(k-l)}}{m^2}\Re\int_s^{s_0}
\frac{\Big(\frac{G_0(H^{-1}_{1}(u)+iH_{2}^{-1}(u-s))}{G_0(H^{-1}_{1}(u)-iH_{2}^{-1}(u-s))}\Big)^{k-(m-1)}}{H'_1(H^{-1}_{1}(u))H'_2(H_{2}^{-1}(u-s))}du.
\end{align}
Suppose that $g:\R_{\geq 0}\to\R_{> 0}$ is a $C^1$-map satisfying \eqref{prop:g} and such that $g(0)=1$. Recall that $\tau:(0,1]\to\R_{\geq 0}$ is given by $\tau(s)=\frac{s}{g(-\log s)}$

\begin{lemma}\label{lem:deg1}
Suppose that $H_2(y)=y^2$ and $H_1$ is an even $C^2$-map such that $(H^{-1}_{1})'(u)=\frac{g(-\log u)}{\sqrt{u}}>0$. Let $k,l\in\Z_{\geq 0}$ such that $k+l=2(m-1)$. Then $\varphi_{f_{k,l}\circ\zeta^j}:(0,s_0]\to \C$ is a $C^2$-map such that
\begin{equation}\label{eq:fkl1}
\varphi'_{f_{k,l}\circ\zeta^j}(s)=O(\tau(s)^{-1})\quad\text{and}\quad \lim_{s\to 0}\tau(s)\varphi'_{f_{k,l}}(s)=- \frac{\zeta^{2j(k-(m-1))}}{m^2}.
\end{equation}
\end{lemma}

\begin{proof}
As $k+l=2(m-1)$, by \eqref{eq:kl}, we have
\[\varphi_{f_{k,l}\circ\zeta^j}(s)-\zeta^{j(k-l)}\varphi_{f_{m-1,m-1}}(s)=\int_s^{s_0}\frac{2\psi\left(\frac{H_{2}^{-1}(u-s)}{H^{-1}_{1}(u)}\right)\frac{H_{2}^{-1}(u-s)}{H^{-1}_{1}(u)}}{H'_1(H^{-1}_{1}(u))H'_2(H_{2}^{-1}(u-s))}du,\]
where $\psi:\R\to\C$ is an analytic map given by
\[\psi(x)=\zeta^{j(k-l)}\Re \frac{1}{m^2x}\left(\frac{(G_0(1+ix))^{k-(m-1)}}{(G_0(1-ix))^{k-(m-1)}}-1\right).\]
Therefore,
\begin{align*}
\varphi_{f_{k,l}\circ\zeta^j}(s)-\zeta^{j(k-l)}\varphi_{f_{m-1,m-1}}(s)
&=\int_s^{s_0}\frac{\psi\left(\frac{\sqrt{u-s}}{H^{-1}_{1}(u)}\right)\cdot \frac{\sqrt{u-s}}{H^{-1}_{1}(u)}\cdot g(-\log u)}{\sqrt{u}\sqrt{u-s}}du\\
&=\int_s^{s_0}\frac{\psi(\frac{\sqrt{u-s}}{H^{-1}_{1}(u)})g(-\log u)}{\sqrt{u}H^{-1}_{1}(u)}du
\end{align*}
with
\[\varphi_{f_{m-1,m-1}}(s)=\int_s^{s_0}\frac{g(-\log u)}{m^2\sqrt{u}\sqrt{u-s}}du.\]
It follows that
\begin{align*}
\varphi'_{f_{k,l}\circ\zeta^j}(s)&-\zeta^{j(k-l)}\varphi'_{f_{m-1,m-1}}(s)=-\frac{\psi(0)g(-\log s)}{\sqrt{s}H^{-1}_{1}(s)}+\int_s^{s_0}\frac{\psi'\left(\frac{\sqrt{u-s}}{H^{-1}_{1}(u)}\right)g(-\log u)}{2\sqrt{u}\sqrt{u-s}(H^{-1}_{1}(u))^2}du.
\end{align*}
By definition,
\[H^{-1}_{1}(x)=\int_0^x(H^{-1}_{1})'(u)du=\int_0^x\frac{g(-\log u)}{\sqrt{u}}du\geq \int_0^x\frac{g(-\log x)}{\sqrt{u}}du=2\sqrt{x}g(-\log x),\]
so
\begin{equation}\label{eq: Hg_bound}
	\frac{\sqrt{u}}{H^{-1}_{1}(u)}\leq \frac{1}{2g(-\log u)}\leq \frac{1}{2g(-\log s_0)}\text{ for }u\in(0,s_0].
	\end{equation}
Moreover, by concavity of $g$, we have
\begin{align*}
H^{-1}_{1}(x)-2\sqrt{x}g(-\log x)&= \int_0^x\frac{g(-\log u)-g(-\log x)}{\sqrt{u}}du=\int_0^x\int_u^x\frac{g'(-\log v)}{v\sqrt{u}}dvdu\\
&\leq\int_0^x\frac{g'(0)}{v}\int_0^v\frac{du}{\sqrt{u}}dv=4g'(0)\sqrt{x}.
\end{align*}
It follows that there exists $C_g>0$ such that
\begin{equation}\label{eq:szaH}
2\sqrt{x}g(-\log x)\leq H^{-1}_{1}(x)\leq C_g\sqrt{x}g(-\log x)\text{ for all }x\in [0,s_0].
\end{equation}

As $\psi$ is analytic, there is $C>0$ such that %
 $|\psi'(x)|\leq C$ for $|x|\leq \frac{1}{g(-\log s_0)}$.
It follows that
\begin{align*}
|\varphi'_{f_{k,l}\circ\zeta^j}(s)&-\zeta^{j(k-l)}\varphi'_{f_{m-1,m-1}}(s)|\leq
|\psi(0)|\frac{g(-\log s)}{\sqrt{s}H^{-1}_{1}(s)}+C\int_s^{s_0}\frac{g(-\log u)}{2\sqrt{u}\sqrt{u-s}(H^{-1}_{1}(u))^2}du.
\end{align*}
Moreover, {by \eqref{eq: Hg_bound}}, it holds that
\begin{align*}
\int_s^{s_0}\frac{g(-\log u)}{2\sqrt{u}\sqrt{u-s}(H^{-1}_{1}(u))^2}du\leq \int_s^{s_0}\frac{1}{4u\sqrt{u-s}H^{-1}_{1}(u)}du\\
\leq \frac{1}{\sqrt sH^{-1}_{1}(s)}\int_1^{+\infty}\frac{1}{4u\sqrt{u-1}}du.
\end{align*}
Therefore,
\begin{equation}\label{eq:singdiff}
\varphi'_{f_{k,l}\circ\zeta^j}(s)-\zeta^{j(k-l)}\varphi'_{f_{m-1,m-1}}(s)=O\Big(\frac{g(-\log s)}{\sqrt{s}H^{-1}_{1}(s)}\big)=O(1/s).
\end{equation}
As $\varphi_{f_{m-1,m-1}}(s)=\int_s^{s_0}\frac{g(-\log u)}{m^2\sqrt{u}\sqrt{u-s}}du$, in view of \eqref{eq:ups1} and \eqref{eq:ups2}, we have
\begin{equation*}
\varphi'_{f_{m-1,m-1}}(s)=O\Big(\frac{ g(-\log s)}{s}\Big)\quad\text{and}\quad \lim_{s\to 0}\frac{s}{g(-\log s)}\varphi'_{f_{m-1,m-1}}(s)=-\frac{1}{m^2}.
\end{equation*}
Together with \eqref{eq:singdiff}, this gives \eqref{eq:fkl1}.
\end{proof}

\begin{lemma}\label{lem:deg2}
Suppose that $H_1(y)=y^2$ and $H_2$ is an even $C^2$-map such that $(H^{-1}_{2})'(u)=\frac{g(-\log u)}{\sqrt{u}}>0$. Let $k,l\in\Z_{\geq 0}$ such that $k+l=2(m-1)$. Then $\varphi_{f_{k,l}\circ\zeta^j}:(0,s_0]\to \C$ is a $C^2$-map such that
\begin{equation}\label{eq:fkl2}
\varphi'_{f_{k,l}\circ\zeta^j}(s)=O(\tau(s)^{-1})\quad\text{and}\quad \lim_{s\to 0}\tau(s)\varphi'_{f_{k,l}\circ\zeta^j}(s)=-\frac{\zeta^{2j(k-(m-1))}\Re\zeta^{k-(m-1)}}{m^2}.
\end{equation}
\end{lemma}

\begin{proof}
As $k+l=2(m-1)$, by \eqref{eq:kl}, we have
\begin{align*}
\varphi_{f_{k,l}\circ\zeta^j}(s)-\zeta^{j(k-l)}\Re\zeta^{k-(m-1)}\varphi_{f_{m-1,m-1}}(s)&=\int_s^{s_0}\frac{2\widetilde\psi(\frac{H_{1}^{-1}(u)}{H^{-1}_{2}(u-s)})\frac{H_{1}^{-1}(u)}{H^{-1}_{2}(u-s)}}{H'_1(H^{-1}_{1}(u))H'_2(H_{2}^{-1}(u-s))}du\\
&=
\int_s^{s_0}\frac{\widetilde\psi(\frac{\sqrt{u}}{H^{-1}_{2}(u-s)})\frac{\sqrt{u}}{H^{-1}_{2}(u-s)}g(-\log (u-s))}{\sqrt{u}\sqrt{u-s}}du\\
&=\int_0^{s_0-s}\frac{\widetilde\psi(\frac{\sqrt{u+s}}{H^{-1}_{2}(u)})g(-\log (u))}{H^{-1}_{2}(u)\sqrt{u}}du
\end{align*}
with
\[\varphi_{f_{m-1,m-1}}(s)=\int_s^{s_0}\frac{g(-\log (u-s))}{m^2\sqrt{u}\sqrt{u-s}}du,\]
where $\tilde \psi:\R\to\R$ is an analytic map given by
\[\widetilde \psi(x)=e^{j(k-l)}\Re \frac{1}{m^2x}\left(\frac{(G_0(x+i))^{k-(m-1)}}{(G_0(x-i))^{k-(m-1)}}-\zeta^{k-(m-1)}\right).\]
Thus, by Leibniz integration formula, for $s\in(0,s_0/2]$ it holds that
\begin{align}\label{eq: comp0}
\begin{aligned}
\varphi'_{f_{k,l}\circ\zeta^j}(s)-\zeta^{j(k-l)}\Re\zeta^{k-(m-1)}\varphi'_{f_{m-1,m-1}}(s)&=-\frac{\widetilde\psi(\frac{\sqrt{s_0}}{H^{-1}_{2}(s_0-s)})g(-\log (s_0-s))}{H^{-1}_{2}(s_0-s)\sqrt{s_0-s}}\\
&\quad+\int_0^{s_0-s}\frac{\widetilde\psi'(\frac{\sqrt{u+s}}{H^{-1}_{2}(u)})g(-\log u)}{2\sqrt{u+s}\sqrt{u}(H^{-1}_{2}(u))^2}du.
\end{aligned}
\end{align}
Recall that $H^{-1}_{2}(x)\geq 2\sqrt{x}g(-\log x)$.
As $\psi$ is analytic, for every $a>0$ there exists $C_a>0$ such that $|\widetilde\psi(x)|,|\widetilde\psi'(x)|\leq C_a$ for $|x|\leq a$. Moreover,
\begin{align*}
\widetilde \psi'(x)&=-\zeta^{j(k-l)}\Re \frac{1}{m^2x^2}\left(\frac{(G_0(x+i))^{k-(m-1)}}{(G_0(x-i))^{k-(m-1)}}-\zeta^{k-(m-1)}\right)\\
&\quad-\zeta^{j(k-l)}\Re \frac{(k-(m-1))}{m^3x}\frac{(G_0(x+i))^{k-(m-1)-m}}{(G_0(x-i))^{k-(m-1)-m}}\frac{2i}{(x-i)^2}.
\end{align*}
Therefore, there exists $C>0$ such that $|x^2\widetilde \psi'(x)|\leq C$ for all $x\in\R$.

As $0<\frac{\sqrt{s_0}}{H^{-1}_{2}(s_0-s)}\leq a_0:=\frac{\sqrt{s_0}}{H^{-1}_{2}(s_0/2)}$ for $s\in(0,s_0/2]$, we have
\begin{equation}\label{eq: comp1}
\left|\frac{\widetilde\psi(\frac{\sqrt{s_0}}{H^{-1}_{2}(s_0-s)})g(-\log (s_0-s))}{H^{-1}_{2}(s_0-s)\sqrt{s_0-s}}\right|\leq
\frac{C_{a_0}g(-\log (s_0/2))}{H^{-1}_{2}(s_0/2)\sqrt{s_0/2}}\quad\text{for all}\quad s\in(0,s_0/2].
\end{equation}
Moreover, for any $s\in(0,s_0/2]$ and for any $0<a(s)\leq s_0-s$,
\begin{align}\label{eq: comp2}
\begin{aligned}
\left|\int_0^{a(s)}\frac{\widetilde\psi'(\frac{\sqrt{u+s}}{H^{-1}_{2}(u)})g(-\log u)}{\sqrt{u+s}\sqrt{u}(H^{-1}_{2}(u))^2}du\right|
&\leq C\int_0^{a(s)}\frac{g(-\log u)}{(u+s)^{3/2}\sqrt{u}}du\\
&\leq \frac{C}{s^{3/2}}\int_0^{a(s)}(H_2^{-1})'(u)du\leq C\frac{H_2^{-1}(a(s))}{s^{3/2}}.
\end{aligned}
\end{align}
Suppose that there exists $a>0$ such that for any $s\in(0,s_0/2]$ and $a(s)\leq u\leq s_0-s$, we have $\frac{\sqrt{u+s}}{H^{-1}_{2}(u)}\leq a$. Then, in view of \eqref{eq:szaH} {(with replacing $H_1$ by $H_2$)}, we have
\begin{align}\label{eq: comp3}
\left|\int_{a(s)}^{s_0-s}\frac{\widetilde\psi'(\frac{\sqrt{u+s}}{H^{-1}_{2}(u)})g(-\log u)}{\sqrt{u+s}\sqrt{u}(H^{-1}_{2}(u))^2}du\right|
\leq \frac{C_a}{H^{-1}_{2}(a(s))}\int_{a(s)}^{s_0-s}\frac{du}{2u\sqrt{u+s}}\leq \frac{C_a}{\sqrt{a(s)}H^{-1}_{2}(a(s))}.
\end{align}

\medskip
{The term on the right-hand side in \eqref{eq: comp1} is bounded. Hence, we will now focus on proving that the terms on the right-hand side of \eqref{eq: comp3} and \eqref{eq: comp4} have lower order than $\frac{g(-\log s)}{s}$.}
Note first that $\sqrt{g}:\R_{\geq 0}\to\R_{> 0}$ is a $C^1$-map satisfying \eqref{prop:g}. Let $\widecheck H:[-\check x_0,\check x_0]\to[0,\check s_0]$ be the related potential, i.e.\ $(\widecheck{H}^{-1})'(u)=\frac{\sqrt{g(-\log u)}}{\sqrt{u}}$ with $\check x_0\in(0,1]$ and $\check s_0=\widecheck{H}(\check x_0)\leq 1$.
Let $a:(0,\check x^2_0/4]\to[0,\check s_0]\subset [0,1]$ be given by $a(s)=\widecheck{H}(2\sqrt{s})$. In view of \eqref{eq:szaH},
\begin{gather*}
2\sqrt{a(s)}g(-\log a(s))\leq H^{-1}_{2}(a(s))\leq C_g\sqrt{a(s)}g(-\log a(s)),\\
2\sqrt{a(s)}\sqrt{g(-\log a(s))}\leq \widecheck H^{-1}(a(s))=2\sqrt{s}\leq C_{\sqrt{g}}\sqrt{a(s)}\sqrt{g(-\log a(s))}.
\end{gather*}
It follows that, for every $0<s\leq \check x^2_0$, we have
\begin{gather}
\label{eq: comp4}
\frac{H_2^{-1}(a(s))}{s\sqrt{s}}\leq \frac{C_g\sqrt{a(s)}g(-\log a(s))}{s\sqrt{a(s)}\sqrt{g(-\log a(s))}}=\frac{C_g\sqrt{g(-\log a(s))}}{s},\\
\label{eq: comp5}
\frac{1}{\sqrt{a(s)}H^{-1}_{2}(a(s))}\leq \frac{1}{2a(s)g(-\log a(s))}\leq \frac{C^2_{\sqrt{g}}}{s},\\
\nonumber
a(s)= a(s)g(0)\leq a(s)g(-\log a(s))\leq s.
\end{gather}
Moreover, if $a(s)\leq u\leq s$, then
\[\frac{\sqrt{u+s}}{H^{-1}_{2}(u)}\leq \frac{2\sqrt{s}}{H^{-1}_{2}(a(s))}=\frac{\widecheck H^{-1}(a(s))}{H^{-1}_{2}(a(s))}\leq\frac{C_{\sqrt{g}}}{2\sqrt{g(-\log a(s))}}\leq \frac{C_{\sqrt{g}}}{2\sqrt{g(0)}}\leq C_{\sqrt{g}},\]
On the other hand, if $ns\leq u\leq (n+1)s$ for some $n\in \N$, then
\[\frac{\sqrt{u+s}}{H^{-1}_{2}(u)}\leq \frac{\sqrt{n+2}\sqrt{s}}{H^{-1}_{2}(ns)}\leq\frac{\sqrt{n+2}\sqrt{s}}{2\sqrt{ns}g(-\log (ns))}\leq \frac{\sqrt{n+2}}{2\sqrt{n}g(0)}\leq 1.\]
It follows that
\[\frac{\sqrt{u+s}}{H^{-1}_{2}(u)}\leq a:=C_{\sqrt{g}}\quad\text{for all}\quad a(s)\leq u\leq s_0-s,\]
in particular, \eqref{eq: comp3} holds.

As $a(s)\leq s$, for any $s\in(0,s_0/2]$, there is $-\log s\leq y(s)\leq -\log a(s)$ such that
\begin{align*}
\log\frac{g(-\log a(s))}{g(-\log s)}=\frac{g'(y(s))}{g(y(s))}\log\frac{s}{a(s)}\leq \frac{g'(-\log s)}{g(-\log s)}\log\frac{s}{a(s)}.
\end{align*}
Since $\frac{g'(-\log s)}{g(-\log s)}\to 0$ as $s\to 0$, for any $0<\vep<1/2$, there is $s_\vep>0$ such that $\frac{g'(-\log s)}{g(-\log s)}<\vep$ for $s\in(0,s_\vep)$.
Hence, by \eqref{eq: comp5} and mean value theorem, for $s\in(0,s_\vep)$ we have
\[\frac{\sqrt{g(-\log a(s))}}{g(-\log s)}=\frac{\frac{g(-\log a(s))}{g(-\log s)}}{\sqrt{g(-\log a(s))}}\leq \frac{\big(\frac{s}{a(s)}\big)^{\vep}}{\sqrt{g(-\log a(s))}}\leq \frac{C_{\sqrt{g}}^{2\vep}}{g(-\log a(s))^{1/2-\vep}}.\]
Hence,
\[\lim_{s\to 0}\frac{\sqrt{g(-\log a(s))}}{g(-\log s)}= 0.\]
In view of \eqref{eq: comp0}, \eqref{eq: comp1}, \eqref{eq: comp2}, \eqref{eq: comp3}, \eqref{eq: comp4}, and \eqref{eq: comp5}, it follows that
\begin{equation}\label{eq:singdiff2}
\varphi'_{f_{k,l}\circ\zeta^j}(s)-\zeta^{j(k-l)}\Re\zeta^{k-(m-1)}\varphi'_{f_{m-1,m-1}}(s)=o\Big(\frac{g(-\log s)}{s}\Big).
\end{equation}
As $\varphi_{f_{m-1,m-1}}(s)=\int_s^{s_0}\frac{g(-\log (u-s))}{m^2\sqrt{u}\sqrt{u-s}}du$, in view of \eqref{eq:ups3} and \eqref{eq:ups4}, we have
\begin{equation*}
\varphi'_{f_{m-1,m-1}}(s)=O\Big(\frac{ g(-\log s)}{s}\Big)\quad\text{and}\quad \lim_{s\to 0}\frac{s}{g(-\log s)}\varphi'_{f_{m-1,m-1}}(s)=-\frac{1}{m^2}.
\end{equation*}
Together with \eqref{eq:singdiff2}, this gives \eqref{eq:fkl2}.
\end{proof}

\subsection{Final discussion}\label{sec:impfinal}
Suppose that $g:\R_{\geq 0}\to\R_{> 0}$ is a $C^\infty$-map satisfying \eqref{prop:g} and such that $g(0)=1$.
Let $H_2(y)=y^2$ and let $H_1$ be an even $C^2$-{map} such that $(H^{-1}_{1})'(u)=\frac{g(-\log u)}{\sqrt{u}}$.
Let $\widetilde{H}:[-1,1]^2\to\R$ be given by $\widetilde{H}(x,y)=-H_1(x)+H_2(y)$. Let $m\geq 1$. Then $H:\mathcal{D}\to\R$ given by
$H(z,\bar{z})=\widetilde H(z^m,\bar{z}^m)$ is of class $C^{2m}$.

\begin{theorem}\label{thm:genCf}
Let $f:\C\to\C$ be a homogenous polynomial of degree $2m-2$
\[f(z,\bar{z})=\sum_{0\leq k\leq 2m-2}a_{k,2m-2-k}z^k\bar{z}^{2m-2-k}.\]
Then for any $1\leq j\leq 2m$, we have
\begin{equation}\label{eq:fpoly}
(\varphi^j_{f})'(s)=O(\tau(s)^{-1})\quad\text{and}\quad \lim_{s\to 0}\tau(s)(\varphi^j_{f})'(s)=-C^j(f),
\end{equation}
where
\begin{align*}
C^{2j+1}(f)&=\frac{1}{m^2}\sum_{0\leq k\leq 2m-2}a_{k,2m-2-k}\zeta^{-2j(k-(m-1))},\\
C^{2j+2}(f)&=\frac{1}{m^2}\sum_{0\leq k\leq 2m-2}a_{k,2m-2-k}\zeta^{-2j(k-(m-1))}\Re\zeta^{k-(m-1)}.
\end{align*}
for $0\leq j<2m$.
\end{theorem}

\begin{proof}
In view of \eqref{eq:phifmulti3},
\[\varphi^{2j+1}_{f_{k,l}}=\varphi^{1}_{f_{k,l}\circ \zeta^{-j}},\quad
\varphi^{2j+2}_{f_{k,l}}=\varphi^{2}_{f_{k,l}\circ \zeta^{-j}}.\]
It follows that
\begin{align*}
\varphi^{2j+1}_{f}=\sum_{k,l}a_{k,l}\varphi^{1}_{f_{k,l}\circ \zeta^{-j}},\quad \varphi^{2j+2}_{f}=\sum_{k,l}a_{k,l}\varphi^{2}_{f_{k,l}\circ \zeta^{-j}}.
\end{align*}
By \eqref{eq:phifmulti2}, Lemmas~\ref{lem:deg1}~and~\ref{lem:deg2}, this gives \eqref{eq:fpoly}.
\end{proof}

\begin{figure}[h!]
 \includegraphics[width=0.5\textwidth]{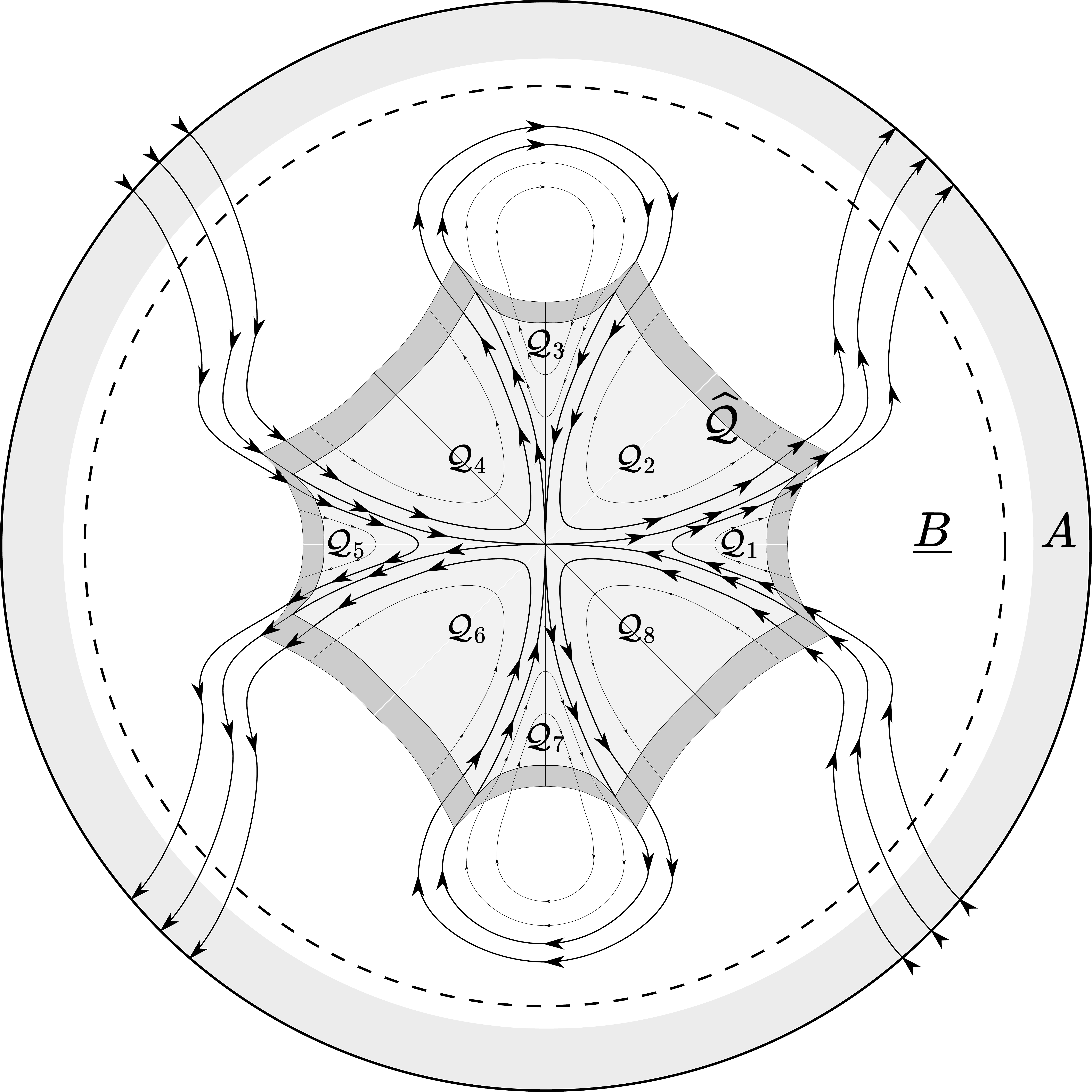}
 \caption{} \label{fig:QQ}
\end{figure}

\begin{example}
Let us consider a specific example where $m=2$ and $f:\mathcal{D}\to\R$ given by $f=f_{2,0}+f_{0,2}$. Then $\zeta=i$. As $\zeta^{2}=-1$, we have $f_{2,0}\circ\zeta=-f_{2,0}$, $f_{0,2}\circ\zeta=-f_{0,2}$ and $f\circ\zeta=-f$. The local Hamiltonian flow has $8$ invariant parabolic sectors for which
\begin{gather*}
\varphi^{1}_{f_{0,2}}=\varphi^{1}_{f_{2,0}},\quad \varphi^{2}_{f_{0,2}}=\varphi^{2}_{f_{2,0}},\\
\varphi^{2j+1}_{f_{0,2}}=\varphi^{1}_{f_{0,2}\circ \zeta^{-j}}=
(-1)^j\varphi^{1}_{f_{0,2}}=(-1)^j\varphi^{1}_{f_{2,0}}=\varphi^{1}_{f_{2,0}\circ \zeta^{-j}}=
\varphi^{2j+1}_{f_{2,0}}\\
\varphi^{2j+2}_{f_{0,2}}=\varphi^{2}_{f_{0,2}\circ \zeta^{-j}}
=(-1)^j\varphi^{2}_{f_{0,2}}=(-1)^{j}\varphi^{2}_{f_{2,0}}=\varphi^{2}_{f_{2,0}\circ \zeta^{-j}}=
\varphi^{2j+2}_{f_{2,0}}.
\end{gather*}
Therefore
\begin{gather*}
\varphi^{1}_{f}=2\varphi^{1}_{f_{2,0}},\quad \varphi^{2}_{f}=2\varphi^{2}_{f_{2,0}},\quad \varphi^{3}_{f}=-2\varphi^{1}_{f_{2,0}},\quad \varphi^{4}_{f}=-2\varphi^{2}_{f_{2,0}},\\
 \varphi^{5}_{f}=2\varphi^{1}_{f_{2,0}},\quad \varphi^{6}_{f}=2\varphi^{2}_{f_{2,0}},\quad \varphi^{7}_{f}=-2\varphi^{1}_{f_{2,0}},\quad \varphi^{8}_{f}=-2\varphi^{2}_{f_{2,0}}.
 \end{gather*}
In view of Theorem~\ref{thm:genCf},
\begin{gather}
\label{eq:tau1}
(\varphi^j_{f})'(s)=O(\tau(s)^{-1})\text{ for all }1\leq i\leq 8,\\
\label{eq:tau2}
C^{1}(f)=-C^3(f)=C^5(f)=-C^7(f)=\frac{1}{2}\text{ and }\\
\nonumber
C^{2}(f)=-C^4(f)=C^6(f)=-C^8(f)=0.
\end{gather}
\end{example}
\begin{figure}[h!]
 \includegraphics[width=1\textwidth]{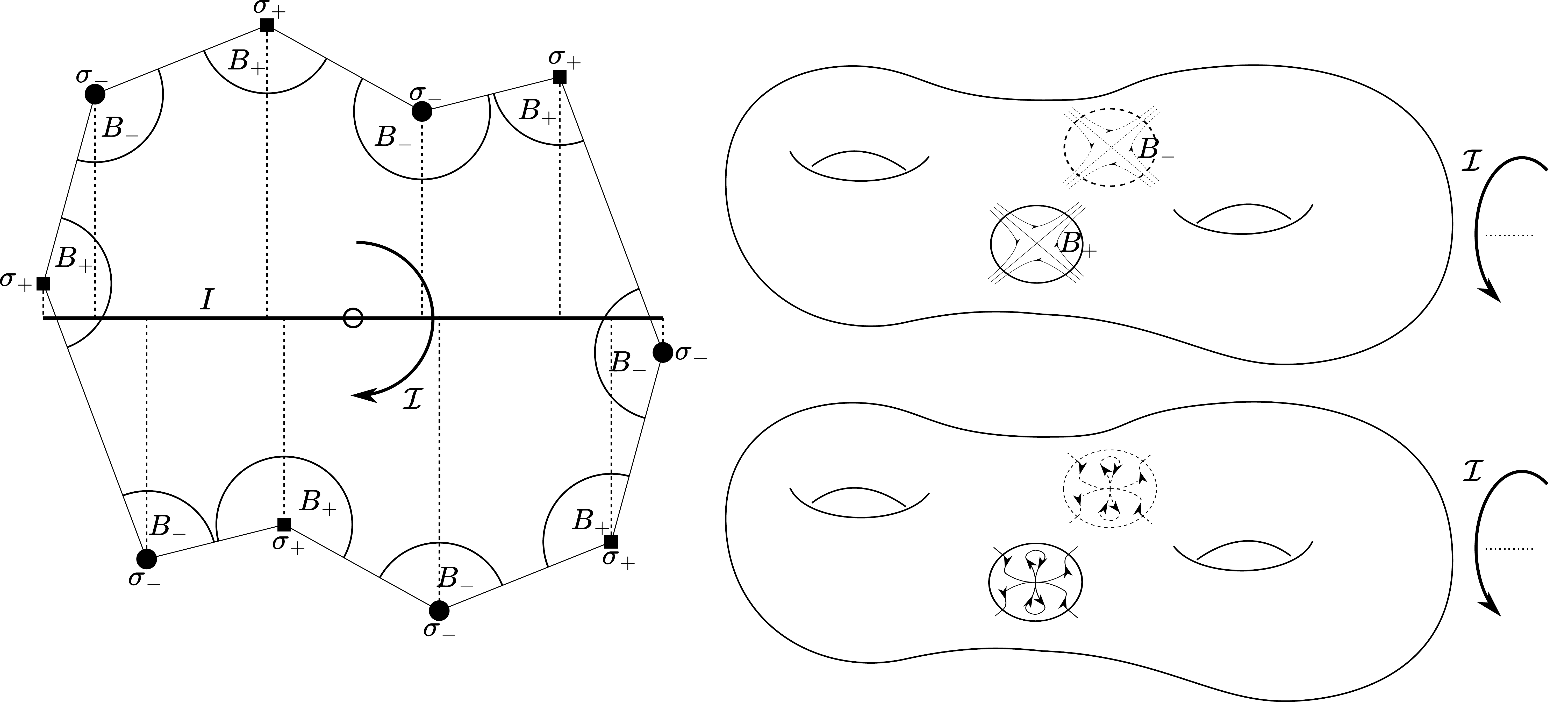}
 \caption{} \label{fig:surf}
\end{figure}

\textbf{Step 1. Extension of the imperfect multi-saddle to a disk.} Choose small numbers $0<s_0<\hat{s}_0$ and set $\mathcal{Q}:=\mathcal{Q}(s_0)$ and $\widehat{\mathcal{Q}}:=\mathcal{Q}(\hat{s}_0)$. In Figure~\ref{fig:QQ}, $\mathcal{Q}$ is the light gray central area, while the difference $\widehat{\mathcal{Q}}\setminus \mathcal{Q}$ is dark gray. Let us consider the Hamiltonian $H:\widehat{\mathcal{Q}}\to\R$, $H(z,\bar{z})=\widetilde{H}(z^2,\bar{z}^2)$ and the corresponding local flow $(\Psi_t)$ on $\widehat{\mathcal{Q}}$. The phase portrait of this flow is presented in Figure~\ref{fig:QQ}. We extend the action of $(\Psi_t)$ to a disk $B$ centered at zero. First, we extend the Hamiltonian $H$ to a smaller central disk $\underline{B}\subsetneq B$ to close the separatrices touching sectors $\widehat{\mathcal{Q}}_3$ and $\widehat{\mathcal{Q}}_7$, creating two saddle loops, one above sector $\widehat{\mathcal{Q}}_3$ and the other below sector $\widehat{\mathcal{Q}}_7$, see Figure~\ref{fig:QQ}. This extension gives rise to two local minima for $H:\underline{B}\to\R$, which are central fixed points for the extended $(\Psi_t)$. As $H$ on $\widehat{\mathcal{Q}}$ is horizontally and vertically symmetric, we can extend it so that the extended Hamiltonian $H:\underline{B}\to\R$ is also horizontally and vertically symmetric.

Next, suppose that the local flow $(\Psi_t)$ on an open neighborhood of the boundary of the disk $B$ (i.e.\ an annulus $A \varsubsetneq B\setminus \underline{B}$ surrounding $\partial B$) is related to the equation $z'=\frac{i}{2z}=\frac{i\bar{z}}{2|z|^2}$. Then $(\Psi_t)$ on $A$ is a local Hamiltonian flow with the Hamiltonian $H(z,\bar{z})=-\frac{1}{4}(z^2+\bar{z}^2)$ and $V(z,\bar{z})=|z|^2$.
On the other hand $(\Psi_t)$ on $A$ coincides with the vertical flows on the translation disk $(B,2z\, dz)$. As $H$ and $V$ on $A$ are horizontally and vertically symmetric, we can find horizontally and vertically symmetric $C^\infty$-maps $H:B\to\R$ and $V:B\to\R_{>0}$ such that $H$ is an extension of the Hamiltonian constructed so far on $\underline{B}$ and $A$, and $V=1$ on $\underline{B}$.
Finally $(\Psi_t)$ on $B$ is the local Hamiltonian flow associated with the Hamiltonian $H:B\to\R$ with the density $V:B\to\R_{>0}$.

\textbf{Step 2. From hyper-elliptic translation surfaces to anti-symmetric Hamiltonian flows.}
Let $(M,\omega)$ be a translation surface in $\mathcal{H}(1,1)$ such that its vertical flow has no saddle connections. As $\mathcal{H}(1,1)$ is a hyper-elliptic stratum, $(M,\omega)$ has a polygonal representation for which the sides of the polygon consist of $5$ pairs of the same intervals symmetrically distributed, as in Figure~\ref{fig:surf}. Then, the rotation through the angle $\pi$ with respect to the center of the polygon is the hyper-elliptic involution $\mathcal{I}:M\to M$ of the translation surface $(M,\omega)$. The surface has two singularities $\sigma_+$, $\sigma_-$ such that $\mathcal{I}(\sigma_+)=\sigma_-$. Let us consider disks $B_+$, $B_-$ centered at $\sigma_+$, $\sigma_-$ respectively such that $\mathcal{I}(B_+)=B_-$. Finally we modify (on $B_\pm$) the vertical flow on $M$ so that the modified flow denoted by $\psi_\R$ coincides on $B_+$ with
the local Hamiltonian flow $(\Psi_t)$ constructed in Step~1, and $\psi_t=\mathcal{I}\circ \Psi_{-t}\circ \mathcal{I}$ on $B_-$. Then $\psi_\R$ is a locally Hamiltonian flow with four saddle loops such that $\psi_t\circ \mathcal{I}=\mathcal{I}\circ \psi_{-t}$ and $\psi_\R$ is minimal on $M'$ the complement of the four periodic components inside saddle loops. Let $I \subseteq M'$ be the interval centered at the center of symmetry of the polygon whose length is equal to the width of the polygon. Then, $T:I\to I$ is a symmetric IET.

\textbf{Step 3. Construction of the map $f:M\to\R$.} Let $f:M\to\R$ be a $C^{\infty}$-map such that $f\circ \mathcal{I}=-f$ as follows. First, $f$
is zero on $M\setminus B_{\pm}$. Next, let $f=f_{2,0}+f_{0,2}$ on $\mathcal{Q}$, that is on the light gray area on Figure~\ref{fig:QQ}.
The map $f$ on $B$ is also equal to zero outside $\widehat{\mathcal{Q}}$. As $f$ on $\mathcal{Q}$ is vertically and horizontally symmetric with $f\circ\zeta=-f$ and $f=0$ outside $\widehat{\mathcal{Q}}$, we can find a $C^{\infty}$ extension $f:B\to\R$ which is vertically and horizontally symmetric with $f\circ\zeta=-f$. We finish to define $f$ on $B_+$ taking this map and $f(x)=-f(\mathcal{I}x)$ on $B_-$.

\begin{theorem}
For almost every IET $T:I\to I$ the skew product flow $\psi^f_\R$ on $M'\times\R$ is ergodic.
\end{theorem}

\begin{proof}
As usual, we need to show that the skew product map $T_{\varphi_f}:I\times\R\to I\times\R$ is ergodic. First note that the map $\varphi_f:I\to\R$
is anti-symmetric. Indeed, as $\mathcal{I}$ maps $I$ to $I$ and it is a central reflection, we have
\begin{align*}
\varphi_f(\mathcal{I}(Tx))&=\int_0^{\tau(\mathcal{I}(Tx))}f(\psi_t(\mathcal{I}(Tx)))\,dt=
\int_0^{\tau(\mathcal{I}(Tx))}f(\mathcal{I}\circ \psi_{-t}(Tx))\,dt\\
&=
-\int_0^{\tau(\mathcal{I}(Tx))}f(\psi_{\tau(x)-t}(x))\,dt=-\int_{\tau(x)-\tau(\mathcal{I}(Tx))}^{\tau(x)}f(\psi_{t}(x))\,dt.
\end{align*}
Note that $\tau(T^{-1}\circ\mathcal{I}x)$ is the first return times to $I$ for the inverse flow $(\psi_{-t})_{t\in\R}$ starting from $\mathcal{I}x\in I$. As $\psi_{-t}\circ\mathcal{I}=\mathcal{I}\circ\psi_t$, $\tau(T^{-1}\circ\mathcal{I}x)$ is the first return times to $I$ for the inverse flow $(\psi_{t})_{t\in\R}$ starting from $x\in I$, so $\tau(\mathcal{I}(Tx))=\tau(T^{-1}\circ\mathcal{I}x)=\tau(x)$. This gives $\varphi_f(\mathcal{I}(Tx))=-\varphi_f(x)$.

As $f$ is zero outside $B_{\pm}$ and in $B_{\pm}\setminus\widehat{\mathcal{Q}}$, the function $\varphi_f:I\to\R$ is zero outside the $\hat{s}_0$-neighbourhood of discontinuities of $T$. By choosing $\hat{s}_0$ sufficiently small, we can guarantee that the mentioned neighborhoods are disjoint. Moreover, the last exchanged interval does not touch the sets $\widehat{\mathcal{Q}}$ when we move it through the flow until its return to $I$. It follows that $\varphi_f$ is zero on the last interval. The behaviour of $\varphi_f$ in other neighborhoods are fully determined
by the functions $\widehat{\varphi}^j_f:(0,\hat{s}_0]\to\R$, $1\leq j\leq 8$, where $\widehat{\varphi}_f^j$ counts ergodic integrals of $f$ from the entrance to the exit of the set $\widehat{\mathcal{Q}}_j$. Recall that ${\varphi}_f^j:(0,s_0]\to\R$ counts ergodic integrals of $f$ from the entrance to the exit of the set ${\mathcal{Q}}_j$. Note that
\begin{equation}\label{eq:2468}
\widehat{\varphi}_f^2=\widehat{\varphi}_f^4=\widehat{\varphi}_f^6=\widehat{\varphi}_f^8=0.
\end{equation}
Indeed, as $f\circ\zeta=-f$ and the flow $\psi_\R$ on $\widehat{\mathcal{Q}}$ is auto-conjugated by the rotation by $\zeta$, we have $\widehat{\varphi}_f^2=\widehat{\varphi}_{f\circ\zeta}^4=-\widehat{\varphi}_{f}^4$ and $\widehat{\varphi}_f^6=\widehat{\varphi}_{f\circ\zeta}^8=-\widehat{\varphi}_{f}^8$. Denote by $\upsilon$ the vertical reflection $\upsilon(x,y)=(-x,y)$. As $H\circ \upsilon=H$ on $\widehat{\mathcal{Q}}$, the flow $\psi_\R$ on $\widehat{\mathcal{Q}}$ is conjugated via $\upsilon$ to its inverse on $\widehat{\mathcal{Q}}$. Since $\widehat{\mathcal{Q}}_4=\upsilon(\widehat{\mathcal{Q}}_2)$ and $\widehat{\mathcal{Q}}_8=\upsilon(\widehat{\mathcal{Q}}_4)$, it follows that
$\widehat{\varphi}_f^2=\widehat{\varphi}_{f}^4$ and $\widehat{\varphi}_f^6=\widehat{\varphi}_{f}^8$. Therefore, $\widehat{\varphi}_f^j=0$ for every even $j$.
It follows that if a discontinuity $r_\alpha\in I$ lies on the separatrix going to $\sigma_+$, then $\varphi_f(x)=\widehat{\varphi}_f^{2j}(r_\alpha-x)+\widehat{\varphi}_{f}^{2j+2}(r_\alpha-x)=0$ for $x\in[r_\alpha-\hat{s}_0,r_\alpha)$ ($j=1$ or $3$), and similarly if $l_\alpha\in I$ lies on the separatrix going to $\sigma_-$, then $\varphi_f(x)=\widehat{\varphi}_f^{2j}(x-l_\alpha)+\widehat{\varphi}_{f}^{2j+2}(x-l_\alpha)=0$ for $x\in(l_\alpha,l_\alpha+\hat{s}_0]$.
Moreover, if $l_\alpha\in I$ lies on the separatrix going to $\sigma_+$, then $\varphi_f(x)=\widehat{\varphi}_f^{j}(x-l_\alpha)$ ($j=1$ or $5$) for $x\in(l_\alpha,l_\alpha+\hat{s}_0]$, and
if $r_\alpha\in I$ lies on the separatrix going to $\sigma_-$, then $\varphi_f(x)=\widehat{\varphi}_{-f}^{j}(r_\alpha-x)=-\widehat{\varphi}_{f}^{j}(r_\alpha-x)$ for $x\in[r_\alpha+\hat{s}_0,r_\alpha)$.

\begin{figure}[h!]
 \includegraphics[width=0.3\textwidth]{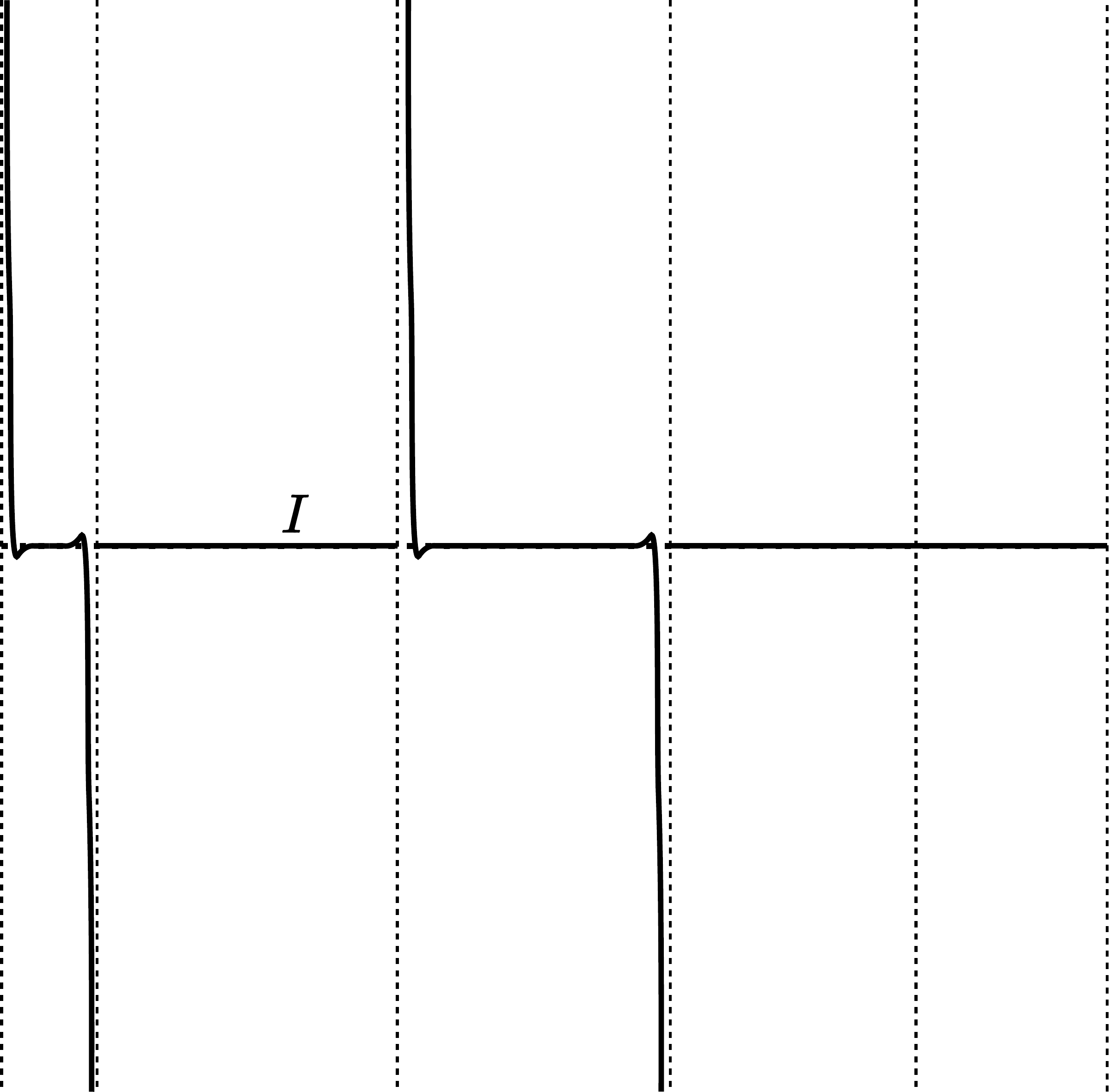}
 \caption{The graph of $\varphi_f$} \label{fig:vphi}
\end{figure}

Since $f$ is of class $C^\infty$, the map $\widehat{\varphi}_{f}^{j}$ on $(0,\hat{s}_0]$ is also of class $C^\infty$. Moreover, $\widehat{\varphi}_{f}^{j}-{\varphi}_{f}^{j}$ on $[0,s_0]$ is of class $C^\infty$, as it counts ergodic integrals of $f$ from the entrance to the exit of $\widehat{\mathcal{Q}}\setminus {\mathcal{Q}}$ (the dark gray set in Figure~\ref{fig:QQ}) which is far from the saddle. In view \eqref{eq:tau1} and \eqref{eq:tau2},
it follows that
\begin{equation}\label{eq:15}
(\widehat{\varphi}_{f}^{j})'(s)=O(\tau(s)^{-1})\text{ and }\lim_{s\to 0}\tau(s)(\widehat{\varphi}_{f}^{j})'(s)=-\frac{1}{2}\text{ for }j=1,5.
\end{equation}
Recall that $\tau(s)=\frac{s}{g(-\log s)}$ and $\theta(x)=\int_0^{\log x}g(u)\,du$.
Hence $\varphi_f:I\to\R$ is an anti-symmetric piecewise $C^\infty$-map with four one-side singularities (see Figure~\ref{fig:vphi}) so that $\varphi_f\in \Upsilon_\theta(\bigsqcup_{\alpha\in \mathcal{A}}
I_{\alpha})$ (i.e.\ $Z_\theta(\varphi_f)$ defined in \eqref{def:Zf} is finite) and $z_\theta(\varphi_f)$ defined in \eqref{def:zf} is positive.

In summary, we are almost ready to use Theorem~\ref{thm:anti} to prove the ergodicity $T_{\varphi_f}$ for a.e.\ IET $T$ because $\varphi_f$ satisfies its assumptions (i) and (ii). Finally, we still need to ensure that (iii) is also satisfied, which demands $\varphi_f$ being locally monotonic. In view of \eqref{eq:2468} and \eqref{eq:15}, $\varphi_f$ is non-increasing on any one-side $\vep$-neighbourhood ($0<\vep<s_0$ is sufficiently small) of any discontinuity of $T$. However, on intervals of the form $[l_\alpha+\vep,r_\alpha-\vep]$ this kind of monotonicity can be violated. To get rid of this problem, we decompose $\varphi_f=\varphi+\zeta$ so that both $\varphi,\zeta:I\to\R$ are anti-symmetric $C^\infty$-maps, $\varphi$ is piecewise non-increasing and $\zeta$ is non-zero only on intervals $[l_\alpha+\vep,r_\alpha-\vep]$. As $\zeta$ vanishes around discontinuities, we have $\partial_{\mathcal{O}}(\zeta)=0$ for any $\mathcal{O}\in \Sigma(\pi)$. As $\zeta$ is anti-symmetric, by Proposition~\ref{prop:van}, for a.e.\ $T$, we have $d_i(\zeta)=1$ for $i=1,2$ (here $g=2$), so $\mathfrak{h}(\zeta)=0$. In view of Proposition~8.12 in \cite{Fr-Ul2}, $\zeta$ is a coboundary, so the skew products $T_{\varphi_f}$ and $T_\varphi$ are isomorphic. As $\varphi$ satisfies all the assumptions of Theorem~\ref{thm:anti}, the skew product $T_\varphi$ is ergodic for a.e.\ $T$. It follows that $T_{\varphi_f}$ and $\psi^f_\R$ on $M'\times\R$ are also ergodic.
\end{proof}

\section*{Data Availabilty Statement}
Data sharing not applicable to this article as no datasets were generated or analysed during the current study.
\end{document}